\documentclass{article}
\usepackage{amsmath}
\usepackage{amsfonts}
\usepackage{amsthm}
\usepackage{enumerate}

\newtheorem{thm}{Theorem}[section]
\newtheorem{prop}[thm]{Proposition}
\newtheorem{cor}[thm]{Corollary}
\newtheorem{lem}[thm]{Lemma}

\theoremstyle{definition}
\newtheorem{defn}[thm]{Definition}

\theoremstyle{remark}
\newtheorem{rmk}[thm]{Remark}

\begin{document}

\title{Shimura's Vector-Valued Modular Forms, Weight Changing Operators, and Laplacians}

\author{Shaul Zemel}

\maketitle

\begin{abstract}
We investigate the various types of weight raising and weight lowering operators on quasi-modular forms, or equivalently on Shimura's vector-valued modular forms involving symmetric power representations. We also present all the eigenfunctions of the two possible Laplacian operators.
\end{abstract}

\section*{Introduction}

The systematic study of quasi-modular forms has started with the paper \cite{[KZ]}, which determined the structure of the ring of holomorphic scalar-valued quasi-modular forms on $SL_{2}(\mathbb{Z})$. This ring contains, in particular, the usual modular forms, as well as the classical example of a quasi-modular form---the holomorphic Eisenstein series of weight 2 on $SL_{2}(\mathbb{Z})$. It is the most natural ring containing the ring of modular forms on $SL_{2}(\mathbb{Z})$ that is also closed under differentiation (see, e.g., \cite{[MR1]}). The reference \cite{[A]} obtained some results for modular and quasi-modular forms on a larger class of Fuchsian groups. The prequel \cite{[Ze1]} to the current paper then showed how quasi-modular forms are related to the vector-valued modular forms defined in \cite{[Sh]} (and previously, in a different language, in \cite{[E]}), that involve symmetric powers of the standard representation, and established some properties of these vector-valued modular forms. These objects complement other generalizations of modular forms, such as (scalar-valued and vector-valued) modular forms of arbitrary weight (appearing in \cite{[Kn1]}, \cite{[N]}, and many others), mock modular forms (first uncovered by \cite{[Zw]}, then expanded by \cite{[BO]} and others, including the development in \cite{[BFu]} of the theory of the closely related harmonic weak Maass forms), or modular forms of higher order (see \cite{[CDO]} for the initial definition, and \cite{[DS]} for a classification result).

Classical (not necessarily holomorphic) modular forms are operated on by the non-holomorphic weight changing operators, named after Maa\ss\ and Shimura. Explicitly, the operator $\delta_{k}=\frac{\partial}{\partial\tau}+\frac{k}{2iy}$ sends modular forms of weight $k$ to modular forms of weight $k+2$, and $4y^{2}\frac{\partial}{\partial\overline{\tau}}$ decreases the weight by 2 (but annihilates holomorphic and meromorphic functions). Their Lie-theoretic origin is explained in \cite{[V]}. On the other hand, the usual differentiation preserves quasi-modularity and increases the weight again by 2, but also increases the depth by 1 (see \cite{[MR1]}). The first goal of this paper is to interpret these results in terms of the vector-valued modular forms of Shimura, and deduce some of their properties. In particular we show that while every weight raising operator $\delta_{l}$ increases the weight $k$ and the depth $d$ (meaning that the transformation law of the modular form in question involves additional $d$ functions) of a quasi-modular form by 2 and 1 respectively, the operator $\delta_{k-d}$ leaves the depth unchanged. As a corollary we simplify the proof of the existence and the uniqueness of the Rankin--Cohen brackets for quasi-modular forms appearing in \cite{[MR2]}. This result for classical modular forms (i.e., the case $d=0$) is known long ago, where the initial construction of the Rankin--Cohen brackets was carried out by looking for the combination of the differentiations that preserve modularity. In the setting of this paper they should alternatively be viewed as the combinations of weight raising operators that preserves holomorphicity. This approach may be more intuitive, since modularity is harder to preserve and holomorphicity is easier to check.

A very important operator in the theory of modular forms is the Laplacian of the appropriate weight. It's Lie-theoretic origin is the Casimir operator of an $\mathfrak{sl}_{2}$-triple. As already mentioned in \cite{[V]}, the weight changing operators form, together with the multiplication-by-weight operator, an $\mathfrak{sl}_{2}$-triple. Another $\mathfrak{sl}_{2}$-triple appears implicitly in \cite{[A]} for holomorphic quasi-modular forms. Every such $\mathfrak{sl}_{2}$-triple produces naturally an invariant operator, namely the Casimir or Laplacian operator (though the normalization of the latter usually differs from that of the former). We prove that the only $\mathfrak{sl}_{2}$-triples that can be defined using our weight changing operators are those from \cite{[V]} and \cite{[A]} (up to normalization), so that the two Laplacian operators arising from these references are the only possible such operators.

Once a Laplacian is given in some setting, the most natural question to ask is what are its eigenvalues and eigenfunctions. This question is well-known for the classical modular forms, in various settings (e.g., those of eigenvalue 0 are the harmonic weak Maass forms mentioned above). We determine in this paper all the eigenspaces of the two Laplacians from the previous paragraph. We remark that the analysis of the eigenspaces in depth $d$ becomes more difficult when the weight is an integer between $d+1$ and $2d$, a case which was also shown in \cite{[Ze1]} to be more delicate (e.g., this is the case where the dimension formulae from that reference depend on whether the Fuchsian group has cusps or not). Interestingly, this investigation leads to a special variant of the sesqui-harmonic modular forms defined, for example, in \cite{[BDR]}.

Our results about eigenfunctions can be described quite explicitly. The Laplacian operator from \cite{[V]} is the classical Laplacian operator of weight $k$ from the theory of modular forms, which is already known to have the usual modular eigenforms and the meromorphic quasi-modular forms as eigenfunctions. It follows that images of meromorphic quasi-modular forms under the weight raising operators are also eigenfunctions, and we prove that these examples already encompass \emph{all} the quasi-modular eigenfunctions of this Laplacian. The Laplacian from \cite{[A]} is simpler (it involves only differentiation of order 1), and its quasi-modular eigenfunctions are just holomorphic derivatives of (arbitrary) modular forms, plus the complex conjugates of some nearly meromorphic quasi-modular forms in case $k$ is one of the more delicate weights mentioned above.

\smallskip

This paper is divided into two sections. Section \ref{OpsQMFVVMF} describes our differential operators in the various settings, and proves the result about $\mathfrak{sl}_{2}$-triples. In Section \ref{LapEigen} we present the action of the two Laplacians, and determine their eigenspaces.

\smallskip

I would like to thank M. Neururer, during the conversation with whom I realized that this project could be carried out. I am also very grateful to the anonymous referees for many valuable suggestions and corrections.

\section{Operators on Quasi-Modular Forms and on Vector-Valued Modular Forms \label{OpsQMFVVMF}}

In this Section we describe the weight changing operators on the various spaces of modular and quasi-modular forms considered in \cite{[Ze1]}.

\subsection{Operators on Quasi-Modular Forms}

The Lie group $SL_{2}(\mathbb{R})$ has a well-known operation on the Poincar\'{e} upper half-plane $\mathcal{H}=\{\tau=x+iy\in\mathbb{C}|y>0\}$ via fractional linear transformations. Explicitly, the action of $\gamma=\big(\begin{smallmatrix} a & b \\ c & d\end{smallmatrix}\big)$ takes $\tau\in\mathcal{H}$ to $\gamma\tau=\frac{a\tau+b}{c\tau+d}$, with the factor of automorphy $j(\gamma,\tau)=c\tau+d$. We shall also use the notation $j_{\gamma}(\tau)$ for $j(\gamma,\tau)$, so that the lower left entry of $\gamma$ is just the derivative $j_{\gamma}'$ of $j_{\gamma}$ (independently of $\tau$). We shall work, as in \cite{[Ze1]}, with arbitrary Fuchsian groups $\Gamma \leq SL_{2}(\mathbb{R})$, namely discrete subgroups such that the quotient $\Gamma\backslash\mathcal{H}$ has finite volume with respect to the $SL_{2}(\mathbb{R})$-invariant measure $\frac{dxdy}{y^{2}}$.

Let $k\in\mathbb{Z}$ be a weight, and let $\rho$ be a representation of a Fuchsian subgroup $\Gamma$ of $SL_{2}(\mathbb{R})$ on some finite-dimensional complex vector space $V_{\rho}$. In fact, all of our definitions and results will hold in a general context: $k$ may be in $\frac{1}{2}\mathbb{Z}$ if $\Gamma$ is a subgroup of the metaplectic double cover of $SL_{2}(\mathbb{R})$, or alternatively $k$ can be an arbitrary real (or even complex) number if we take $\rho$ to be a (possibly vector-valued) multiplier system of weight $k$. For the definitions of the few notions required for these generalizations see, e.g., Subsection 1.1 of \cite{[Ze1]}. We do remark that an equivalent approach to multiplier systems in general weights can be obtained by considering subgroups $\Gamma$ of the universal covering group $\widetilde{SL}_{2}(\mathbb{R})$ (a realization of which can be obtained by replacing the metaplectic data $\sqrt{j(\gamma,\tau)}$ attached to some $\gamma \in SL_{2}(\mathbb{R})$ by a choice of $\log j(\gamma,\tau)$), and taking representations of such groups. The paper \cite{[Ze1]} has defined, extending previous definitions of \cite{[KZ]}, \cite{[MR1]}, and others, a \emph{quasi-modular form} of weight $k$, depth $d$, and representation (or multiplier system) $\rho$ with respect to $\Gamma$ to be a function $f:\mathcal{H} \to V_{\rho}$ for which there exist functions $f_{r}$, $0 \leq r \leq d$, with $f_{d}\neq0$, such that the functional equation
\begin{equation}
f(\gamma\tau)=\sum_{r=0}^{d}j_{\gamma}(\tau)^{k-r}(j_{\gamma}')^{r}\rho(\gamma)f_{r}(\tau) \label{QMFdef}
\end{equation}
holds for every $\gamma\in\Gamma$ and $\tau\in\mathcal{H}$. One typically requires some continuity or analyticity properties of the $f_{r}$'s, as we indeed do as well in Definition \ref{MFQMFdef} below. The assumption $f_{d}\neq0$ is made so that the depth of $f$ is precisely $d$ (otherwise it is smaller than $d$). Setting $\gamma=I$ in Equation \eqref{QMFdef} shows that $f_{0}=f$. A \emph{modular form} of weight $k$ and representation (or multiplier system) $\rho$ with respect to $\Gamma$, satisfying just
$f(\gamma\tau)=j_{\gamma}(\tau)^{k}\rho(\gamma)f(\tau)$ for every such $\tau$ and $\gamma$, is just a quasi-modular form of depth 0. Note that the transformation law from \cite{[Ze1]} (appearing in Equation (4) of that reference) allows for an anti-holomorphic power of $j(\gamma,\tau)$ to appear as well, while we restrict attention only to holomorphic powers. However, it is well-known that powers of $y=\Im\tau$ can always be used to avoid the anti-holomorphic powers, so that by appropriately normalizing all the differential operators these powers contribute no extra value to the theory. We adopt from \cite{[Ze1]} the following notation and definition.
\begin{defn}
Denote by $\mathcal{M}_{k}^{an}(\rho)$ the space of real-analytic modular forms of weight $k$ and representation $\rho$. The spaces $\mathcal{M}_{k}^{sing}(\rho)$, $\mathcal{M}_{k}^{hol}(\rho)$, and $\mathcal{M}_{k}^{mer}(\rho)$ are defined similarly, but with ``real--analytic'' replaced by ``real-analytic except for discrete singularities'', ``holomorphic'', and ``meromorphic'' respectively. For $\mathcal{M}_{k}^{sing}(\rho)$ we allow only singularities in discrete sets of points $\{\sigma\}_{\sigma\in\Sigma}$ such that around each $\sigma\in\Sigma$ the function grows to $\infty$ polynomially in $\frac{1}{|\tau-\sigma|}$. In case $\Gamma$ has cusps, elements of $\mathcal{M}_{k}^{hol}(\rho)$ are required to be holomorphic at the cusps, and we also have the space $\mathcal{M}_{k}^{cusp}(\rho)$ of cusp forms and the space $\mathcal{M}_{k}^{wh}(\rho)$ of weakly holomorphic modular forms (i.e., those modular forms that are holomorphic on $\mathcal{H}$ but may have poles at the cusps, i.e., whose expansion using the local variable at some cusp may involve a pole). In this case we shall assume that elements of $\mathcal{M}_{k}^{an}(\rho)$ may grow at the cusp, but at most exponentially (i.e., we demand the growth condition $F(\alpha\tau)=O(e^{cy})$ for some real number $c$ as $y\to\infty$ wherever $\alpha\infty$ is a cusp of $\Gamma$). The space of quasi-modular forms (of arbitrary depth) with the appropriate differential properties is denoted similarly, but with the symbol $\mathcal{M}$ replaced by $\widetilde{\mathcal{M}}$. The subspace $\widetilde{\mathcal{M}}_{k}^{*,\leq d}$ of $\widetilde{\mathcal{M}}_{k}^{*}$ (for $*$ being any of the superscripts indicating differential properties as above) denotes the space of those quasi-modular forms whose depth is bounded by $d$. We recall that a function is called \emph{nearly holomorphic}, or \emph{nearly meromorphic}, if it can be written as a polynomial in $\frac{1}{2iy}$ over the ring of holomorphic (resp. meromorphic) functions. We denote the spaces of nearly holomorphic and nearly meromorphic modular forms of weight $k$ and representation $\rho$ by $\mathcal{M}_{k}^{nh}(\rho)$ and $\mathcal{M}_{k}^{nm}(\rho)$ respectively. By $\mathcal{M}_{k}^{nh,\leq d}(\rho)$ and $\mathcal{M}_{k}^{nm,\leq d}(\rho)$ we mean those nearly holomorphic (resp. nearly meromorphic) modular forms whose \emph{depth}, defined to be their degree as a polynomial in $\frac{1}{2iy}$, is bounded by $d$. \label{MFQMFdef}
\end{defn}

\smallskip

We recall from \cite{[MR1]} that the ring $\bigoplus_{k\in\mathbb{Z}}\widetilde{\mathcal{M}}_{k}^{hol}\big(SL_{2}(\mathbb{Z})\big)$ of holomorphic quasi-modular forms on $SL_{2}(\mathbb{Z})$ is closed under the holomorphic differentiation $\partial_{\tau}=\frac{\partial}{\partial\tau}$, an operation that increases the weight by 2 and the depth by 1. On the other hand, an operation that preserves modularity, but not holomorphicity, is the weight raising operator $\delta_{k}=\partial_{\tau}+\frac{k}{2iy}$, also increasing the weight by 2. The ring $\bigoplus_{k\in\mathbb{Z}}\mathcal{M}_{k}^{nh}\big(SL_{2}(\mathbb{Z})\big)$ of nearly holomorphic modular forms on $SL_{2}(\mathbb{Z})$ is closed under the appropriate $\delta_{k}$ operations, and is canonically isomorphic to $\bigoplus_{k\in\mathbb{Z}}\widetilde{\mathcal{M}}_{k}^{hol}\big(SL_{2}(\mathbb{Z})\big)$ as $\mathbb{C}$-algebras that are graded (by the weight), filtered (by the depth), and differential (with the usual derivative and the $\delta_{k}$'s). The explicit isomorphism is given by the restriction to $f_{0}$ and $F_{0}$ of the maps from Theorem \ref{QMFVVMF} below. For the proof of this result, which we shall generalize in this paper to arbitrary Fuchsian groups and arbitrary representations or multiplier systems (scalar or vector-valued), see Lemma 118 and Propositions 131, 132, and 135 of \cite{[MR1]}. In addition, when we leave the holomorphic/meromorphic world, the weight lowering operator $y^{2}\partial_{\overline{\tau}}$ (where $\partial_{\overline{\tau}}=\frac{\partial}{\partial\overline{\tau}}$) also becomes of interest, as it decreases the weight of modular forms by 2. The paper \cite{[A]} considers an operator on (holomorphic) quasi-modular forms that lowers the weight by 2 and the depth by 1 by sending a quasi-modular form $f$, with functions $f_{r}$, $0 \leq r \leq d$ as in Equation \eqref{QMFdef}, to $f_{1}$. Indeed, Lemma 1.1 of \cite{[Ze1]} (generalizing Lemma 119 of \cite{[MR1]} and Proposition 2 of \cite{[A]}) shows that if $f\in\widetilde{\mathcal{M}}_{k}^{*,\leq d}(\rho)$ then $f_{r}\in\widetilde{\mathcal{M}}_{k-2r}^{*,\leq d-r}(\rho)$ for every $0 \leq r \leq d$, and that the associated function with index $0 \leq h \leq d-r$ is just $\binom{r+h}{r}f_{r+h}$, i.e., Equation \eqref{QMFdef} for $f_{r}$ becomes
\begin{equation}
f_{r}(\gamma\tau)=\sum_{h=0}^{d-r}j_{\gamma}(\tau)^{k-2r-h}(j_{\gamma}')^{h}\binom{r+h}{r}\rho(\gamma)f_{r+h}(\tau). \label{frQMF}
\end{equation}
We summarize these assertions as follows.
\begin{prop}
Let $k$, $\Gamma$, and $\rho$ be as above, let $d$ be a depth bound, and take $f\in\widetilde{\mathcal{M}}_{k}^{*,\leq d}(\rho)$ for $*$ being $an$ or $sing$, with associated functions $f_{r}$, $0 \leq r \leq d$. \begin{enumerate}[$(i)$]
\item We have $\partial_{\tau}f\in\widetilde{\mathcal{M}}_{k+2}^{*,\leq d+1}(\rho)$. The $r$th function associated to this quasi-modular form is $\partial_{\tau}f_{r}+(k+1-r)f_{r-1}$.
\item We also have $\frac{f}{-2iy}\in\widetilde{\mathcal{M}}_{k+2}^{*,\leq d+1}(\rho)$ with $\big(\frac{f}{-2iy}\big)_{r}$ being $\frac{f_{r}}{-2iy}+f_{r-1}$.
\item For $\delta_{k-d}f$ the $r$th respective function is $\delta_{k-d}f_{r}+(d+1-r)f_{r-1}$, so that $\delta_{k-d}f\in\widetilde{\mathcal{M}}_{k+2}^{*,\leq d}(\rho)$.
\item The function $y^{2}\partial_{\overline{\tau}}f$ is in $\widetilde{\mathcal{M}}_{k-2}^{*,\leq d}(\rho)$, with the $r$th function $y^{2}\partial_{\overline{\tau}}f_{r}$.
\item The map $f \mapsto f_{1}$ takes elements of $\widetilde{\mathcal{M}}_{k}^{*,\leq d}(\rho)$ to $\widetilde{\mathcal{M}}_{k-2}^{*,\leq d-1}(\rho)$.
\end{enumerate} \label{difopQMF}
\end{prop}
Note that for $r=0$ or for $r=d$ the formulae here involve the function $f_{-1}$ or $f_{d+1}$. These functions are naturally defined to be 0, the value with which all the formulae hold.
\begin{proof}
Part $(i)$ follows, as in Lemma 118 of \cite{[MR1]}, from the simple observation that $\partial_{\tau}(\gamma\tau)=\frac{1}{j(\gamma,\tau)^{2}}$, after differentiating Equation \eqref{QMFdef} with respect to $\tau$. Recalling that $\Im(\gamma\tau)=\frac{y}{|j(\gamma,\tau)|^{2}}$, we multiply Equation \eqref{QMFdef} by $\frac{|j(\gamma,\tau)|^{2}}{2iy}$ and observe that $\overline{j(\gamma,\tau)}=j(\gamma,\tau)-2iyj_{\gamma}'$ to deduce part $(ii)$. Part $(iii)$ is a consequence of parts $(i)$ and $(ii)$, with the extra observation that the associated function with index $d+1$ vanishes. Part $(iv)$ follows simply by applying $y^{2}\partial_{\overline{\tau}}f_{r}$ to Equation \eqref{QMFdef}, since $j_{\gamma}$ is holomorphic while $j_{\gamma}'$ and $\rho(\gamma)$ are constants that are independent of $\tau$. Part $(v)$ was already seen above to be contained in Lemma 1.1 of \cite{[Ze1]}.
\end{proof}

In particular, the observations that $\delta_{k}$ is a map from $\mathcal{M}_{k}^{*}(\rho)$ to $\mathcal{M}_{k+2}^{*}(\rho)$ and $y^{2}\partial_{\overline{\tau}}$ sends $\mathcal{M}_{k}^{*}(\rho)$ to $\mathcal{M}_{k-2}^{*}(\rho)$ (again with $*$ being just $an$ or $sing$) are the special cases of parts $(iii)$ and $(iv)$ respectively of Proposition \ref{difopQMF} in which $d=0$.

\smallskip

Before we turn to the other objects appearing in \cite{[Ze1]}, we show how Proposition \ref{difopQMF} can be used for obtaining a short and direct proof for the construction of Rankin--Cohen brackets on quasi-modular forms appearing in \cite{[MR2]}. Moreover, our argument generalizes the result of \cite{[MR2]} to a much broader context, since only the scalar-valued case with $\Gamma$ a congruence subgroup is considered in that reference (though the proof does not use these facts). In particular, this argument simplifies (and generalizes) the proof of the properties of the classical Rankin--Cohen brackets on modular forms, as the depth 0 case. We remark that in our convention the set $\mathbb{N}$ of natural numbers includes 0.

Consider two weights $k$ and $l$, two depths $d$ and $e$ from $\mathbb{N}$, and a parameter $n\in\mathbb{N}$. We call this set of parameters \emph{special} if both $d-k$ and $e-l$ are in $\mathbb{N}$ and simultaneously $n$ satisfies $\max\{d-k,e-l\}<n \leq d+e-k-l+1$, and \emph{generic} otherwise. Our result is proved for the generic case, and see Remark \ref{RC2dim} for the special situation.
\begin{thm}
Assume that the combination of the weights $k$ and $l$, the depths $d$ and $e$, and the number $n\in\mathbb{N}$ is generic. Then there exists a linear combination $[\cdot,\cdot]_{n;k,d;l,e}$, unique up to global scalar multiplication, of the bilinear operators of the form $(f,g)\mapsto\partial_{\tau}^{r}f\otimes\partial_{\tau}^{n-r}g$ with $0 \leq r \leq n$, for which if $f\in\widetilde{\mathcal{M}}_{k}^{*,\leq d}(\rho)$ and $g\in\widetilde{\mathcal{M}}_{l}^{*,\leq e}(\eta)$ (with $\rho$ and $\eta$ representations of the same group $\Gamma$, or multiplier systems corresponding to the weights $k$ and $l$ respectively), then $[f,g]_{n;k,d;l,e}$ lies in $\widetilde{\mathcal{M}}_{k+l+2n}^{*,\leq d+e}(\rho\otimes\eta)$. In one normalization we have \[[f,g]_{n;k,d;l,e}=\sum_{r=0}^{n}(-1)^{r}\binom{n}{r}\Bigg[\prod_{j=r}^{n-1}(k-d+j)\cdot\prod_{q=n-r}^{n-1}(l-e+q)\Bigg]
(\partial_{\tau}^{r}f\otimes\partial_{\tau}^{n-r}g).\] Here $*$ can be any of the types $an$, $sing$, $mer$, $hol$, and for $\Gamma$ with cusps also $wh$ or $cusp$. Moreover, if $\Gamma$ has cusps, $n\geq1$, and $*=hol$, then the resulting function $[f,g]_{n;k,d;l,e}$ is in $\widetilde{\mathcal{M}}_{k+l+2n}^{cusp,\leq d+e}(\rho\otimes\eta)$. \label{RCbrack}
\end{thm}

\begin{proof}
We recall that the weights and depths of quasi-modular forms are additive with respect to (tensor) products. Now, part $(iii)$ of Proposition \ref{difopQMF} shows that $\delta_{k-d}$ is the only differential operator of the form $\delta_{l}$ that does not increase the depth bound on elements of $\widetilde{\mathcal{M}}_{k}^{sing,\leq d}(\rho)$. Therefore the combinations of images of $f$ and $g$ in the designated spaces that land in $\widetilde{\mathcal{M}}_{k+l+2n}^{sing,\leq d+e}(\rho\otimes\eta)$ (without increasing depth further) are the tensor products $\delta_{k-d}^{r}f\otimes\delta_{l-e}^{n-r}g$, where $\delta_{m}^{s}$ is the natural combination $\delta_{m+2s-2}\circ\ldots\circ\delta_{m}$. However, these tensor products would involve derivatives of $f$ and $g$, but also division by $2iy$. We are therefore looking for the combinations not involving powers of $\frac{1}{2iy}$. In fact, this observation already implies that the Rankin--Cohen brackets on $\widetilde{\mathcal{M}}_{k}^{*,\leq d}(\rho)$ and $\widetilde{\mathcal{M}}_{l}^{*,\leq e}(\eta)$ must coincide with the classical Rankin--Cohen brackets on $\mathcal{M}_{k-d}^{*}(\rho)$ and $\mathcal{M}_{l-e}^{*}(\eta)$ (and they are, of course, independent of the representations $\rho$ and $\eta$).

Now, the power $\delta_{w}^{m}$ has an explicit expression in terms of $\partial_{\tau}$ and $\frac{1}{2iy}$, given in, e.g.,  Equation (56) of \cite{[Za]}. It reads $\delta_{w}^{m}\varphi=\sum_{u=0}^{m}\binom{m}{u}\big[\prod_{q=m-u}^{m-1}(w+q)\big]\frac{\partial_{\tau}^{m-u}\varphi}{(2iy)^{u}}$, where the empty product appearing for $u=0$ is defined to be 1. This formula is easily proved by induction on $m$. Therefore if $\{a_{r}\}_{r=0}^{n}$ are complex coefficients and $f$ and $g$ are as above, then the element $\sum_{r=0}^{n}\binom{n}{r}a_{r}\cdot\delta_{k-d}^{r}f\otimes\delta_{l-e}^{n-r}g$ is, explicitly, \[\sum_{r=0}^{n}\sum_{i=0}^{r}\sum_{p=0}^{n-r}\binom{n}{r}a_{r}\binom{r}{i}\binom{n-r}{p}\prod_{j=r-i}^{r-1}\!(k-d+j)\!
\!\prod_{q=n-r-p}^{n-r-1}\!\!(l-e+q)\frac{\partial_{\tau}^{r-i}f\otimes\partial_{\tau}^{n-r-p}g}{(2iy)^{i+p}}.\] After the summation index changes $t=i+p$ and $s=r-i$ (i.e., $r=s+i$ and $p=t-i$) and simple manipulations of the binomial coefficients, the latter expression becomes
\[\sum_{t=0}^{n}\sum_{s=0}^{n-t}\frac{n!}{s!t!(n\!-\!t\!-\!s)!}\bigg[\sum_{i=0}^{t}\binom{t}{i}a_{s+i}\!\!\prod_{j=s}^{s+i-1}\!\!(k-d+j)\!
\!\!\prod_{q=n-s-t}^{n-s-i-1}\!\!\!(l-e+q)\bigg]\frac{\partial_{\tau}^{s}f\otimes\partial_{\tau}^{n-t-s}g}{(2iy)^{t}}.\] The terms with $t=0$ give just $\sum_{s=0}^{n}\binom{n}{s}a_{s}\partial_{\tau}^{s}f\otimes\partial_{\tau}^{n-s}g$. The vanishing of the terms with $t=1$ yields $a_{s}(l-e+n-s-1)+a_{s+1}(k-d+s)=0$ for each $0 \leq s<n$, and we claim that in the generic case these equations already determine all the coefficients $a_{s}$ up to a global constant. This is clear if either $d-k$ or $e-l$ is not an integer between 0 and $n-1$. Moreover, if these two numbers are integers between 0 and $n-1$ the claim is still valid if $n \geq d+e-k-l+2$. Indeed, in this case we get the vanishing of $a_{s}$ for any $s \leq d-k$ and for any $s \geq n-e+l$, but the remaining coefficients still satisfy linear equations of co-dimension 1. It is easily seen that with $a_{r}=(-1)^{r}\binom{n}{r}\big[\prod_{j=r}^{n-1}(k-d+j)\cdot\prod_{q=n-r}^{n-1}(l-e+q)\big]$ the desired equalities hold, and we have $a_{r}\neq0$ for some $r$ (by the same considerations). Plugging this expression for $a_{s+i}$ into the coefficient of $\frac{\partial_{\tau}^{s}f\otimes\partial_{\tau}^{n-t-s}g}{(2iy)^{t}}$ above, we find that the sum over $i$ above becomes \[\prod_{j=s}^{n-1}(k-d+j)\cdot\prod_{q=n-t-s}^{n-1}(l-e+q)\cdot\bigg[\sum_{i=0}^{t}\binom{t}{i}(-1)^{s+i}\bigg],\] which is $(-1)^{s}$ times the binomial expansion of $(1-1)^{t}$ and thus vanishes for $t>0$ (and reduces trivially to $(-1)^{s}$ for $t=0$). Hence the combination with these $a_{r}$'s yields an operator of the desired form (i.e., involving only powers of $\partial_{\tau}$). Since $\partial_{\tau}$ preserves holomorphicity and does not generate new singularities, the superscript $sing$ can be replaced by any other superscript from Definition \ref{MFQMFdef}. Moreover, if $\Gamma$ has cusps then for $n>0$ the Rankin--Cohen brackets map $\widetilde{\mathcal{M}}_{k}^{hol,\leq d}(\rho)\otimes\widetilde{\mathcal{M}}_{l}^{hol,\leq e}(\eta)$ into $\widetilde{\mathcal{M}}_{k+l+2n}^{cusp,\leq d+e}(\rho\otimes\eta)$ since derivatives annihilate constant coefficients at the cusps.
\end{proof}

\begin{rmk}
In the case where the set of parameters from Theorem \ref{RCbrack} is special, one can show that the space of solutions to the equations involving the coefficients $a_{s}$ then has dimension 2. One solution has $a_{r}=\frac{(n-1-d+k)!(e-l+r-n)!}{(r-1-d+k)!(e-l)!}$ for $r \geq d-k+1$ and $a_{r}=0$ otherwise, and another solution is defined by taking $a_{r}=\frac{(n-1-e+l)!(d-k-r)!}{(n-1-r-e+l)!(d-k)!}$ if $r \leq n-e+l-1$ and $a_{r}=0$ otherwise. These solutions are linearly independent because we assume that $d-k+1>n-e+l-1$. One can verify that both these two solutions indeed give rise to holomorphic Rankin--Cohen brackets. The reason for this is \emph{Bol's identity} (which is a special case of Equation (56) of \cite{[Za]} mentioned above), from which it follows that when $d-k\in\mathbb{N}$ and $n \geq s>d-k$, the operator $\delta_{k-d}^{s}$ is just $\delta_{d-k+2}^{s-d+k-1}\circ\partial_{\tau}^{d-k+1}$. Similarly, if $e-l\in\mathbb{N}$ then $\delta_{l-e}^{n-s}$ is just $\delta_{e-l+2}^{n-s-e+l-1}\circ\partial_{\tau}^{e-l+1}$ wherever $n \geq n-s>e-l$ (this explains the vanishing of $a_{r}$ for $r \leq d-k$ in the former case and for $r \geq n-e+l$ in the latter one). Hence if $d-k$ is an integer between 0 and $n-1$, then $[f,g]_{n;k,d;l,e}$ is a constant multiple of $[\partial_{\tau}^{d-k+1}f,g]_{n-d+k-1;2d-k+2,d;l,e}$, and in case $e-l$ is an integer in that interval the former Rankin--Cohen brackets give a multiple of $[f,\partial_{\tau}^{e-l+1}g]_{n-e+l-1;k,d;2e-l+2,e}$ (the exact change of indices is explained by part $(iii)$ of Proposition \ref{difopQMF} and Bol's identity). Combining these cases, if both $d-k$ and $e-l$ are integers and $n \geq d+e-k-l+2$, then $[f,g]_{n;k,d;l,e}$ becomes a multiple of $[\partial_{\tau}^{d-k+1}f,\partial_{\tau}^{e-l+1}g]_{n-d-e+k+l-2;2d-k+2,d;2e-l+2,e}$ (one can also verify this using the formulae for the coefficients). Therefore when $n$ combines with the other parameters to a special set, the two bilinear operations sending $f$ and $g$ either to $[\partial_{\tau}^{d-k+1}f,g]_{n-d+k-1;2d-k+2,d;l,e}$ or to $[f,\partial_{\tau}^{e-l+1}g]_{n-e+l-1;k,d;2e-l+2,e}$ satisfy the desired properties, and they are linearly independent. \label{RC2dim}
\end{rmk}

We remark that the Leibnitz rule for Rankin--Cohen brackets appearing in Theorem 2 of \cite{[MR2]} can also be established using the non-holomorphic operators $\delta_{m}$. However, for this assertion the original proof from \cite{[MR2]} is much simpler.

\subsection{Operators on Vector--Valued Modular Forms}

Let $V_{m}$ be the $m$th symmetric power of the standard representation of $SL_{2}(\mathbb{R})$ on $\mathbb{C}^{2}$, and let $k$, $\Gamma$, and $\rho$ be as above. We allow ourselves the abuse of notation in which $V_{m}$ also denotes the vector space underlying the representation $V_{m}$, and we write elements of that vector space as sums of products of elements of $\mathbb{C}^{2}$ of total degree $m$. The paper \cite{[Ze1]} investigated the spaces $\mathcal{M}_{k}^{*}(V_{m}\otimes\rho)$ for various differential conditions $*$, and proved the following result (see Proposition 1.2 and Theorem 2.3 of that reference), which in particular extends Propositions 132 and 133 of \cite{[MR1]}.
\begin{thm}
For $*=an$ or $*=sing$ the spaces $\widetilde{\mathcal{M}}_{k}^{*,\leq m}(\rho)$, $\bigoplus_{s=0}^{m}\mathcal{M}_{k-2s}^{*}(\rho)$, and $\mathcal{M}_{k-m}^{*}(V_{m}\otimes\rho)$ are canonically isomorphic. The map from the first space to the second takes a quasi-modular form $f$ with functions $f_{r}$, $0 \leq r \leq m$ to \[F_{s}:\tau\mapsto\sum_{r=s}^{m}\binom{r}{s}\frac{f_{r}(\tau)}{(2iy)^{r-s}}\in\mathcal{M}_{k-2s}^{*}(\rho),\quad 0 \leq s \leq m.\] The inverse map sends the sequence $(F_{s})_{s=0}^{m}$ to the quasi-modular form \[f=f_{0}\quad\mathrm{with\ the\ functions}\quad f_{r}:\tau\mapsto\sum_{s=r}^{m}\binom{s}{r}\frac{F_{s}(\tau)}{(-2iy)^{s-r}},\quad 0 \leq r \leq m.\] Assuming that $f$ and $(F_{s})_{s=0}^{m}$ are related in this manner, the associated element $F\in\mathcal{M}_{k-m}^{*}(V_{m}\otimes\rho)$ is the vector-valued modular form defined by
\[F(\tau)=\sum_{s=0}^{m}\frac{F_{s}(\tau)}{(-2iy)^{s}}\binom{\tau}{1}^{m-s}\binom{\overline{\tau}}{1}^{s}=\sum_{r=0}^{m}f_{r}(\tau)\binom{\tau}{1}^{m-r}\binom{1}{0}^{r}\] (taking values in $V_{m} \otimes V_{\rho}$). Given such an $F$, the corresponding elements of $\widetilde{\mathcal{M}}_{k}^{*,\leq m}(\rho)$ and $\bigoplus_{s=0}^{m}\mathcal{M}_{k-2s}^{*}(\rho)$ are obtained by simply taking components in the appropriate bases of $V_{m}$. Moreover, $f$ is holomorphic (or meromorphic, or cuspidal, or weakly holomorphic) if and only if the $f_{r}$'s and $F$ have the same property. \label{QMFVVMF}
\end{thm}
For another point of view on the last assertion of Theorem \ref{QMFVVMF}, see the last part of Remark \ref{altpf} below. As already mentioned above (and in \cite{[MR1]}), for $f\in\widetilde{\mathcal{M}}_{k}^{hol,\leq m}(\rho)$ the modular form $F_{s}$ with $0 \leq s \leq m$ is not holomorphic, but rather nearly holomorphic. More precisely, it belongs to $\mathcal{M}_{k-2s}^{nh,\leq m-s}(\rho)$. Similarly, if $f$ is in $\widetilde{\mathcal{M}}_{k}^{mer,\leq m}(\rho)$, then $F_{s}\in\mathcal{M}_{k-2s}^{nm,\leq m-s}(\rho)$.

\smallskip

For every $m$ there is a map $i_{m}:V_{m} \to V_{m+1}$ defined through multiplication by $\binom{\tau}{1}$. Corollary 2.2 of \cite{[Ze1]} shows that applying it to modular forms yields a map, also denoted by $i_{m}$, from $\mathcal{M}_{k-m}^{*}(V_{m}\otimes\rho)$ to $\mathcal{M}_{k-m-1}^{*}(V_{m+1}\otimes\rho)$. The following immediate consequence of Theorem \ref{QMFVVMF} is also implicitly contained in \cite{[Ze1]}.
\begin{cor}
The map $i_{m}:\mathcal{M}_{k-m}^{*}(V_{m}\otimes\rho)\to\mathcal{M}_{k-m-1}^{*}(V_{m+1}\otimes\rho)$ commutes with the trivial inclusion of $\widetilde{\mathcal{M}}_{k}^{*,\leq m}(\rho)$ into $\widetilde{\mathcal{M}}_{k}^{*,\leq m+1}(\rho)$ and with the natural embedding of $\bigoplus_{s=0}^{m}\mathcal{M}_{k-2s}^{*}(\rho)$ into $\bigoplus_{s=0}^{m+1}\mathcal{M}_{k-2s}^{*}(\rho)$. The image of the latter map consists precisely of those sequences whose last coordinate vanishes. \label{depthim}
\end{cor}
The symbol $*$ in Corollary \ref{depthim} can be, in fact, any of the ones appearing in Definition \ref{MFQMFdef} (except on the spaces of sequences).

\smallskip

We now turn to operations on direct sums like $\bigoplus_{s=0}^{m}\mathcal{M}_{k-2s}^{*}(\rho)$ that change the weight. One such operation that increases the weight by 2 is letting, for each $s$, a multiple of $\delta_{k-2s}$ act on the $s$-th component. Another one is replacing $F_{s}$ by a multiple of $F_{s-1}$ (with $F_{-1}$ being 0). Any linear combination of these two operations will also do. For lowering the weight, one option is applying a scalar multiple of $y^{2}\partial_{\overline{\tau}}$ on each component. We may also replace $F_{s}$ by a multiple of $F_{s+1}$ (where $F_{m+1}=0$), or take any linear combination of these operations. We now find what are the operations on this space that correspond to those from Proposition \ref{difopQMF} under the isomorphism from Theorem \ref{QMFVVMF}.
\begin{prop}
Let $f\in\widetilde{\mathcal{M}}_{k}^{*,\leq m}(\rho)$ be as above, again with $*=an$ or $*=sing$, and denote the $m$-tuple of modular forms that is associated to $f$ in Theorem \ref{QMFVVMF} by $(F_{s})_{s=0}^{m}\in\bigoplus_{s=0}^{m}\mathcal{M}_{k-2s}^{*}(\rho)$, where we also set $F_{-1}=F_{m+1}=0$.
\begin{enumerate}[$(i)$]
\item Assuming that the depth of $f$ is at most $m-1$, the sequence corresponding to $\partial_{\tau}f\in\widetilde{\mathcal{M}}_{k+2}^{*,\leq m}(\rho)$ is $\big(\delta_{k-2s}F_{s}+(k+1-s)F_{s-1}\big)_{s=0}^{m}$.
\item Under the same assumption on $f$, the function $\frac{f}{-2iy}$ is associated with the translated sequence $(F_{s-1})_{s=0}^{m}$ (with vanishing component for $s=0$).
\item For any $f\in\widetilde{\mathcal{M}}_{k}^{*,\leq m}(\rho)$ the element $\delta_{k-m}f\in\widetilde{\mathcal{M}}_{k+2}^{*,\leq m}(\rho)$ is assigned to $\big(\delta_{k-2s}F_{s}+(m+1-s)F_{s-1}\big)_{s=0}^{m}$.
\item The quasi-modular form $y^{2}\partial_{\overline{\tau}}f$ corresponds to $\big(y^{2}\partial_{\overline{\tau}}F_{s}(\tau)+\frac{s+1}{4}F_{s+1}\big)_{s=0}^{m}$.
\item Considering $f_{1}$ as a quasi-modular form in $\widetilde{\mathcal{M}}_{k-2}^{*,\leq m-1}(\rho)$, the associated sequence is $\big((s+1)F_{s+1}\big)_{s=0}^{m}$ (with vanishing $m$th function).
\end{enumerate} \label{difopseq}
\end{prop}
Note that in part $(iii)$ here elements of depth precisely $m$ are allowed (unlike the previous parts), and for such elements $\delta_{k-m}$ does not increase the depth.

\begin{proof}
We recall from Theorem \ref{QMFVVMF} that $F_{s}(\tau)=\sum_{r=s}^{m}\binom{r}{s}\frac{f_{r}(\tau)}{(2iy)^{r-s}}$ (so that indeed $F_{m}=0$ if and only if $f\in\widetilde{\mathcal{M}}_{k}^{*,\leq m-1}(\rho)$). We do the same transformation on the sequences of functions appearing in the various parts of Proposition \ref{difopQMF}, and express the results in terms of the original functions $F_{s}$. In part $(i)$, with $F_{m}=0$, we consider $\sum_{r=s}^{m}\binom{r}{s}\frac{\partial_{\tau}f_{r}(\tau)+(k+1-r)f_{r-1}(\tau)}{(2iy)^{r-s}}$. Comparing the first term with index $r=s$ with the definition of $F_{s}$, the effect of the derivative should be related to $\delta_{k-2s}F_{s}$. Recalling that $\delta_{l}=y^{-l}\partial_{\tau}y^{l}$, we get the useful formula
\begin{equation}
\delta_{l}\bigg(\frac{g(\tau)}{y^{t}}\bigg)=\frac{\delta_{l-t}g(\tau)}{y^{t}}. \label{deltalyt}
\end{equation}
From Equation \eqref{deltalyt} it follows that \[\delta_{k-2s}F_{s}(\tau)=\sum_{r=s}^{m}\binom{r}{s}\frac{\delta_{k-r-s}f_{r}(\tau)}{(2iy)^{r-s}}=
\sum_{r=s}^{m}\binom{r}{s}\frac{\partial_{\tau}f_{r}(\tau)+\frac{k-r-s}{2iy}f_{r}(\tau)}{(2iy)^{r-s}}.\] Therefore subtracting $\delta_{k-2s}F_{s}$ from the desired function indeed cancels the terms involving $\partial_{\tau}f_{r}$, and after replacing the summation index in the remaining terms of the desired function, we see that the difference becomes just
\[\tau\mapsto\sum_{r=s-1}^{m-1}\bigg[(k-r)\binom{r+1}{s}-\binom{r}{s}(k-r-s)\bigg]\frac{f_{r}(\tau)}{(2iy)^{r-s+1}}.\] Note that changing the range of summation is valid, since the term with $s-1$ in the part coming from $\delta_{k-2s}F_{s}$ vanishes (since $\binom{s-1}{s}=0$), and the terms with $s=m$ can also be omitted since we also assume that $f_{m}=0$. The coefficient in brackets can be simplified to $(k-r)\binom{r}{s-1}+s\binom{r}{s}$, and the second term $\frac{r!}{(s-1)!(r-s)!}$ can also be written as $(r+1-s)\binom{r}{s-1}$ (this is valid for $r=s-1$ as well, since then both $\binom{r}{s}$ and the coefficient $r+1-s$ vanish and $(r-s)!$ in the denominator becomes infinite). We thus get a global coefficient of $k+1-s$ times $\sum_{r=s-1}^{m-1}\binom{r}{s-1}\frac{f_{r}(\tau)}{(2iy)^{r-s+1}}$, which proves part $(i)$ since the latter sum is $F_{s-1}$ by definition.

For part $(ii)$ we have to evaluate the expression $\sum_{r=s}^{m}\binom{r}{s}\frac{f_{r-1}(\tau)-(f_{r}/2iy)}{(2iy)^{r-s}}$ arising from part $(ii)$ of Proposition \ref{difopQMF}, which after a similar summation index change becomes $\sum_{r=s-1}^{m-1}\big[\binom{r+1}{s}-\binom{r}{s}\big]\frac{f_{r}(\tau)}{(2iy)^{r-s+1}}$. The required assertion is thus a consequence of the definition of $F_{s-1}$. Part $(iii)$ follows from parts $(i)$ and $(ii)$, and the extension to quasi-modular forms of depth $m$ is well-defined since the undesired function with index $m+1$ vanishes because of the coefficient $m+1-s$ (in correspondence with the depth assertion in part $(iii)$ of Proposition \ref{difopQMF}). Next, we evaluate the difference between the expansion $\sum_{r=s}^{m}\binom{r}{s}\frac{y^{2}\partial_{\overline{\tau}}f_{r}(\tau)}{(2iy)^{r-s}}$ from part $(iv)$ of Proposition \ref{difopQMF} and $y^{2}\partial_{\overline{\tau}}F_{s}(\tau)$ as \[-\sum_{r=s}^{m}\binom{r}{s}f_{r}(\tau) \cdot y^{2}\frac{\partial(\tau-\overline{\tau})^{s-r}}{\partial\overline{\tau}}=-\sum_{r=s}^{m}\binom{r}{s}(r-s)f_{r}(\tau)\frac{y^{2}}{(2iy)^{r+1-s}}.\] Since this simplifies to $+\frac{s+1}{4}\sum_{r=s}^{m}\binom{r}{s+1}\frac{f_{r}(\tau)}{(2iy)^{r-1-s}}=+\frac{s+1}{4}F_{s+1}(\tau)$ by the considerations from above, this establishes part $(iv)$ as well. Finally, Equation \eqref{frQMF} directs us to consider the sum $\sum_{r=s}^{m}\binom{r}{s}\frac{(r+1)f_{r+1}(\tau)}{(2iy)^{r-s}}$ for part $(v)$ here, where we can omit the term with $r=m$ since $f_{m+1}=0$. The combinatorial coefficient coincides with $(s+1)\binom{r+1}{s+1}$, and a summation index change now identifies this sum as the desired expression $(s+1)F_{s+1}$.
\end{proof}

\smallskip

Next, we consider the equivalent operators on the space $\mathcal{M}_{k-m}^{sing}(V_{m}\otimes\rho)$ of vector-valued modular forms. All the results extend to the case with the superscript $an$ instead of $sing$ as well. The operators $\delta_{k-m}$ and $y^{2}\partial_{\overline{\tau}}$ would again take elements of this space to $\mathcal{M}_{k-m+2}^{sing}(V_{m}\otimes\rho)$ and $\mathcal{M}_{k-m-2}^{sing}(V_{m}\otimes\rho)$ respectively. But also here there are other weight changing operators, defined on appropriate subspaces of $\mathcal{M}_{k-m}^{sing}(V_{m}\otimes\rho)$, that do not involve differentiation. Indeed, apart from the map $i_{m}$ from Corollary 2.2 of \cite{[Ze1]} (or Corollary \ref{depthim} here), one can define the complex conjugate map $\overline{i}_{m}:V_{m} \to V_{m+1}$ using multiplication by the complex conjugate vector $\binom{\overline{\tau}}{1}$. In addition, the same (simple) argument from Corollary 2.2 of \cite{[Ze1]} proves that $\frac{\overline{i}_{m}}{-2iy}$ takes elements of $\mathcal{M}_{k-m}^{sing}(V_{m}\otimes\rho)$ injectively into $\mathcal{M}_{k-m+1}^{sing}(V_{m+1}\otimes\rho)$, and the image consists of those elements \[\widetilde{F}\in\mathcal{M}_{k-m+1}^{sing}(V_{m-1}\otimes\rho),\quad\widetilde{F}(\tau)=\sum_{m_{+}+m_{-}=m+1}\widetilde{f}^{m_{+},m_{-}}(\tau)\binom{\tau}{1}^{m_{+}}\binom{\overline{\tau}}{1}^{m_{-}},\] in which the coefficient $\widetilde{f}^{m+1,0}$ vanishes. Combining these arguments with the injectivity of $i_{m-1}$ and $\frac{\overline{i}_{m-1}}{-2iy}$, we obtain the following weight-changing operators.
\begin{lem}
Let an element $F\in\mathcal{M}_{k-m}^{sing}(V_{m}\otimes\rho)$ be given.
\begin{enumerate}[$(i)$]
\item If $F \in i_{m-1}\big(\mathcal{M}_{k-m+1}^{sing}(V_{m-1}\otimes\rho)\big)$, then $\frac{\overline{i}_{m-1}}{-2iy}(i_{m-1}^{-1}F)\in\mathcal{M}_{k+2-m}^{sing}(V_{m}\otimes\rho)$.
\item When $F$ is in $\frac{\overline{i}_{m-1}}{-2iy}\big(\mathcal{M}_{k-m-1}^{sing}(V_{m-1}\otimes\rho)\big)$, we have that the modular form $i_{m-1}\circ\big(-2iy\overline{i}_{m-1}^{\ -1}F\big)$ lies in $\mathcal{M}_{k-2-m}^{sing}(V_{m}\otimes\rho)$.
\item Composing the operator from part $(ii)$ with the operator multiplying the coefficient $\widetilde{f}^{m_{+},m_{-}}$ of $\binom{\tau}{1}^{m_{+}}\binom{\overline{\tau}}{1}^{m_{-}}$ in $F$ by a constant $c_{m_{+},m_{-}}$ also yields a weight lowering operator. If $c_{m,0}=0$ then the resulting composition is defined to all the elements of $\mathcal{M}_{k-m}^{sing}(V_{m}\otimes\rho)$, not only $\frac{\overline{i}_{m-1}}{-2iy}$-images.
\end{enumerate} \label{wcShimnodif}
\end{lem}

\begin{rmk}
The last statement in part $(iii)$ of Lemma \ref{wcShimnodif} holds because multiplying $\widetilde{f}^{m,0}$ by 0 always yields an $\frac{\overline{i}_{m-1}}{-2iy}$-image. A composition as in part $(iii)$ of Lemma \ref{wcShimnodif} but with the operator from part $(i)$ of that lemma also defines a weight raising operator (which can be defined on the whole space $\mathcal{M}_{k-m}^{sing}(V_{m}\otimes\rho)$ if $c_{0,m}=0$), but this will not be of use in what follows. The operators that will later show up are the one from part $(i)$ and the one from part $(iii)$ with $c_{m_{+},m_{-}}=m_{-}$. We denote the composition with $-2iy\overline{i}_{m-1}^{\ -1}$ in this case by $D_{m}$ (so that the operator from part $(iii)$ is $i_{m-1} \circ D_{m}$), since it looks like ``differentiating the vectors $\frac{1}{(-2iy)^{m_{-}}}\binom{\tau}{1}^{m_{+}}\binom{\overline{\tau}}{1}^{m_{-}}$ with respect to $\frac{1}{-2iy}\binom{\overline{\tau}}{1}$'' (this interpretation will be useful in the point of view presented in Remark \ref{limops} below). \label{Dmdef}
\end{rmk}

The connections relating the operators $\delta_{k-m}$ and $y^{2}\partial_{\overline{\tau}}$ and the ones from Lemma \ref{wcShimnodif} and Remark \ref{Dmdef} to those appearing in Propositions \ref{difopQMF} and \ref{difopseq} are as follows.
\begin{thm}
Let $F\in\mathcal{M}_{k-m}^{sing}(V_{m}\otimes\rho)$ be associated with the quasi-modular form $f\in\widetilde{\mathcal{M}}_{k}^{sing,\leq m}(\rho)$, with functions $f_{r}$, $0 \leq r \leq m$, as in Theorem \ref{QMFVVMF}.
\begin{enumerate}[$(i)$]
\item The modular form $\delta_{k-m}F\in\mathcal{M}_{k+2-m}^{sing}(V_{m}\otimes\rho)$ is associated with $\delta_{k-m}f$.
\item When $F$ is an $i_{m-1}$-image (so that $f\in\widetilde{\mathcal{M}}_{k}^{sing,\leq m-1}(\rho)$), the modular form $\frac{\overline{i}_{m-1}}{-2iy}(i_{m-1}^{-1}F)\in\mathcal{M}_{k+2-m}^{sing}(V_{m}\otimes\rho)$ corresponds to $\frac{f}{-2iy}\in\widetilde{\mathcal{M}}_{k+2}^{sing,\leq m}(\rho)$.
\item Under the restriction from part $(ii)$, the modular form assigned to $\delta_{l}f$ (in particular to $\partial_{\tau}f=\delta_{0}f$) is $\delta_{k-m}F+(k-m-l)\frac{\overline{i}_{m-1}}{-2iy}(i_{m-1}^{-1}F)$.
\item With $y^{2}\partial_{\overline{\tau}}f\in\widetilde{\mathcal{M}}_{k-2}^{sing,\leq m}(\rho)$ we associate $y^{2}\partial_{\overline{\tau}}F\in\mathcal{M}_{k-2-m}^{sing}(V_{m}\otimes\rho)$.
\item Considering $f_{1}\in\widetilde{\mathcal{M}}_{k-2}^{sing,\leq m-1}(\rho)$ as embedded in $\widetilde{\mathcal{M}}_{k-2}^{sing,\leq m}(\rho)$, the corresponding modular form in $\mathcal{M}_{k-2-m}^{sing}(V_{m}\otimes\rho)$ is $i_{m-1}(D_{m}F)$, where $D_{m}$ is defined in Remark \ref{Dmdef}.
\end{enumerate} \label{difopShim}
\end{thm}

\begin{proof}
We denote the element of $\bigoplus_{s=0}^{m}\mathcal{M}_{k-2s}^{sing}(\rho)$ associated with $f$ and $F$ by $(F_{s})_{s=0}^{m}$, and recall the two presentations of $F$ from Theorem \ref{QMFVVMF}. We begin by evaluating $\delta_{k-m}F$, using the presentation of $F$ in terms of the functions $f_{r}$ (with $f_{-1}=0$). As the part $\partial_{\tau}$ of $\delta_{k-m}$ has to operate also on the vectors $\binom{\tau}{1}^{m-r}\binom{1}{0}^{r}$, yielding a contribution involving the next vector $(m-r)\binom{\tau}{1}^{m-r-1}\binom{1}{0}^{r+1}$, the function $\delta_{k-m}F$ is easily seen (by a simple summation index change) to be associated with the quasi-modular form for which the $r$th function appearing in Equation \eqref{QMFdef} is $\delta_{k-m}f_{r}+(m+1-r)f_{r-1}$. Part $(i)$ here hence follows from part $(iii)$ of Proposition \ref{difopQMF}. For part $(ii)$ we use the expression for $F$ involving the functions $F_{s}$ in Theorem \ref{QMFVVMF}. It is clear from the definition, using another summation index change, that if $F$ is an $i_{m-1}$-image (i.e., if $F_{m}=0$), then $\frac{\overline{i}_{m-1}}{-2iy}\big(i_{m-1}^{-1}F\big)$ takes $\tau$ to
$\sum_{s=1}^{m}\frac{F_{s-1}(\tau)}{(-2iy)^{s}}\binom{\tau}{1}^{m-s}\binom{\overline{\tau}}{1}^{s}$ (with no coefficient associated with $s=0$ since $F_{-1}=0$). Our part $(ii)$ is therefore a consequence of part $(ii)$ of Proposition \ref{difopseq}, and part $(iii)$ is a linear combination of parts $(i)$ and $(ii)$. Considering the presentation of $F$ using the functions $f_{r}$ once again, and observing that the vectors $\binom{\tau}{1}^{m-r}\binom{1}{0}^{r}$ are holomorphic, we deduce that $y^{2}\partial_{\overline{\tau}}$ operates on $F$ only through its action on the components $f_{r}$, $0 \leq r \leq m$. The desired part $(iv)$ therefore follows from part $(iv)$ of Proposition \ref{difopQMF}. Finally, we apply the operator $D_{m}$ from Remark \ref{Dmdef} on the presentation of $F$ involving the $F_{s}$'s in Theorem \ref{QMFVVMF}, and it is clear (after yet another summation index change) that $i_{m-1}(D_{m}F)$ sends $\tau$ to $\sum_{s=0}^{m-1}\frac{(s+1)F_{s+1}(\tau)}{(-2iy)^{s}}\binom{\tau}{1}^{m-s}\binom{\overline{\tau}}{1}^{s}$. The required part $(v)$ is then established by applying part $(v)$ of Proposition \ref{difopseq}.
\end{proof}

\begin{rmk}
The parts of Theorem \ref{difopShim} can be interpreted as completing commutative diagrams. For example, parts $(i)$ is described as by the commutativity of the diagram
\[\begin{array}{ccc}\widetilde{\mathcal{M}}_{k}^{sing,\leq m}(\rho) & \stackrel{\sim}{\longrightarrow} & \mathcal{M}_{k-m}^{sing}(V_{m}\otimes\rho) \\ \\ \delta_{k-m}\Big\downarrow & & \Big\downarrow\delta_{k-m} \\ \\ \widetilde{\mathcal{M}}_{k+2}^{sing,\leq m}(\rho) & \stackrel{\sim}{\longrightarrow} & \mathcal{M}_{k+2-m}^{sing}(V_{m}\otimes\rho).\end{array}\] For part $(ii)$ (resp. $(iii)$) we replace the symbol $\delta_{k-m}$ on the arrow on the left by $f\mapsto\frac{f}{-2iy}$ (resp. $\delta_{l}$), while the one on the right should be $\frac{\overline{i}_{m-1}}{-2iy} \circ i_{m-1}^{-1}$ (resp. $\delta_{k-m}+(k-m-l)\big(\frac{\overline{i}_{m-1}}{-2iy} \circ i_{m-1}^{-1}\big)$). In the diagram corresponding to part $(iv)$ (resp. $(v)$) the weights $k+2$ and $k+2-m$ have to be $k-2$ and $k-2-m$ respectively, the symbol on the arrow on the left should be $y^{2}\partial_{\overline{\tau}}$ (resp. $f \mapsto f_{1}$), and on the arrow on the left we put $y^{2}\partial_{\overline{\tau}}$ (resp. $i_{m-1} \circ D_{m}$). \label{comdiaiso}
\end{rmk}

\begin{rmk}
Parts $(ii)$ and $(v)$ of Theorem \ref{difopShim} could have been proven using Proposition \ref{difopQMF} alone, by applying the equality $\binom{\overline{\tau}}{1}=\binom{\tau}{1}-2iy\binom{1}{0}$ (indeed, $i_{m-1} \circ D_{m}$ takes each vector $\binom{\tau}{1}^{m-r}\binom{1}{0}^{r}$ to $r\binom{\tau}{1}^{m+1-r}\binom{1}{0}^{r-1}$). Similarly, differentiating the expression for $F$ involving the functions $F_{s}$ in Theorem \ref{QMFVVMF} and using the same equality provides an alternative proof for parts $(i)$ and $(iv)$ as well, using just Proposition \ref{difopseq}. However, the proof we chose for each part is the simpler one. In addition, the fact that any modular or quasi-modular form is meromorphic on $\mathcal{H}$ if and only if it is annihilated by the operator $y^{2}\partial_{\overline{\tau}}$ (assuming that it has no essential singularities, a situation that we exclude also with the superscripts $sing$, $mer$ etc. in any case) suggests another proof of the holomorphicity/meromorphicity assertion in Theorem \ref{QMFVVMF}, as an application of part $(iv)$ of Theorem \ref{difopShim}. \label{altpf}
\end{rmk}

\begin{rmk}
Our operators also have a geometric origin, which we now explain. The representations $V_{m}$ can be seen as complex local systems on $\mathcal{H}$ and its quotients by discrete groups. If $\mathcal{V}_{m}$ is the vector bundle arising from $V_{m}$ (via the tensor product with the structure sheaf), then there is a natural connection $\nabla$, called the \emph{Gauss--Manin connection}, that takes sections of $\mathcal{V}_{m}$ to sections of $\mathcal{A}^{1}\otimes\mathcal{V}_{m}$, where $\mathcal{A}^{1}$ represents real-analytic differential forms on the quotient $\Gamma\backslash\mathcal{H}$. The splitting of differential forms in $\mathcal{A}^{1}$ into those involving $d\tau$ and those with $d\overline{\tau}$ decomposes $\nabla$ naturally as the sum of $\nabla^{h}$ (resulting in expressions involving $d\tau$) and $\nabla^{\overline{h}}$ (attaining differential forms with $d\overline{\tau}$). Now, $\mathcal{V}_{m}$ carries a natural Hodge filtration, in which $F^{p}\mathcal{V}_{m}$ is the sub-bundle whose sections are precisely those vector-valued modular forms that are associated to the subspace $\widetilde{\mathcal{M}}_{k}^{an,\leq m-p}(\rho)$ of $\widetilde{\mathcal{M}}_{k}^{an,\leq m}(\rho)$ in Theorem \ref{QMFVVMF}. Thus $\mathcal{V}_{m}$ can be viewed as a variation of Hodge structures on $\mathcal{H}$. Then each of $\nabla^{h}$ and $\nabla^{\overline{h}}$ decomposes again into two components, one of which takes sections with a given Hodge weight to sections having the same Hodge weight. The other one shifts the Hodge weight by 1, and is $C^{\infty}$-linear. These components are evaluated, for the vector bundle $\mathcal{V}_{m}$, in \cite{[Zu]}, which is a good reference for many details of this construction. Once the sections are determined, we may take the tensor product with the vector space $V_{\rho}$ (so that the operators are tensored with $Id_{V_{\rho}}$), and get the spaces from Definition \ref{MFQMFdef}. It is easily seen that $\nabla^{h}$ increases the weight by 2, while $\nabla^{\overline{h}}$ decrease it by 2, and that they come from the operators $\delta_{k-m}$ and $y^{2}\partial_{\overline{\tau}}$ on $\mathcal{M}_{k-m}^{an}(V_{m}\otimes\rho)$. Recall from Remark \ref{comdiaiso} that parts $(i)$ and $(iv)$ of Theorem \ref{difopShim} can also be proved using the action on elements of $\mathcal{M}_{k-m}^{an}(V_{m}\otimes\rho)$ presented as in Theorem \ref{QMFVVMF} with the functions $F_{s}$. The evaluations leading to these alternative proofs also imply that the $C^{\infty}$-linear, Hodge weight changing part of $\nabla^{\overline{h}}$ is just our operator $i_{m-1} \circ D_{m}$ (up to some normalizing scalars). The other part of $\nabla^{\overline{h}}$ operates just as $y^{2}\partial_{\overline{\tau}}$ on the components $F_{s}$, $0 \leq s \leq m$. For the decomposition of $\nabla^{h}$ we get an analogous picture, involving some multiple of the operator $\frac{\overline{i}_{m-1}}{-2iy} \circ i_{m-1}^{-1}$ (which depends on the actual definition of $\nabla^{h}$, whether it is $\delta_{k-m}$ or some combination like in part $(iii)$ of Theorem \ref{difopShim}), as well as an operator that is ``complex conjugate'' to $i_{m-1} \circ D_{m}$ up to powers of $2iy$. \label{geom}
\end{rmk}

\smallskip

We recall from \cite{[Ze1]} that the multiplicative structure of the ring of quasi-modular forms was adapted not to a single representation $V_{m}$, but rather to their direct limit using the maps $i_{m}$. Indeed, that reference shows that if $V_{\infty}$ in the direct limit of the spaces $V_{m}$ via the $i_{m}$'s (which is an infinite-dimensional space), then the direct limit of the spaces $\mathcal{M}_{k-m}^{*}(V_{m}\otimes\rho)$ (for any superscript $*$ from Definition \ref{MFQMFdef}) with respect to the $i_{m}$'s is the space $\mathcal{M}_{k-\infty}^{*}(V_{\infty}\otimes\rho)$. Elements of the latter space can be presented either as $\tau\mapsto\sum_{s}\frac{F_{s}(\tau)}{(-2iy)^{s}}\binom{\tau}{1}^{\infty-s}\binom{\overline{\tau}}{1}^{s}$ or as $\tau\mapsto\sum_{r}f_{r}(\tau)\binom{\tau}{1}^{\infty-r}\binom{1}{0}^{r}$, where the sums over both $r$ and $s$ are finite, and the functions $F_{s}$ and $f_{r}$ are related precisely as in Theorem \ref{QMFVVMF}.

We therefore investigate the relations between the operators from Theorem \ref{difopShim} and $i_{m}$. Using the fact that $i_{m}$ corresponds to the natural inclusion of $\widetilde{\mathcal{M}}_{k}^{sing,\leq m}(\rho)$ in $\widetilde{\mathcal{M}}_{k}^{sing,\leq m+1}(\rho)$, these relations are as follows.
\begin{cor}
\begin{enumerate}[$(i)$]
\item For $F\in\mathcal{M}_{k-m}^{sing}(V_{m}\otimes\rho)$ we get $y^{2}\partial_{\overline{\tau}}(i_{m}F)=i_{m}(y^{2}\partial_{\overline{\tau}}F)$ and $i_{m}[(i_{m-1} \circ D_{m})F]=(i_{m} \circ D_{m+1})(i_{m}F)$, while for $\delta_{k-m}$ the equality $i_{m}(\delta_{k-m}F)=\big[\delta_{k-m-1}-\big(\frac{\overline{i}_{m}}{-2iy} \circ i_{m}^{-1}\big)\big](i_{m}F)$ holds. In case $F$ is an $i_{m-1}$-image we get $i_{m}\big[\big(\frac{\overline{i}_{m-1}}{-2iy} \circ i_{m-1}^{-1}\big)F\big]=\big(\frac{\overline{i}_{m}}{-2iy} \circ i_{m}^{-1}\big)(i_{m}F)$. Thus, for every $l$ the modular form $i_{m}\big[\delta_{k-m}F+(k-m-l)\frac{\overline{i}_{m-1}}{-2iy}(i_{m-1}^{-1}F)\big]$ coincides with $\delta_{k-m-1}(i_{m}F)+(k-m-1-l)\big(\frac{\overline{i}_{m}}{-2iy} \circ i_{m}^{-1}\big)(i_{m}F)$.
\item There exist limit operators on $\mathcal{M}_{k-\infty}^{sing}(V_{\infty}\otimes\rho)$, which we denote by $y^{2}\partial_{\overline{\tau}}$, $\frac{\overline{i}}{-2iy}$, $D$, and $\tilde{\delta}_{l}$, that commute with the embedding of $\mathcal{M}_{k-m}^{sing}(V_{m}\otimes\rho)$ into $\mathcal{M}_{k-\infty}^{sing}(V_{\infty}\otimes\rho)$ and the operators $y^{2}\partial_{\overline{\tau}}$, $\frac{\overline{i}_{m-1}}{-2iy} \circ i_{m-1}^{-1}$, $i_{m-1} \circ D_{m}$, and $\delta_{k-m}+(k-m-l)\big(\frac{\overline{i}_{m-1}}{-2iy} \circ i_{m-1}^{-1}\big)$ respectively.
\end{enumerate} \label{comim}
\end{cor}

\begin{proof}
Apply parts $(iv)$, $(v)$, $(i)$, $(ii)$ and $(iii)$ of Theorem \ref{difopShim} respectively, and observe that the maps taking $f\in\widetilde{\mathcal{M}}_{k}^{sing,\leq m}(\rho)$ to $y^{2}\partial_{\overline{\tau}}f$, $f_{1}$, and $\delta_{k-m}f$, as well as to $\frac{f}{-2iy}$ and to $\delta_{l}f$ in case $f$ is in $\widetilde{\mathcal{M}}_{k}^{sing,\leq m-1}(\rho)$, do not distinguish between $f$ as an element of $\widetilde{\mathcal{M}}_{k}^{sing,\leq m}(\rho)$ or of $\widetilde{\mathcal{M}}_{k}^{sing,\leq m+1}(\rho)$. This establishes part $(i)$, and part $(ii)$ immediately follows by taking the limit (the latter is slightly more natural when part $(i)$ is viewed as in Remark \ref{comdiaim} below).
\end{proof}

\begin{rmk}
Part $(i)$ of Corollary \ref{comim} can be interpreted as extending the diagrams from Remark \ref{comdiaiso} to larger commutative diagrams involving $i_{m}$. The one with $y^{2}\partial_{\overline{\tau}}$ thus becomes
\[\begin{array}{ccccc}\widetilde{\mathcal{M}}_{k}^{sing,\leq m}(\rho) & \stackrel{\sim}{\longrightarrow} & \mathcal{M}_{k-m}^{sing}(V_{m}\otimes\rho) & \stackrel{i_{m}}{\longrightarrow} & \mathcal{M}_{k-1-m}^{sing}(V_{m+1}\otimes\rho) \\ \\ y^{2}\partial_{\overline{\tau}}\Big\downarrow & & \Big\downarrow y^{2}\partial_{\overline{\tau}} & & \Big\downarrow y^{2}\partial_{\overline{\tau}} \\ \\ \widetilde{\mathcal{M}}_{k-2}^{sing,\leq m}(\rho) & \stackrel{\sim}{\longrightarrow} & \mathcal{M}_{k-2-m}^{sing}(V_{m}\otimes\rho) & \stackrel{i_{m}}{\longrightarrow} & \mathcal{M}_{k-3-m}^{sing}(V_{m+1}\otimes\rho),\end{array}\] while the diagram with $f \mapsto f_{1}$ and $i_{m-1} \circ D_{m}$ is completed with $i_{m} \circ D_{m+1}$ on the arrow on the right. In the remaining three diagrams the space on the lower right corner is $\mathcal{M}_{k+1-m}^{sing}(V_{m+1}\otimes\rho)$, and when the two arrows going down on the left and on the middle are both $\delta_{k-m}$ (resp. $f\mapsto\frac{f}{-2iy}$ and $\frac{\overline{i}_{m-1}}{-2iy} \circ i_{m-1}^{-1}$, resp. $\delta_{l}$ and $\delta_{k-m}+(k-m-l)\big(\frac{\overline{i}_{m-1}}{-2iy} \circ i_{m-1}^{-1}\big)$), the one on the right completes a commutative diagram with $\delta_{k-m-1}-\big(\frac{\overline{i}_{m}}{-2iy} \circ i_{m}^{-1}\big)$ (resp. $\frac{\overline{i}_{m}}{-2iy} \circ i_{m}^{-1}$, resp. $\delta_{k-m-1}+(k-m-1-l)\big(\frac{\overline{i}_{m}}{-2iy} \circ i_{m}^{-1}\big)$). \label{comdiaim}
\end{rmk}

\begin{rmk}
We note that the assertions from part $(i)$ of Corollary \ref{comim} involving $\frac{\overline{i}_{m}}{-2iy}$ or $D_{m}$ also follow directly from the commutative structure of symmetric powers and the fact that $D_{m}$ is based on coefficients that depend only on $m_{-}$ (indeed, all the $D_{m}$'s were seen in Remark \ref{Dmdef} to behave like taking derivatives with respect to $\frac{1}{-2iy}\binom{\overline{\tau}}{1}$, independently of the power of $\binom{\tau}{1}$). This is also the case for $y^{2}\partial_{\overline{\tau}}$, since the anti-holomorphic differentiation does not operate on the extra vector in the formula $i_{m}F=F\cdot\binom{\tau}{1}$. For the remaining operators we can use the Leibnitz rule for $\delta$ operators (which is a simple consequence of the usual Leibnitz rule for $\partial_{\tau}$), decomposing $\delta_{k-m-1}(i_{m}F)$ as the sum of $i_{m}(\delta_{k-m}F)$ and $F\cdot\delta_{-1}\binom{\tau}{1}$, and the latter multiplier of $F$ is just $-\binom{\overline{\tau}}{1}/2iy$. Part $(ii)$ of that corollary can be considered as the definition of the operators $\frac{\overline{i}}{-2iy}$, $D$, and $\tilde{\delta}_{l}$, as well as the assertion that the limit of the operators $y^{2}\partial_{\overline{\tau}}$ on the spaces $\mathcal{M}_{k-m}^{sing}(V_{m}\otimes\rho)$ is just the same operator on $\mathcal{M}_{k-\infty}^{sing}(V_{\infty}\otimes\rho)$. The first two operators are defined on the space $V_{\infty}$ itself, sending $\frac{1}{(-2iy)^{s}}\binom{\tau}{1}^{\infty-s}\binom{\overline{\tau}}{1}^{s}$ to $\frac{1}{(-2iy)^{s}}\binom{\tau}{1}^{\infty-s-1}\binom{\overline{\tau}}{1}^{s+1}$ and $\frac{s}{(-2iy)^{s-1}}\binom{\tau}{1}^{\infty-s+1}\binom{\overline{\tau}}{1}^{s-1}$ respectively, with the same respective description via multiplication by the vector $\frac{1}{-2iy}\binom{\overline{\tau}}{1}$ and differentiation with respect to it. They are simpler in the limit because we no longer have to worry about landing in the space with the same index $m$. Moreover, taking the direct limit of the commutative diagrams from Remarks \ref{comdiaiso} and \ref{comdiaim} (where the map from $\widetilde{\mathcal{M}}_{k}^{sing,\leq m}(\rho)$ into $\widetilde{\mathcal{M}}_{k}^{sing,\leq m}(\rho)$ are the evident inclusions---see Corollary \ref{depthim}) produces commutative diagrams like \[\begin{array}{ccc}\widetilde{\mathcal{M}}_{k}^{sing}(\rho) & \stackrel{\sim}{\longrightarrow} & \mathcal{M}_{k-\infty}^{sing}(V_{\infty}\otimes\rho) \\ \\ \delta_{l}\Big\downarrow & & \Big\downarrow\tilde{\delta}_{l} \\ \\ \widetilde{\mathcal{M}}_{k+2}^{sing}(\rho) & \stackrel{\sim}{\longrightarrow} & \mathcal{M}_{k+2-\infty}^{sing}(V_{\infty}\otimes\rho),\end{array}\] or the one having $f\mapsto\frac{f}{-2iy}$ on the arrow on the left and $\frac{\overline{i}}{-2iy}$ on the one on the right. Once again, when the arrow on the left is $y^{2}\partial_{\overline{\tau}}$ or $f \mapsto f_{1}$ (into $\widetilde{\mathcal{M}}_{k-2}^{sing}(\rho)$), the space on the lower right corner has to be replaced by $\mathcal{M}_{k+2-\infty}^{sing}(V_{\infty}\otimes\rho)$, and the arrow on the right will carry $y^{2}\partial_{\overline{\tau}}$ or $D$ respectively. \label{limops}
\end{rmk}

We recall from \cite{[Ze1]} that the modular forms with representations involving $V_{m}$ are endowed with a multiplicative structure, arising from the tensor product and the natural projection from $V_{m} \otimes V_{p}$ onto $V_{m+p}$. This multiplication corresponds to the usual (tensor) product of quasi-modular forms via Theorem \ref{QMFVVMF}. It also behaves well with respect to the embeddings $i_{m}$. On the other hand, the operators appearing in the Rankin--Cohen brackets in Theorem \ref{RCbrack} do not commute well with the inclusions $i_{m}$. Therefore the only assertion about Rankin--Cohen brackets for modular forms involving $V_{m}$ that we can make at this point is the following.
\begin{cor}
The combinations defining the Rankin--Cohen brackets in Theorem \ref{RCbrack} yield, when composed with the projection $V_{m} \otimes V_{p} \to V_{m+p}$, bilinear operators from $\mathcal{M}_{k-m}^{*}(V_{m}\otimes\rho)\times\mathcal{M}_{l-p}^{*}(V_{p}\otimes\eta)$ to $\mathcal{M}_{k+l+2n-m-p}^{*}(V_{m+p}\otimes\rho\otimes\eta)$ for every type $*$. The same assertion holds for the combinations from Remark \ref{RC2dim} for the cases excluded in that theorem. \label{RCShim}
\end{cor}

\begin{proof}
Part $(i)$ of Theorem \ref{difopShim} shows that the form of these compositions is the same on quasi-modular forms and on the spaces in question. The assertions now follow from the proof of Theorem \ref{RCbrack}, together with the considerations in Remark \ref{RC2dim} in the excluded cases.
\end{proof}
Note that unlike Corollary \ref{comim}, the Rankin--Cohen brackets from Corollary \ref{RCShim} do not go naturally over to $V_{m+p+1}$ (or to the limit space with $V_{\infty}$). This is not surprising, since the Rankin--Cohen brackets defined in Theorem \ref{RCbrack} depend on the depth of the quasi-modular forms on which they operate.

\subsection{$\mathfrak{sl}_{2}$-Triples}

Recall that the operators acting on some vector space form a Lie algebra with the commutator $[X,Y]=XY-YX$. The triples of operators in which we are interested are the following ones.
\begin{defn}
Three elements $H$, $E$, and $F$ of a Lie algebra are called an \emph{$\mathfrak{sl}_{2}$-triple} if they satisfy the commutation relations $[H,E]=2E$, $[H,F]=-2F$, and $[E,F]=H$. \label{sl2def}
\end{defn}
Operators satisfying the commutation relations from Definition \ref{sl2def} are (classically) called $\mathfrak{sl}_{2}$-triples since the three natural generators $H=\big(\begin{smallmatrix} 1 & 0 \\ 0 & -1\end{smallmatrix}\big)$,  $E=\big(\begin{smallmatrix} 0 & 1 \\ 0 & 0\end{smallmatrix}\big)$, and $F=\big(\begin{smallmatrix} 0 & 0 \\ 1 & 0\end{smallmatrix}\big)$ of the Lie algebra $\mathfrak{sl}_{2}(\mathbb{R})$ of real traceless $2\times2$ matrices satisfy them. Two reasons for the importance of $\mathfrak{sl}_{2}$-triples are that they generate the isomorphism type of the minimal non-Abelian simple Lie algebra, and that one can associate a natural Laplacian to each such triple.

An action of an $\mathfrak{sl}_{2}$-triple on a space decomposes the space according to the eigenvalue of $H$. The operators $E$ and $F$ send an element with a certain eigenvalue $\kappa$ to elements of eigenvalue $\kappa\pm2$. We are interested in $\mathfrak{sl}_{2}$-triples acting on (quasi-)modular forms in which $H$ is the operator $W$ multiplying every (quasi-)modular form by its weight. Then $E$ and $F$ must correspond to a weight raising and a weight lowering operator respectively. Indeed, it is shown in \cite{[V]} that when working with real-analytic functions on $\mathcal{H}$ or on $SL_{2}(\mathbb{R})$, the operator $W$ forms such an $\mathfrak{sl}_{2}$-triple with the operators $2i\delta_{k}$ and $-2iy^{2}\partial_{\overline{\tau}}$. As another example, Equation (8) of \cite{[A]} implies that $W$ and the (holomorphic) operators $\partial_{\tau}$ and $f \mapsto f_{1}$ become another such $\mathfrak{sl}_{2}$-triple, now operating on (holomorphic) quasi-modular forms, after inverting the sign of one of the two latter operators.

\smallskip

For finding $\mathfrak{sl}_{2}$-triples we shall first need to evaluate the commutation relations between our operators.
\begin{prop}
Take a modular form $F\in\mathcal{M}_{k-m}^{sing}(V_{m}\otimes\rho)$ for some $m$, associated by Theorem \ref{QMFVVMF} to the element $f\in\widetilde{\mathcal{M}}_{k}^{sing,\leq m}(\rho)$ with the functions $f_{r}$, $0 \leq r \leq m$, and let $F_{\infty}$ be the image of $F$ in $\mathcal{M}_{k-\infty}^{sing}(V_{\infty}\otimes\rho)$. Then the following equalities hold:
\begin{enumerate}[$(i)$]
\item We have the two equalities $y^{2}\partial_{\overline{\tau}}\big[(i_{m-1} \circ D_{m})F\big]=(i_{m-1} \circ D_{m})(y^{2}\partial_{\overline{\tau}}F)$ and $y^{2}\partial_{\overline{\tau}}(DF_{\infty})=D(y^{2}\partial_{\overline{\tau}}F_{\infty})$.
\item If $F \in i_{m-1}\big(\mathcal{M}_{k+1-m}^{sing}(V_{m-1}\otimes\rho)\big)$, then $\delta_{k+2-m}\big[\big(\frac{\overline{i}_{m-1}}{-2iy} \circ i_{m-1}^{-1}\big)F\big]$ coincides with $\big(\frac{\overline{i}_{m-1}}{-2iy} \circ i_{m-1}^{-1}\big)\big[\delta_{k-m}F-\big(\frac{\overline{i}_{m-1}}{-2iy} \circ i_{m-1}^{-1}\big)F\big]$. We also have the equality $\tilde{\delta}_{l+2}\big(\frac{\overline{i}}{-2iy}F_{\infty}\big)=\frac{\overline{i}}{-2iy}(\tilde{\delta}_{l}F_{\infty})$ for every $l$.
\item For $F$ as in part $(ii)$, the composition $(i_{m-1} \circ D_{m})\big[\big(\frac{\overline{i}_{m-1}}{-2iy} \circ i_{m-1}^{-1}\big)F\big]$ gives the same value as $\big(\frac{\overline{i}_{m-1}}{-2iy} \circ i_{m-1}^{-1}\big)\big((i_{m-1} \circ D_{m})F\big)+F$. Moreover, the equality $D\big(\frac{\overline{i}}{-2iy}F_{\infty}\big)=\frac{\overline{i}}{-2iy}(DF_{\infty})+F_{\infty}$ holds.
\item Back to a general $F$, the modular form $(i_{m-1} \circ D_{m})(\delta_{k-m}F)$ is the same as $\big(\delta_{k-2-m}-2\frac{\overline{i}_{m-1}}{-2iy} \circ i_{m-1}^{-1}\big)\big((i_{m-1} \circ D_{m})F\big)+mF$. When $F$ is an $i_{m-1}$-image we may add $(k-m-l)\frac{\overline{i}_{m-1}}{-2iy} \circ i_{m-1}^{-1}$ to both $\delta$-operators, and obtain an equality if the multiplier $m$ is replaced by $k-l$. In the limit $m\to\infty$ we obtain $D(\tilde{\delta}_{l}F_{\infty})=\tilde{\delta}_{l}(DF_{\infty})+(k-l)F_{\infty}$.
\item We also get $y^{2}\partial_{\overline{\tau}}\big[\big(\frac{\overline{i}_{m-1}}{-2iy} \circ i_{m-1}^{-1}\big)F\big]=\big(\frac{\overline{i}_{m-1}}{-2iy} \circ i_{m-1}^{-1}\big)(y^{2}\partial_{\overline{\tau}}F)+\frac{1}{4}F$ on $i_{m-1}$-images. In addition, we have $y^{2}\partial_{\overline{\tau}}\big(\frac{\overline{i}}{-2iy}F_{\infty}\big)=\frac{\overline{i}}{-2iy}(y^{2}\partial_{\overline{\tau}}F_{\infty})+\frac{1}{4}F_{\infty}$.
\item The equality $y^{2}\partial_{\overline{\tau}}(\delta_{k-m}F)=\delta_{k-2-m}(y^{2}\partial_{\overline{\tau}}F)-\frac{k-m}{4}F$ holds for every $F\in\mathcal{M}_{k-m}^{sing}(V_{m}\otimes\rho)$. Assuming that $F$ is again in the image of $i_{m-1}$, adding $(k-m-l)\frac{\overline{i}_{m}}{-2iy} \circ i_{m-1}^{-1}$ from each $\delta$ gives another valid equality if the constant multiplier of $F$ is replaced by $-\frac{l}{4}$. Finally, the equality $y^{2}\partial_{\overline{\tau}}(\tilde{\delta}_{l}F_{\infty})=\tilde{\delta}_{l-2}(y^{2}\partial_{\overline{\tau}}F_{\infty})-\frac{l}{4}F_{\infty}$ also holds.
\end{enumerate} \label{comrels}
\end{prop}

\begin{proof}
Parts $(iv)$ and $(v)$ of Theorem \ref{difopShim} imply that the two sides of the first equality in part $(i)$ here are associated with $y^{2}\partial_{\overline{\tau}}f_{1}$ and the function with index 1 arising from $y^{2}\partial_{\overline{\tau}}f$ in Equation \eqref{QMFdef}. As these quasi-modular forms coincide by part $(iv)$ of Proposition \ref{difopQMF}, this equality holds, and the second one follows by passing to the limit from Corollary \ref{comim}. This proves part $(i)$. Using parts $(i)$ and $(ii)$ of Theorem \ref{difopShim}, the first modular form from part $(ii)$ here corresponds via Theorem \ref{QMFVVMF} to the element $\delta_{k+2-m}\big(\frac{f}{-2iy}\big)$ of $\widetilde{\mathcal{M}}_{k+4}^{sing,\leq m}(\rho)$. As the latter expression equals $\frac{\delta_{k+1-m}f}{-2iy}$ by Equation \eqref{deltalyt}, the same parts of Theorem \ref{difopShim} establishes the first equality. Assuming further that $F$ is an $(i_{m-1} \circ i_{m-2})$-image, we may add $(k-m-l)\big(\frac{\overline{i}_{m}}{-2iy} \circ i_{m-1}^{-1}\big)^{2}F$ to both sides of the equation, and taking the limit $m\to\infty$ via Corollary \ref{comim} produces the second equality in part $(ii)$ here as well. The first modular form appearing in part $(iii)$ is attached, as parts $(ii)$ and $(v)$ of Theorem \ref{difopShim} show, to the function with index 1 associated with $\frac{f}{-2iy}$ in Equation \eqref{QMFdef}, which was seen in part $(ii)$ of Proposition \ref{difopQMF} to be $\frac{f_{1}}{-2iy}+f$. We go back using the same parts of Theorem \ref{difopShim} to obtain the first equality, which yields the second one in the limit from Corollary \ref{comim} once more. The desired part $(iii)$ is thus also established.

Using parts $(i)$ and $(v)$ of Theorem \ref{difopShim} we identify the quasi-modular form associated with the first modular form appearing in part $(iv)$ as the function with index 1 corresponding to $\delta_{k-m}f$, which is evaluated as $\delta_{k-m}f_{1}+mf$ in part $(iii)$ of Proposition \ref{difopQMF}. We can go back using part $(iii)$ of Theorem \ref{difopShim} (which is applicable since $i_{m-1}(D_{m}F)$ is an $i_{m-1}$-image), which establishes the first equality, and the assertion with adding a multiple of $\frac{\overline{i}_{m-1}}{-2iy} \circ i_{m-1}^{-1}$ is now a consequence of part $(iii)$ here. The usual limit process from Corollary \ref{comim} then implies the third equality, and part $(iv)$ follows. Next, on the left hand side of the first equality in part $(v)$ we obtain the modular form corresponding to $y^{2}\partial_{\overline{\tau}}\big(\frac{f}{-2iy}\big)$ by parts $(ii)$ and $(iv)$ of Theorem \ref{difopShim}. Evaluating this derivative as $\frac{y^{2}\partial_{\overline{\tau}}f}{-2iy}+\frac{f}{4}$ and using the inverse argument produces the first equality, as well as the second one via another application of Corollary \ref{comim}. This proves part $(v)$ as well. Finally, the first equality in part $(vi)$ can be obtained by a direct evaluation, since the difference between the functions $y^{2}\partial_{\overline{\tau}}(\delta_{k-m}F)$ and $\delta_{k-2-m}(y^{2}\partial_{\overline{\tau}}F)$, which equals $y^{2}\delta_{k-m}\partial_{\overline{\tau}}F$ by Equation \eqref{deltalyt}, is just $y^{2}F\cdot\partial_{\overline{\tau}}\frac{k-m}{2iy}$ (alternatively, one can obtain the same equality but with $F$ replaced by $f$ and apply parts $(i)$ and $(iv)$ of Theorem \ref{difopShim}). The second assertion, for $i_{m-1}$-images, thus follows from part $(v)$, and Corollary \ref{comim} implies the equality in the limit once again.
\end{proof}

\begin{rmk}
The results of Proposition \ref{comrels} are, in fact, commutation relations between operators on the spaces $\mathcal{M}_{k-m}^{sing}(V_{m}\otimes\rho)$ and $\mathcal{M}_{k-\infty}^{sing}(V_{\infty}\otimes\rho)$. Explicitly, the assertions involving $F_{\infty}$ in parts $(i)$, $(iii)$, and $(v)$ can be written as the commutators $[y^{2}\partial_{\overline{\tau}},D]=0$, $\big[D,\frac{\overline{i}}{-2iy}\big]=I$, and $\big[y^{2}\partial_{\overline{\tau}},\frac{\overline{i}}{-2iy}\big]=\frac{1}{4}I$ respectively (where $I$ stands for the identity operator). From parts $(ii)$, $(iv)$, and $(vi)$ we deduce the equalities $\tilde{\delta}_{l+2}\circ\frac{\overline{i}}{-2iy}=\frac{\overline{i}}{-2iy}\circ\tilde{\delta}_{l}$, $D\circ\tilde{\delta}_{l}-\tilde{\delta}_{l} \circ D=(k-l)I$, and $y^{2}\partial_{\overline{\tau}}\circ\tilde{\delta}_{l}-\tilde{\delta}_{l-2} \circ y^{2}\partial_{\overline{\tau}}=-\frac{l}{4}I$ respectively for every $l$. Similar equalities hold for operators on $\mathcal{M}_{k-m}^{sing}(V_{m}\otimes\rho)$, but now with the indices $m$ and the embeddings $i_{m}$ and their inverses wherever necessary. Some of these relations were also seen to be well-defined only on the subspace $i_{m-1}\big(\mathcal{M}_{k+1-m}^{sing}(V_{m-1}\otimes\rho)\big)$ of $\mathcal{M}_{k-m}^{sing}(V_{m}\otimes\rho)$. We shall not write the formulae for finite $m$, since their spirit is seen in the limit $m\to\infty$ already presented, and since we shall only use the commutators on the spaces $\mathcal{M}_{k-\infty}^{sing}(V_{\infty}\otimes\rho)$ in what follows. \label{commutators}
\end{rmk}
Note that one must be careful when writing the equalities involving the $\tilde{\delta}_{l}$'s in Remark \ref{commutators} as commutators, since the index $l$ varies throughout the relation. We also remark that parts $(i)$ and $(ii)$ of Proposition \ref{comrels} show that operators going in the same direction (i.e., both increase or both decrease the weight) commute, so that non-trivial commutators only occur between a weight raising operator and a weight lowering operator.

\smallskip

We are interested in $\mathfrak{sl}_{2}$-triples arising from $W$ and the various weight changing operators on the spaces $\mathcal{M}_{k-\infty}^{sing}(V_{\infty}\otimes\rho)$ (or their finite-dimensional counterparts). We remark that while the weight of elements in $\mathcal{M}_{k-m}^{sing}(V_{m}\otimes\rho)$ with finite $m$ is $k-m$, we let $W$ act on them (hence also on elements of the limit $\mathcal{M}_{k-\infty}^{sing}(V_{\infty}\otimes\rho)$) as multiplication by $k$ and not $k-m$. In this way the action of $W$ commutes with $i_{m}$ and the isomorphisms from Theorem \ref{QMFVVMF}, and multiplies quasi-modular forms by their true weight.

For finding such $\mathfrak{sl}_{2}$-triples we prove the following lemma.
\begin{lem}
Any simple operator $\delta$ on $\bigoplus_{k}\mathcal{M}_{k-\infty}^{sing}(V_{\infty}\otimes\rho)$ that raises the weight reduces to $d_{k}\tilde{\delta}_{0}-a_{k}\frac{\overline{i}}{-2iy}$ on a single constituent $\mathcal{M}_{k-\infty}^{sing}(V_{\infty}\otimes\rho)$, where $a_{k}$ and $d_{k}$ are complex constants. Similarly, the restriction of a simple operator $\overline{\delta}$ that lowers the weight by 2 to such a constituent is $4b_{k}y^{2}\partial_{\overline{\tau}}-c_{k}D$ with constants $b_{k}$ and $c_{k}$. Restricting the commutation relation $[\delta,\overline{\delta}]$ to $\mathcal{M}_{k-\infty}^{sing}(V_{\infty}\otimes\rho)$ yields the sum of the operators $4(b_{k}d_{k-2}-b_{k+2}d_{k})(\tilde{\delta}_{0} \circ y^{2}\partial_{\overline{\tau}})$, $(c_{k+2}d_{k}-c_{k}d_{k-2})(\tilde{\delta}_{0} \circ D)$, $4(a_{k}b_{k+2}-a_{k-2}b_{k}-2b_{k+2}d_{k})\big(\frac{\overline{i}}{-2iy} \circ y^{2}\partial_{\overline{\tau}}\big)$, $-(a_{k}c_{k+2}-a_{k-2}c_{k})\big(\frac{\overline{i}}{-2iy} \circ D\big)$, and $(a_{k}b_{k+2}-a_{k}c_{k+2}+kc_{k+2}d_{k})I$, where $I$ is the identity operator. \label{sl2trip}
\end{lem}
By a \emph{simple} operator we mean a linear combination of the operators appearing in part $(ii)$ of Corollary \ref{comim}. This excludes compositions such as $\tilde{\delta}_{0}D\tilde{\delta}_{0}$ for $\delta$ or $y^{2}\partial_{\overline{\tau}}D\frac{\overline{i}}{-2iy}$ for $\overline{\delta}$. We observe that in $D$, as well as in $\frac{\overline{i}}{-2iy}$, we had the choice of scalars in Lemma \ref{wcShimnodif} (we chose, for a finite space $V_{m}$, the multipliers $m_{-}$ for $D_{m}$ and just 1 for $\frac{\overline{i}}{-2iy}$, but many other choices exist). It is fairly easy to see that all the other choices can be obtained as linear combinations of operators that are typically not simple (for example, the weight raising operator sending $F(\tau)=\sum_{s=0}^{m}\frac{F_{s}(\tau)}{(-2iy)^{s}}\binom{\tau}{1}^{\infty-s}\binom{\overline{\tau}}{1}^{s}$ to $\sum_{s=0}^{m}\frac{sF_{s-1}(\tau)}{(-2iy)^{s}}\binom{\tau}{1}^{\infty-s}\binom{\overline{\tau}}{1}^{s}$ is $\frac{\overline{i}}{-2iy}D\frac{\overline{i}}{-2iy}$). We shall not consider such operators in what follows, since our choices of scalars for the simple operators was already explained in Theorem \ref{difopShim} (also motivated in Remark \ref{geom}). Also note that the operator $\delta$ from Lemma \ref{sl2trip} is just $d_{k}\tilde{\delta}_{a_{k}/d_{k}}$ if $d_{k}\neq0$, but it cannot be written in this way when $d_{k}=0$.

\begin{proof}
The first two assertions follow easily from the fact that $\tilde{\delta}_{0}$ and $\frac{\overline{i}}{-2iy}$ form a basis for the weight raising operators in part $(ii)$ of Corollary \ref{comim} (the more general operators $\tilde{\delta}_{l}$ are just linear combination of them), and the weight lowering operators there are generated freely by $y^{2}\partial_{\overline{\tau}}$ and $D$. Their commutator reduces, on a single space $\mathcal{M}_{k-\infty}^{sing}(V_{\infty}\otimes\rho)$, to the difference \[\big(d_{k-2}\tilde{\delta}_{0}-a_{k-2}\tfrac{\overline{i}}{-2iy}\big)\circ(4b_{k}y^{2}\partial_{\overline{\tau}}-c_{k}D)-
(4b_{k+2}y^{2}\partial_{\overline{\tau}}-c_{k+2}D)\circ\big(d_{k}\tilde{\delta}_{0}-a_{k}\tfrac{\overline{i}}{-2iy}\big),\] which decomposes as the sum of the desired terms $4(b_{k}d_{k-2}-b_{k+2}d_{k})(\tilde{\delta}_{0} \circ y^{2}\partial_{\overline{\tau}})$ and $(c_{k+2}d_{k}-c_{k}d_{k-2})(\tilde{\delta}_{0} \circ D)$, the expressions $a_{k}(4b_{k+2}y^{2}\partial_{\overline{\tau}}-c_{k+2}D)\circ\frac{\overline{i}}{-2iy}$ and $-a_{k-2}\frac{\overline{i}}{-2iy}\circ(4b_{k}y^{2}\partial_{\overline{\tau}}-c_{k}D)$, and the differences $4b_{k+2}d_{k}(\tilde{\delta}_{0} \circ y^{2}\partial_{\overline{\tau}}-y^{2}\partial_{\overline{\tau}}\circ\tilde{\delta}_{0})$ and $-c_{k+2}d_{k}(\tilde{\delta}_{0} \circ D-D\circ\tilde{\delta}_{0})$. To evaluate the latter two differences we apply parts $(vi)$ and $(iv)$ of Proposition \ref{comrels} (interpreted as in Remark \ref{commutators}), and together with the definition of the operators $\tilde{\delta}_{l}$ in part $(ii)$ of Corollary \ref{comim} they become $-8b_{k+2}d_{k}\frac{\overline{i}}{-2iy} \circ y^{2}\partial_{\overline{\tau}}$ and $+kc_{k+2}d_{k}I$ respectively. Next, we decompose the term with $a_{k}$ as the sum of $a_{k}\frac{\overline{i}}{-2iy}\circ(4b_{k+2}y^{2}\partial_{\overline{\tau}}-c_{k+2}D)$, $4a_{k}b_{k+2}\big[y^{2}\partial_{\overline{\tau}},\frac{\overline{i}}{-2iy}\big]$, and $-a_{k}c_{k+2}\big[D,\frac{\overline{i}}{-2iy}\big]$. Evaluating the two commutators using parts $(iii)$ and $(v)$ of Proposition \ref{comrels} (this is also easier in terms of Remark \ref{commutators}) then produces the desired result, after gathering all the similar terms.
\end{proof}

Our main result in this section essentially means that the references \cite{[V]} and \cite{[A]} already contain all the possible $\mathfrak{sl}_{2}$-triples appearing in our operators, as Remark \ref{normtrip} later explains.
\begin{thm}
Assume that $E$ and $F$, which we shall denote by $\delta$ and $\overline{\delta}$ respectively, are linear combinations of the operators appearing in part $(ii)$ of Corollary \ref{comim} forming an $\mathfrak{sl}_{2}$-triple on a space of the form $\bigoplus_{k\in\kappa+2\mathbb{Z}}\mathcal{M}_{k-\infty}^{sing}(V_{\infty}\otimes\rho)$ with $W$ playing the role of $H$. Then precisely one of the following situations can occur.
\begin{enumerate}[$(i)$]
\item $\delta\big|_{\mathcal{M}_{k-\infty}^{sing}(V_{\infty}\otimes\rho)}=\frac{1}{b_{k+2}}\tilde{\delta}_{k}$ and $\overline{\delta}\big|_{\mathcal{M}_{k-\infty}^{sing}(V_{\infty}\otimes\rho)}=4b_{k}y^{2}\partial_{\overline{\tau}}$ for each $k$ for some sequence $(b_{k})_{k\in\kappa+2\mathbb{Z}}$ of non-zero complex numbers.
\item There is a sequence $(c_{k})_{k\in\kappa+2\mathbb{Z}}$ in $\mathbb{C}^{\times}$ such that $\delta\big|_{\mathcal{M}_{k-\infty}^{sing}(V_{\infty}\otimes\rho)}=\frac{1}{c_{k+2}}\tilde{\delta}_{0}$ and $\overline{\delta}\big|_{\mathcal{M}_{k-\infty}^{sing}(V_{\infty}\otimes\rho)}=-c_{k}D$ for every weight $k\in\kappa+2\mathbb{Z}$.
\end{enumerate} \label{onlyVandA}
\end{thm}

\begin{proof}
For the two commutation relations with $W$ to be the required ones it is necessary and sufficient that $\delta$ would increase the weight by 2 and $\overline{\delta}$ would decrease it by 2. Since we assume that they are simple, they can be described as in Lemma \ref{sl2trip}. The forming of an $\mathfrak{sl}_{2}$-triple with $W$ is now equivalent to their commutator $[\delta,\overline{\delta}]$ being $W$. Explicitly, we require that the action of the commutator $[\delta,\overline{\delta}]$ multiplies elements of the space $\mathcal{M}_{k-\infty}^{sing}(V_{\infty}\otimes\rho)$ by $k$. But the action of that commutator on $\mathcal{M}_{k-\infty}^{sing}(V_{\infty}\otimes\rho)$ is evaluated in Lemma \ref{sl2trip}, so that we have to compare the expression appearing there to just $kI$.

Now, since we work with $k$ in a single coset of $2\mathbb{Z}$, the vanishing of the coefficients multiplying the compositions $\tilde{\delta}_{0} \circ y^{2}\partial_{\overline{\tau}}$, $\tilde{\delta}_{0} \circ D$, and $\frac{\overline{i}}{-2iy} \circ D$ imply that the products $b_{k+2}d_{k}$, $c_{k+2}d_{k}$, and $a_{k}c_{k+2}$ are independent of $k$. If we denote these constants by $\xi$, $\nu$, and $\omega$ respectively, then from the vanishing of the coefficient of $\frac{\overline{i}}{-2iy} \circ y^{2}\partial_{\overline{\tau}}$ we deduce that $a_{k}b_{k+2}-k\xi$ is also a constant, which we denote by $\eta$ (indeed, by writing the latter difference as $\eta_{k}$ the vanishing coefficient in question reduces to $\eta_{k+2}-\eta_{k}$). The coefficient of $I$, which we compare to $k$, therefore takes the form $\eta+k\xi-\omega+k\nu$. This implies the equalities $\eta=\omega$ and $\nu=1-\xi$. But the two decompositions of the product $a_{k}b_{k+2}c_{k+2}d_{k}$ yield $\nu(\eta+k\xi)=\xi\omega$, which after substituting $\omega$ for $\eta$ and $1-\xi$ for $\nu$ reduces to $\xi(1-\xi)k=(2\xi-1)\omega$ for all $k$. Therefore $\xi$ is either 0 or 1 (so that $\nu$ is 1 or 0 respectively), and as $2\xi-1\neq0$ we find that the common value of $\eta$ and $\omega$ is 0.

Substituting everything back, we remain with the equalities $a_{k}c_{k+2}=0$, $a_{k}b_{k+2}=k\xi$, $b_{k+2}d_{k}=\xi$, and $c_{k+2}d_{k}=1-\xi$ for every $k$, with $\xi\in\{0,1\}$. Moreover, comparing the last two expressions yields $d_{k}(c_{k+2}+b_{k+2})=1$, so that $d_{k}$ cannot vanish for any $k$ and the operator $\delta$ can be written as $d_{k}\tilde{\delta}_{a_{k}/d_{k}}$. Considering first the case with $\xi=1$, we deduce the vanishing of $c_{k+2}$ for every $k$ (in correspondence with the first equality), and from the remaining two equalities we find that $b_{k+2}\neq0$, $d_{k}=\frac{1}{b_{k+2}}$, and $a_{k}=\frac{k}{b_{k+2}}$ (so that $a_{k}/d_{k}$ is just $k$). This is the situation appearing in $(i)$. On the other hand, with $\xi=0$ we get $b_{k+2}=0$ for each $k$, and the non-vanishing of $c_{k+2}$ (with $d_{k}=\frac{1}{c_{k+2}}$) implies also $a_{k}=0$ (the latter two assertions agree with the required vanishing of $a_{k}b_{k+2}$). This is the situation from $(ii)$, and they clearly cannot happen simultaneously.
\end{proof}

\begin{rmk}
We remark that the requirement of the direct sum being on a single coset of $2\mathbb{Z}$ is important for having an irreducible $\mathfrak{sl}_{2}$-module. Otherwise we might have a direct sum of several such modules, and our $\mathfrak{sl}_{2}$-triple can involve $\tilde{\delta}_{k}$ and $4y^{2}\partial_{\overline{\tau}}$ on one irreducible component, and be based on $\tilde{\delta}_{0}$ and $D$ on another one. We also note that multiplying the action of $\delta$ by some constant on each $\mathcal{M}_{k-\infty}^{sing}(V_{\infty}\otimes\rho)$ while dividing $\overline{\delta}$ by the same constant on $\mathcal{M}_{k+2-\infty}^{sing}(V_{\infty}\otimes\rho)$ gives an equivalent $\mathfrak{sl}_{2}$-triple (this is just a question of normalization). Therefore the two families of $\mathfrak{sl}_{2}$-triple just amount to the different normalizations of the one appearing in \cite{[V]} and the one coming from \cite{[A]} respectively (in the first reference the chosen normalization is with $b_{k}=\frac{1}{2i}$ for every $k$, while in the second reference one just takes the constant sequence 1). \label{normtrip}
\end{rmk}

\section{Laplacians and Eigenfunctions \label{LapEigen}}

In this Section we present the two Laplacians on the space $\mathcal{M}_{k-\infty}^{sing}(V_{\infty}\otimes\rho)$ and determine their eigenspaces.

\subsection{Eigenfunctions of the Simple Laplacian and Lifts}

The center of the universal enveloping algebra of $\mathfrak{sl}_{2}$ is generated, as a polynomial ring, by the Casimir element $C=H^{2}+2EF+2FE$. Given a space on which $\mathfrak{sl}_{2}$ acts, the Casimir element acts as a central operator, and the Laplacian of the action is a suitable normalization of this operator (in the sense of addition of a polynomial in $H$ as well). We wish to normalize our Laplacians such that they will annihilate holomorphic (and meromorphic) modular forms of depth 0 (as the classical ones from the theory of modular forms do). We therefore write $C$ as $4EF+H^{2}-2H$ using the commutation relation from Definition \ref{sl2def}, and recall that $F=\overline{\delta}$ (in the notation of Lemma \ref{sl2trip} and Theorem \ref{onlyVandA}) annihilates the required functions. Since $C$ preserves the weights (by commuting with $H$), and $H=W$ operates as multiplication by a scalar on the space of modular or quasi-modular forms of a fixed weight, we shall define the Laplacian operator to be the one corresponding to the action of $\frac{C-W^{2}+2W}{4}$, which amounts to $\delta\circ\overline{\delta}$. The resulting operators are as follows.
\begin{defn}
The Laplacian on $\mathcal{M}_{k-\infty}^{sing}(V_{\infty}\otimes\rho)$ that arises from the $\mathfrak{sl}_{2}$-triples from Theorem \ref{onlyVandA} involving multiples of $\tilde{\delta}_{k}$ and of $y^{2}\partial_{\overline{\tau}}$ will be denoted by $\Delta_{k-\infty}^{V}$. The $\mathfrak{sl}_{2}$-triples of the second type, with multiples of $\tilde{\delta}_{0}$ and of $D$, give rise to the Laplacian on $\mathcal{M}_{k-\infty}^{sing}(V_{\infty}\otimes\rho)$ that will be denoted by $\Delta_{k-\infty}^{A}$. We denote by $\Delta_{k}^{V}$ and $\Delta_{k}^{A}$ the corresponding Laplacians on the space $\widetilde{\mathcal{M}}_{k}^{sing}(\rho)$ obtained by transferring these operators via Theorems \ref{QMFVVMF} and \ref{difopShim} and the commutative diagrams from Remark \ref{limops}. They take a quasi-modular form $f\in\widetilde{\mathcal{M}}_{k}^{sing}(\rho)$ to $4\delta_{k-2}y^{2}\partial_{\overline{\tau}}f$ (or equivalently $4y^{2}\delta_{k}\partial_{\overline{\tau}}f$ by Equation \eqref{deltalyt}) and $-\partial_{\tau}f_{1}$ respectively. The classical Laplacian $4\delta_{k-2}y^{2}\partial_{\overline{\tau}}=4y^{2}\delta_{k}\partial_{\overline{\tau}}$ (equality by Equation \eqref{deltalyt} again) on $\mathcal{M}_{k}^{sing}(\rho)$ will be denoted simply by $\Delta_{k}$. \label{Lapdef}
\end{defn}
The reason for the notation in Definition \ref{Lapdef} is that these $\mathfrak{sl}_{2}$-triples arise from the references \cite{[V]} and \cite{[A]} respectively. These operators are also easily seen to be independent of the choice of the sequences $(b_{k})_{k}$ and $(c_{k})_{k}$ (because these coefficients cancel in the product), as already suggested in Remark \ref{normtrip}. We shall thus henceforth assume, in all explicit formulae, that these coefficients are always 1.
\begin{rmk}
As for the spaces $\mathcal{M}_{k-m}^{sing}(V_{m}\otimes\rho)$ with finite $m$, the fact that the operator $\overline{\delta}=-D$ appearing in $\Delta_{k}^{A}$ is the limit image of $i_{m-1} \circ D_{m}$ mapping to $i_{m-1}$-images, this Laplacian is defined on the full space $\mathcal{M}_{k-m}^{sing}(V_{m}\otimes\rho)$ for each such $m$. This property is not shared by $\Delta_{k}^{V}$, which in general is defined only on the subspace of $\mathcal{M}_{k-m}^{sing}(V_{m}\otimes\rho)$ consisting of $i_{m-1}$-images (see Remark \ref{LapVm} below for more details). On the other hand, the operator $\Delta_{k-\infty}^{V}$ is well-defined on the full space $\mathcal{M}_{k-m}^{sing}(V_{m}\otimes\rho)$ in case $m=0$, where it reduces to the usual modular Laplacian $\Delta_{k}$ (see also Remark \ref{presVmLap} below). Note that $\Delta_{k-\infty}^{A}$ is just 0 on that space (since $D$ is). \label{LaponVm}
\end{rmk}

\smallskip

We begin our analysis of eigenfunctions by evaluating the action of the operators $\Delta_{k-\infty}^{V}$ and $\Delta_{k-\infty}^{A}$ from Definition \ref{Lapdef} on elements of $\mathcal{M}_{k-\infty}^{sing}(V_{\infty}\otimes\rho)$. Recall that we consider only elements that are the limit images of elements $\mathcal{M}_{k-m}^{sing}(V_{m}\otimes\rho)$ for some $m$ (so that using the bases appearing in Remark \ref{limops} only finitely many non-zero coefficients appear), and we present the former modular forms as in Theorem \ref{QMFVVMF} with the functions $F_{s}$. In fact, some of the calculations might be shortened if we had used the presentation with the quasi-modular forms $f_{r}$'s there, but as the functions $F_{s}$ are modular we prefer to apply the better known theory of modular (rather than quasi-modular) eigenfunctions.
\begin{lem}
Let an element $F\in\mathcal{M}_{k-\infty}^{sing}(V_{\infty}\otimes\rho)$ be given, and assume that it sends $\tau$ to
$\sum_{s=0}^{d}\frac{F_{s}(\tau)}{(-2iy)^{s}}\binom{\tau}{1}^{\infty-s}\binom{\overline{\tau}}{1}^{s}$ for some depth $d$. Then the images of $F$ under the operators $\Delta_{k-\infty}^{V}$ and $\Delta_{k-\infty}^{A}$ are the functions sending $\tau$ to $\sum_{s=0}^{d+1}\frac{F_{s,\Delta}^{V}(\tau)}{(-2iy)^{s}}\binom{\tau}{1}^{\infty-s}\binom{\overline{\tau}}{1}^{s}$ and $\sum_{s=0}^{d}\frac{F_{s,\Delta}^{A}(\tau)}{(-2iy)^{s}}\binom{\tau}{1}^{\infty-s}\binom{\overline{\tau}}{1}^{s}$ respectively, where
\[F_{s,\Delta}^{V}=\Delta_{k-2s}F_{s}+(s+1)\delta_{k-2-2s}F_{s+1}+(1-s)\big[sF_{s}+4y^{2}\partial_{\overline{\tau}}F_{s-1}\big]\] and
\[F_{s,\Delta}^{A}=-(s+1)\delta_{k-2-2s}F_{s+1}-s(k-1-s)F_{s},\] and in the evaluation of $F_{0,\Delta}^{V}$, $F_{d,\Delta}^{V}$, $F_{d+1,\Delta}^{V}$, and $F_{d,\Delta}^{A}$ we substitute just 0 for $F_{-1}$, $F_{d+1}$, or $F_{d+2}$. \label{Lapeval}
\end{lem}

\begin{proof}
Take some $m>d$, and consider our element $F\in\mathcal{M}_{k-\infty}^{sing}(V_{\infty}\otimes\rho)$ as the image of the modular form $F^{(m)}:\tau\to\sum_{s=0}^{d}\frac{F_{s}(\tau)}{(-2iy)^{s}}\binom{\tau}{1}^{m-s}\binom{\overline{\tau}}{1}^{s}$ from $\mathcal{M}_{k-m}^{sing}(V_{m}\otimes\rho)$ under the limit map. As $F^{(m)}$ is an $i_{m-1}$-image (since $m>d$), we can apply all the operators from Theorem \ref{difopShim} to it. Corollary \ref{comim} shows that for $\overline{\delta}$ we can let either $by^{2}\partial_{\overline{\tau}}$ or $-(i_{m-1} \circ D_{m})$ operate on $F^{(m)}$, and applying the maps from Theorem \ref{QMFVVMF} together with the formulae for transferring operators in Theorem \ref{difopShim} and Proposition \ref{difopseq}, we find that the resulting elements of $\mathcal{M}_{k-2-m}^{sing}(V_{m}\otimes\rho)$ take $\tau$ to
$\sum_{s=0}^{d}\frac{4y^{2}\partial_{\overline{\tau}}F_{s}(\tau)+(s+1)F_{s+1}(\tau)}{(-2iy)^{s}}\binom{\tau}{1}^{m-s}\binom{\overline{\tau}}{1}^{s}$ and
$\sum_{s=0}^{d}\frac{-(s+1)F_{s+1}(\tau)}{(-2iy)^{s}}\binom{\tau}{1}^{m-s}\binom{\overline{\tau}}{1}^{s}$ respectively (this can also be proved via a direct evaluation). Writing this element of $\mathcal{M}_{k-2-m}^{sing}(V_{m}\otimes\rho)$ as the one sending $\tau$ to $\sum_{s=0}^{d}\frac{G_{s}(\tau)}{(-2iy)^{s}}\binom{\tau}{1}^{m-s}\binom{\overline{\tau}}{1}^{s}$ for the moment, we have to find its image under the operator corresponding to $\delta$, which is either $\tilde{\delta}_{k-2}$ or $\tilde{\delta}_{0}$. Another application of Corollary \ref{comim} shows that this operator is $\delta_{k-2-m}-m\big(\frac{\overline{i}_{m+1}}{-2iy} \circ i_{m}^{-1}\big)$ in the first case and $\delta_{k-2-m}+(k-2-m)\big(\frac{\overline{i}_{m+1}}{-2iy} \circ i_{m}^{-1}\big)$ in the second case. A similar argument using Theorems \ref{QMFVVMF} and \ref{difopShim} and Proposition \ref{difopseq} (or another direct evaluation) produces the function sending $\tau$ to $\sum_{s=0}^{d}\frac{\delta_{k-2-2s}G_{s}(\tau)+(1-s)G_{s-1}(\tau)}{(-2iy)^{s}}\binom{\tau}{1}^{m-s}\binom{\overline{\tau}}{1}^{s}$ and $\sum_{s=0}^{d}\frac{\delta_{k-2-2s}G_{s}(\tau)+(k-1-s)G_{s-1}(\tau)}{(-2iy)^{s}}\binom{\tau}{1}^{m-s}\binom{\overline{\tau}}{1}^{s}$  as its respective $\delta$-image. Plugging in the respective expressions defining the functions $G_{s}$, observing the formula for $\Delta_{k-2s}$, and taking back the limit map from $\mathcal{M}_{k-m}^{sing}(V_{m}\otimes\rho)$ to $\mathcal{M}_{k-\infty}^{sing}(V_{\infty}\otimes\rho)$, we obtain the desired expressions for $F_{s,\Delta}^{V}$ and $F_{s,\Delta}^{A}$.
\end{proof}

\begin{rmk}
The fact that $F_{s,\Delta}^{A}$ is defined in Lemma \ref{Lapeval} only for $0 \leq s \leq d$ corresponds to $\Delta_{k-\infty}^{A}$ being well-defined on the spaces $\mathcal{M}_{k-m}^{sing}(V_{m}\otimes\rho)$ for every $m$ (see Remark \ref{LaponVm}). Turning now to $\Delta_{k-\infty}^{V}F$ in case $d=0$, we find that the function $F_{1,\Delta}^{V}$ from Lemma \ref{Lapeval} vanishes as well in this case (because of the coefficient $1-s$). This shows why $\Delta_{k-\infty}^{V}$ is defined on the space $\mathcal{M}_{k}^{sing}(\rho)$ as well, and its coincidence with $\Delta_{k}$ is also evident from the form of $F_{0,\Delta}^{V}$ in this case. \label{presVmLap}
\end{rmk}

\smallskip

We adopt henceforth the usual convention from the spectral theory of modular forms, in which the eigenvalues are with respect to \emph{minus} the Laplacian operator. In addition, we shall see that the exact depth of eigenfunctions, as elements of $\mathcal{M}_{k-\infty}^{sing}(V_{\infty}\otimes\rho)$, will have a strong effect on the possible eigenvalues and on the form of the eigenspaces. We therefore make the following definition.
\begin{defn}
Denote the space of solutions $F\in\mathcal{M}_{k-\infty}^{sing}(V_{\infty}\otimes\rho)$ of the equation $\Delta_{k-\infty}^{V}F=-\lambda F$ by $\mathcal{M}_{k-\infty,\lambda}^{sing,V}(V_{\infty}\otimes\rho)$. The space of functions $F$ satisfying $\Delta_{k-\infty}^{A}F=-\lambda F$ will similarly be denoted by $\mathcal{M}_{k-\infty,\lambda}^{sing,A}(V_{\infty}\otimes\rho)$. We also define $\mathcal{M}_{k-\infty,\lambda}^{sing,V,d}(V_{\infty}\otimes\rho)$ (resp. $\mathcal{M}_{k-\infty,\lambda}^{sing,A,d}(V_{\infty}\otimes\rho)$) to be the set of elements of $\mathcal{M}_{k-\infty,\lambda}^{sing,V}(V_{\infty}\otimes\rho)$ (resp. $\mathcal{M}_{k-\infty,\lambda}^{sing,A}(V_{\infty}\otimes\rho)$) whose depth is precisely $d$. The classical space of solutions $f\in\mathcal{M}_{k}^{sing}(\rho)$ to the equation $\Delta_{k}f=-\lambda f$ will be denoted by $\mathcal{M}_{k,\lambda}^{sing}(\rho)$. \label{defeigen}
\end{defn}
Note that the sets $\mathcal{M}_{k-\infty,\lambda}^{sing,V,d}(V_{\infty}\otimes\rho)$ and $\mathcal{M}_{k-\infty,\lambda}^{sing,A,d}(V_{\infty}\otimes\rho)$ from Definition \ref{defeigen} are not vector subspaces, since the difference of modular forms of depth $d$ need not have exact depth $d$ anymore.

Lemma \ref{Lapeval} has the following immediate consequence.
\begin{cor}
An element $F\in\mathcal{M}_{k-\infty}^{sing}(V_{\infty}\otimes\rho)$, written as in Lemma \ref{Lapeval}, lies in $\mathcal{M}_{k-\infty,\lambda}^{sing,V}(V_{\infty}\otimes\rho)$ (resp. $\mathcal{M}_{k-\infty,\lambda}^{sing,A}(V_{\infty}\otimes\rho)$) if and only if the equality $F_{s,\Delta}^{V}=-\lambda F_{s}$ (resp. $F_{s,\Delta}^{A}=-\lambda F_{s}$) holds for every $s$. Writing this explicitly under the assumption that $F$ has depth at most $d$, this statement with superscript $V$ amounts to the equality
\[\Delta_{k-2s}F_{s}+(s+1)\delta_{k-2-2s}F_{s+1}+(1-s)\big[sF_{s}+4y^{2}\partial_{\overline{\tau}}F_{s-1}\big]=-\lambda F_{s}\] holding for every $1 \leq s \leq d-1$, together with the equalities
\[\Delta_{k-2d}F_{d}+(1-d)\big[dF_{d}+4y^{2}\partial_{\overline{\tau}}F_{d-1}\big]=-\lambda F_{d},\] $-4dy^{2}\partial_{\overline{\tau}}F_{d}=0$ (from $s=d+1$), and
$\Delta_{k}F_{0}+\delta_{k-2}F_{1}=-\lambda F_{0}$. With the superscript $A$ the required equalities are
\[(s+1)\delta_{k-2-2s}F_{s+1}+s(k-1-s)F_{s}=\lambda F_{s}\] for every $0 \leq s \leq d-1$ as well as $d(k-1-d)F_{d}=\lambda F_{d}$. \label{compeigen}
\end{cor}

\begin{proof}
This follows immediately when we compare the coefficients of the two sides of the equality $\Delta_{k-\infty}^{V}F=-\lambda F$ or $\Delta_{k-\infty}^{A}F=-\lambda F$ and use the description and the explicit formulae from Lemma \ref{Lapeval} (after a global sign inversion when the superscript is $A$). The equations with indices 0, $d$, and $d+1$ for $V$, or just $d$ for $A$, are the equalities resulting from the vanishing of $F_{-1}$, $F_{d+1}$, and $F_{d+2}$.
\end{proof}

\begin{rmk}
We have already seen that the formulae for generic $s$ in Corollary \ref{compeigen} are valid equally well for the end cases with $s\in\{0,d,d+1\}$. Hence when we use inductive constructions below we do not have to separate the end cases from the rest. Note that the natural extension of these equalities to $s=-1$ seem to involve the (typically non-vanishing) function $F_{0}$, but the total equation does reduce to $0=0$ because of the coefficient $s+1$ appearing in front of $F_{s+1}$. The validity for $s<-1$ or for $s>d+1$ is in any case immediate. Therefore we can treat, in any inductive construction, the equalities from Corollary \ref{compeigen} as valid for all $s$ without worrying about particular values. \label{relsalls}
\end{rmk}

\smallskip

As the formula for $F_{s,\Delta}^{A}$ in Lemma \ref{Lapeval} and Corollary \ref{compeigen} involves fewer derivatives and fewer terms, the analysis of the latter spaces should be much simpler. Indeed, the case of generic weight $k$ is established rather easily as follows.
\begin{thm}
Take a Fuchsian group $\Gamma$, a weight $k$, a representation (or a multiplier system of weight $k$) $\rho$ of $\Gamma$, and a depth $d$. Then $\mathcal{M}_{k-\infty,\lambda}^{sing,A,d}(V_{\infty}\otimes\rho)$ is non-zero only for $\lambda=d(k-1-d)$. If we further assume that $k$ is not an integer between $d+1$ and $2d$ (this is an empty condition if $d=0$), then for any element $\varphi\in\mathcal{M}_{k-2d}^{sing}(\rho)$ there exists a unique element of $\mathcal{M}_{k-\infty,d(k-1-d)}^{sing,A,d}(V_{\infty}\otimes\rho)$ with $F_{d}=\varphi$, and its component $F_{s}$ is $\binom{d}{s}\delta_{k-2d}^{d-s}\varphi\Big/\prod_{q=0}^{d-s-1}(k-2d+q)$ for every $0 \leq s \leq d$. \label{eigenAgen}
\end{thm}

\begin{proof}
The equality associated with $s=d$ and superscript $A$ in Corollary \ref{compeigen} can clearly hold for non-zero $F_{d}$ if (and only if) the two coefficients coincide. This determines the value of $\lambda$. By setting $F_{d}=\varphi$, we need to show that for $F$ to be in $\mathcal{M}_{k-\infty,d(k-1-d)}^{sing,A,d}(V_{\infty}\otimes\rho)$ it is necessary and sufficient that $F_{s}$ is as stated. This statement amounts to the assertion that all the functions $F_{s}$ take the asserted values if and only if all the equations from Corollary \ref{compeigen} are satisfied. Now, we have seen that the choice of the value of $\lambda$ is equivalent to the validity of the equation for $s=d$, and with this value the equation associated with any $0 \leq s<d$ becomes $(s+1)\delta_{k-2-2s}F_{s+1}=(d-s)(k-1-d-s)F_{s}$ (the coefficient of $F_{s}$ equals $\lambda-s(k-1-s)$ with our value of $\lambda$). Our assumption on $k$ implies that this coefficient of $F_{s}$ does not vanish for any $s$, so that this equation determines $F_{s}$ in terms of $F_{s+1}$. It follows that an element of $\mathcal{M}_{k-\infty,d(k-1-d)}^{sing,A,d}(V_{\infty}\otimes\rho)$ with $F_{d}=\varphi$ exists and is unique. In more detail, the coefficient in front of $\delta_{k-2d}^{d-s}\varphi$ (recall that this operator is the composition $\delta_{k-2-2s}\circ\ldots\circ\delta_{k-2d}$) equals 1 for $s=d$, and the coefficient associated with $s$ is obtained from that for $s+1$ via multiplication by $\frac{s+1}{(d-s)(k-1-d-s)}$. Therefore the assumption that $F$ is an eigenfunction with $F_{d}=\varphi$ determines the functions $F_{s}$ to be the asserted ones (the extra term $k-1-d-s$ in the denominator of $F_{s}$ in comparison to that of $F_{s+1}$ is $k-2d+q$ with the value $q=d-s-1$).
\end{proof}
Note that the form of $F_{s}$ in Theorem \ref{eigenAgen} extends to give the correct value 0 if $s<0$ or if $s>d$ (because of the binomial coefficient), in correspondence with Remark \ref{relsalls}.

The form of eigenvalues determined in Theorem \ref{eigenAgen} motivates the following definition, which will make our presentation of the eigenvalues of $\Delta_{k-\infty}^{V}$ simpler as well.
\begin{defn}
Let $\varphi\in\mathcal{M}_{k-2d}^{sing}(\rho)$ be given. We define a \emph{lift of depth $d$ and eigenvalue $\lambda$ of $\varphi$} to be a function $F\in\mathcal{M}_{k-\infty,\lambda}^{sing,V}(V_{\infty}\otimes\rho)$ or in $\mathcal{M}_{k-\infty,\lambda}^{sing,A}(V_{\infty}\otimes\rho)$ in whose expansion as in Lemma \ref{Lapeval} we have $F_{d}=\varphi$, and such that $F_{s}$ is a constant multiple of $\delta_{k-2-2s}F_{s+1}$ for every $0 \leq s<d$. In case the superscript is $A$ we denote the lift of $\varphi$ by $\mathcal{L}_{k,d}^{A}(\varphi)$. For the superscript $V$ and eigenvalue $\lambda$ we shall write the lift as $\mathcal{L}_{k,d}^{V,\lambda}(\varphi)$, or with the superscript $\lambda$ replaced by another index indicating the value of $\lambda$. \label{liftdef}
\end{defn}

The notation $\mathcal{L}_{k,d}^{A}$ in Definition \ref{liftdef} does not contain $\lambda$ since Theorem \ref{eigenAgen} proved that the eigenvalue is determined by $k$ and $d$ in this case. We remark that the multipliers differentiating $F_{s}$ from $\delta_{k-2-2s}F_{s+1}$ in that definition will depend on the indices $k$, $d$, and $\lambda$, and will be determined by them (we have already seen this for $\mathcal{L}_{k,d}^{A}$ in most cases in Theorem \ref{eigenAgen}, and the remaining cases will be completed in Proposition \ref{eigenAsp} below). This would also imply that any map $\mathcal{L}_{k,d}^{A}$ and $\mathcal{L}_{k,d}^{V,\lambda}$ will be a linear map from the subspace of $\mathcal{M}_{k-2d}^{sing}(\rho)$ on which it is defined into $\mathcal{M}_{k-\infty,d(k-1-d)}^{sing,A}(V_{\infty}\otimes\rho)$ or
$\mathcal{M}_{k-\infty,\lambda}^{sing,V}(V_{\infty}\otimes\rho)$. This subspace was the full space $\mathcal{M}_{k-2d}^{sing}(\rho)$ in Theorem \ref{eigenAgen}, but this is not always the case (Proposition \ref{eigenAsp} already gives the first example when this does not happen). These linear maps are clearly injective, since the modular form $\varphi$ can always be reproduced as the coefficient $F_{d}$ of its lift $\mathcal{L}_{k,d}^{A}(\varphi)$ and $\mathcal{L}_{k,d}^{V,\lambda}(\varphi)$.

\smallskip

For integral weight $d+1 \leq k\leq2d$ the proof of Theorem \ref{eigenAgen} does not work, because some of the coefficients do vanish. Interestingly, these are precisely the weight where the analysis of \cite{[Ze1]} becomes more complicated (for example, the dimension formulae for $\mathcal{M}_{k-d}^{hol}(V_{d}\otimes\rho)$ depend there on whether $\Gamma$ has cusps or not). The result in this case is based on nearly meromorphic modular forms. If we consider an element $h\in\mathcal{M}_{l}^{nm}(\eta)$ for some representation $\eta$ as a polynomial in $\frac{1}{2iy}$ over meromorphic functions, then we write $h^{(r)}$ for its $r$th derivative as such a polynomial. It lies in $\mathcal{M}_{l-2r}^{nm}(\eta)$, since one easily verifies that each derivative corresponds to an application of the weight lowering operator $-4y^{2}\partial_{\overline{\tau}}$. Recall that Bol's identity mentioned above reduces, for every $l\in\mathbb{N}$, the operator $\delta_{-l}^{l+1}$ to just $\partial_{\tau}^{l+1}$. We now prove the following lemma.
\begin{lem}
Fix a non-negative integer $l$ and a representation $\rho$ of a Fuchsian group $\Gamma$. Then an element $\varphi\in\mathcal{M}_{-l}^{sing}(\rho)$ is annihilated by $\delta_{-l}^{l+1}=\partial_{\tau}^{l+1}$ if and only if it is of the form $(2iy)^{l}\overline{h}$ for an element $h\in\mathcal{M}_{l}^{nm,\leq l}(\overline{\rho})$. Moreover, the modular form $\delta_{-l}^{p}\varphi\in\mathcal{M}_{2p-l}^{sing}(\rho)$ equals, in this case, $(2iy)^{l-p}\overline{h^{(r)}}$ for every $p\in\mathbb{N}$. \label{nearmeroBol}
\end{lem}

\begin{proof}
We first prove by induction on $l$ that if a function $\xi$, with singularities as described in Definition \ref{MFQMFdef}, is annihilated by $\partial_{\tau}^{l+1}$, then it is of the form $\sum_{j=0}^{l}\frac{(2iy)^{j}}{j!}\overline{\omega_{l-j}(\tau)}$ with $\omega_{j}$, $0 \leq j \leq l$ being meromorphic functions. This statement is trivially true for $l=-1$ (since $\partial_{\tau}^{0}=Id$ annihilates nothing, and the asserted sum is empty), so we assume that the statement holds for $l-1$ and that $\xi$ satisfies $\partial_{\tau}^{l}\xi=0$. Then $\partial_{\tau}\xi$ is annihilated by $\partial_{\tau}^{l-1}$, so that we can write it as $\sum_{r=0}^{l-1}\frac{(2iy)^{r}}{r!}\overline{\omega_{l-1-r}(\tau)}$. But this is also the $\partial_{\tau}$-image of $\sum_{j=1}^{l}\frac{(2iy)^{j}}{j!}\overline{\omega_{l-j}(\tau)}$ (with $j=r+1$). Hence the difference between $\xi$ and that function is annihilated by $\partial_{\tau}$, and is thus of the form $\tau\mapsto\overline{\omega_{l}(\tau)}$ for some meromorphic $\omega_{l}$ (because of the restriction on the singularities).

In particular our modular form $\varphi$ is of this form. The behavior of weights of modular forms with respect to complex conjugation and powers of $2iy$ now shows that we can write $\varphi$ as $(2iy)^{l}\overline{h}$ for
$h\in\mathcal{M}_{l}^{sing}(\overline{\rho})$, and the form of $\varphi$ we just proved combines with the power of $2iy$ to show that $h$ must be in $\mathcal{M}_{l}^{nm,\leq l}(\overline{\rho})$. This proves the first assertion. We claim that if $\psi=(2iy)^{k}\overline{g}\in\mathcal{M}_{-k}^{sing}(\rho)$ with $g\in\mathcal{M}_{k}^{nm}(\overline{\rho})$ (no depth bound) for some weight, then $\delta_{-k}\psi=(2iy)^{k-2}\overline{g^{(1)}}$. Indeed, the relation between powers of $2iy$ and the operators $\delta_{k}$ shows that $\delta_{-k}\psi=(2iy)^{k}\partial_{\tau}\overline{g}$, and if we write $g(\tau)$ as $\sum_{j}\frac{\omega_{j}(\tau)}{(2iy)^{j}}$ and conjugate, then the operation of $\partial_{\tau}$ on the $j$th resulting term simply produces $j\frac{\overline{\omega_{j}(\tau)}}{(-2iy)^{j+1}}$. As this is the $j$th term of $\overline{g^{(1)}}\big/(2iy)^{2}$, multiplying by $(2iy)^{k}$ proves the required equality. The second assertion now follows via a simple induction.
\end{proof}

Following the proof of Theorem \ref{eigenAgen} we now get the following result.
\begin{prop}
For $\Gamma$, $k$, $\rho$, and $d$ as in Theorem \ref{eigenAgen}, assume now that $k$ is an integer $d+1+e$ with $0 \leq e<d$ (so that $\lambda$ is just $de$). Then the component $F_{d}=\varphi\in\mathcal{M}_{k-2d}^{sing}(\rho)$ of an element of $\mathcal{M}_{k-\infty,de}^{sing,A,d}(V_{\infty}\otimes\rho)$ is $(2iy)^{d-e-1}$ times the complex conjugate of an element  $h\in\mathcal{M}_{d-e-1}^{nm,\leq d-e-1}(\overline{\rho})$. Moreover, $F_{s}$ is $(-1)^{d-s}\binom{d}{s}\frac{(s-e-1)!}{(d-e-1)!}(2iy)^{s-e-1}\overline{h^{(d-s)}}$ for every $e<s \leq d$. In addition, the next function $F_{e}$ can be an arbitrary element $\psi\in\mathcal{M}_{k-2e}^{sing}(\rho)$, and the remaining coefficients $F_{s}$ with $0 \leq s \leq e$ are precisely those of $\mathcal{L}_{k,e}^{A}(\psi)$. \label{eigenAsp}
\end{prop}

\begin{proof}
We have seen that the equality
$(s+1)\delta_{k-2-2s}F_{s+1}=(d-s)(k-1-d-s)F_{s}$ must hold for every $0 \leq s<d$ in the proof of Theorem \ref{eigenAgen}, and the multiplier in the right hand side is $(d-s)(e-s)$ with our value of $k$. This determines $F_{s}$ with $s>e$ as $(-1)^{d-s}\binom{d}{s}\frac{(s-e-1)!}{(d-e-1)!}\delta_{k-2d}^{d-s}\varphi$ (since the coefficient multiplying $F_{s}$ here still does not vanish), where the product $\prod_{q=0}^{d-s-1}(k-2d+q)=\prod_{q=0}^{d-s-1}(e+1+q-d)$ from the denominator from Theorem \ref{eigenAgen} is $(-1)^{d-s}\frac{(d-e-1)!}{(s-e-1)!}$. Taking now $s=e$, the coefficient in front of $F_{e}$ vanishes, implying the vanishing of $\delta_{k-2-2e}F_{e+1}$, namely of $\delta_{k-2d}^{d-e}\varphi=\delta_{e+1-d}^{d-e}\varphi$ (recall the value of $k$). But this operator $\delta_{e+1-d}^{d-e}$ is, by Bol's identity mentioned above, just the $(d-e)$th power $\partial_{\tau}^{d-e}$. Therefore Lemma \ref{nearmeroBol} (with $l=d-e-1$) shows that $\varphi$ is of the desired form, and so are the functions $F_{s}$ for $e<s \leq d$. Now, the next function $F_{e}$ does not appear in the equation with $s=e$ (as we just saw), and the value $de$ of $\lambda$ equals $e(k-1-e)$ with our value of $k$. This puts us back in the situation from the proof of Theorem \ref{eigenAgen}, so that the statements about the functions $F_{s}$ with $0 \leq s \leq e$ follow from that theorem.
\end{proof}

It is clear that the choice $F_{e}=0$ satisfies the condition from Definition \ref{liftdef} for the resulting element of $\mathcal{M}_{k-\infty,de}^{sing,A,d}(V_{\infty}\otimes\rho)$ to be a lift of $\varphi$. Recalling that the lift $\mathcal{L}_{k,d}^{A}$ was defined for generic $k$ using the formula from Theorem \ref{eigenAgen}, Proposition \ref{eigenAsp} allows us to use the same formula for defining $\mathcal{L}_{k,d}^{A}(\varphi)$ also when $k$ is one of the special weights, but only if the lifted modular form $\varphi$ comes from $(2iy)^{2d-k}\overline{\mathcal{M}_{2d-k}^{nm,\leq 2d-k}(\overline{\rho})}$ (note the relation between $k$ and the parameter $e$ from that proposition).

In total, the eigenspaces of $\Delta_{k-\infty}^{A}$ can be summarized as follows.
\begin{cor}
Fix a weight $k$, a representation (or multiplier system) $\rho$, and an eigenvalue $\lambda$, and consider solutions $d$ to the equality $\lambda=d(k-1-d)$ that are non-negative integers. If there are none, the space $\mathcal{M}_{k-\infty,\lambda}^{sing,A}(V_{\infty}\otimes\rho)$ is trivial, and $\lambda$ is not an eigenvalue of $\Delta_{k-\infty}^{A}$. If there exists a single such solution $d$ (with or without multiplicity), then $\mathcal{M}_{k-\infty,\lambda}^{sing,A}(V_{\infty}\otimes\rho)=\mathcal{L}_{k,d}^{A}\big(\mathcal{M}_{k-2d}^{sing}(\rho)\big)$. Finally, assuming that both solutions of this quadratic equation are natural and different, say $d_{-}<d_{+}$, then $\mathcal{M}_{k-\infty,\lambda}^{sing,A}(V_{\infty}\otimes\rho)$ is the direct sum of $\mathcal{L}_{k,d_{-}}^{A}\big(\mathcal{M}_{k-2d_{-}}^{sing}(\rho)\big)$ and $\mathcal{L}_{k,d_{+}}^{A}\Big((2iy)^{2d_{+}-k}\overline{\mathcal{M}_{2d_{+}-k}^{nm,\leq 2d_{+}-k}(\overline{\rho})}\Big)$, and both summands are based on different components $F_{s}$: Such a component may be non-zero only for $d_{-}<s \leq d_{+}$ in the first summand, and for $0 \leq s \leq d_{-}$ in the second one. \label{eigenAtot}
\end{cor}

\begin{proof}
Any (non-trivial) eigenfunction of $\Delta_{k-\infty}^{A}$, with eigenvalue $\lambda$, must have some depth $d$, and Theorem \ref{eigenAgen} showed that the depth $d$ is possible if and only if it solves the required quadratic equation. This establishes the first assertion. Now, assuming that there is a solution $d\in\mathbb{N}$, the second solution to that equation is clearly $k-1-d$. This is an integer between 0 and $d-1$ if and only if $k$ is one of the integers considered in Proposition \ref{eigenAsp}. In particular if only one solution is in $\mathbb{N}$, then we are in the situation described in Theorem \ref{eigenAgen}, which proves the second assertion. The case where there are two such solutions is now easily seen to be the one described in Proposition \ref{eigenAsp}, with $d_{+}$ denoted $d$ there, and $d_{-}$ being just $e$. The assertion of that proposition is clearly equivalent to the last assertion here.
\end{proof}

We remark that the last situation in Corollary \ref{eigenAtot} can occur only if $k$ is an integer $\geq2$, for $k-1$ must be the sum of two distinct natural numbers for this to (possibly) happen. Note that also for such weights $k$ the infinitely many (negative) eigenvalues of the form $\lambda=d(k-1-d)$ for $d \geq k$ still produce just the eigenspaces from Theorem \ref{eigenAgen}, and Proposition \ref{eigenAsp} only applies to finitely many $\lambda$'s. Finally, the case of a double integer solution $d$ (with $\lambda=d^{2}$) corresponds to $k=2d+1$, which is indeed covered by Theorem \ref{eigenAgen} and not Proposition \ref{eigenAsp}.

\smallskip

We conclude the analysis of the eigenfunctions of $\Delta_{k-\infty}^{A}$ with the following result about the differential properties of these eigenfunctions, as well as the corresponding result about quasi-modular forms.
\begin{prop}
For $k$ and $\Delta$ such that $\mathcal{M}_{k-\infty,\lambda}^{sing,A}(V_{\infty}\otimes\rho)$ is non-zero, its intersection with the space $\mathcal{M}_{k-\infty}^{*}(V_{\infty}\otimes\rho)$ for any of the differential types from Definition \ref{MFQMFdef} is non-zero, and consists precisely of the lifts $\mathcal{L}_{k,d}(\varphi)$, with the appropriate $d$, for $\varphi\in\mathcal{M}_{k-2d}^{*}(\rho)$. In the special case with $\lambda=d_{+}d_{-}$ and $k=d_{+}+d_{-}+1$ for natural numbers $d_{-}<d_{+}$ we get, unless $*$ is $sing$ or $an$, only the space $\mathcal{L}_{k,d_{-}}\big(\mathcal{M}_{k-2d_{-}}^{*}(\rho)\big)$. \label{eigenAtypes}
\end{prop}

\begin{proof}
In order to investigate holomorphicity (or meromorphicity) of elements of $\mathcal{M}_{k-\infty}^{sing}(V_{\infty}\otimes\rho)$ we have to present them in a holomorphic basis. We shall use the limit images in $V_{\infty}$ of the basis $\binom{\tau}{1}^{m-r}\binom{1}{0}^{r}$ with $0 \leq r \leq m$ from Theorem \ref{QMFVVMF}. The transformation formula between the bases in that theorem clearly commutes with the maps $i_{m}$ (just adding another power of $\binom{\tau}{1}$ throughout), so that they also apply in the limit $V_{\infty}$. We therefore evaluate the functions $f_{r}$ associated with a sequence $F_{s}$ as the ones appearing in Theorem \ref{eigenAgen} and in the $\mathcal{L}_{k,d}$ part of Proposition \ref{eigenAsp}. We therefore substitute the formulae from that Theorem (which were seen to be valid also for the lifts of the form $\mathcal{L}_{d_{+}+d_{-}+1,d_{+}}(\varphi)$ from Proposition \ref{eigenAsp} provided that $s>d_{-}$) into the expression for $f_{r}$ from Theorem \ref{QMFVVMF}. We get $f_{r}=\sum_{s=r}^{d}\binom{s}{r}\binom{d}{s}\frac{\delta_{k-2d}^{d-s}\varphi}{(-2iy)^{s-r}}\Big/\prod_{q=0}^{d-s-1}(k-2d+q)$, a formula that is valid when $d$ is the only natural solution to $\lambda=d(k-1-d)$, when $d$ is the smaller such solution $d_{-}$, or when $d=d_{+}$ but $r>d_{-}$ (to avoid the lower bound on $s$). We can replace $\delta_{k-2d}^{d-s}\varphi$ by $\sum_{p=0}^{d-s}\binom{d-s}{p}\big[\prod_{q=d-s-p}^{d-s-1}(k-2d+q)\big]\frac{\partial_{\tau}^{d-s-p}\varphi}{(2iy)^{p}}$ by another application of Equation (56) of \cite{[Za]}, and note that in the summand corresponding to $s$ and $p$, the product of the three binomial coefficients equals $\binom{d}{r}\binom{d-r}{s+p-r}\binom{s+p-r}{s-r}$. Moreover, the denominators with $2iy$ become a total of $(-1)^{s-r}(2iy)^{s+p-r}$, and the products over $q$ cancel to a denominator of $\prod_{q=0}^{d-s-p-1}(k-2d+q)$. We make a summation index change by setting $m=s+p$ to be an integer between $r$ and $d$ and let $s$ go from $r$ to $m$ (with the substitution $p=m-s$). The terms $\binom{d}{r}\binom{d-r}{m-r}$, $\frac{\partial_{\tau}^{d-m}\varphi}{(2iy)^{m-r}}$, and the denominator $\prod_{q=0}^{d-m-1}(k-2d+q)$ are independent of $s$, and the remaining terms combine to $\sum_{s=r}^{m}\binom{m-r}{s-r}(-1)^{s-r}$. But this is the well-known expansion of $(1-1)^{m-r}=\delta_{m,r}$, reducing the total coefficient $f_{r}$ to just $\binom{d}{r}\partial_{\tau}^{d-r}\varphi\Big/\prod_{q=0}^{d-r-1}(k-2d+q)$. As the derivatives of $\varphi$ have the same differential properties as $\varphi$, this proves that all the functions $f_{r}$ of $\mathcal{L}_{k,d}(\varphi)$, hence also this lift itself, also share these properties. This proves the first assertion, and the second one now follows from the fact that the functions on which $\mathcal{L}_{k,d}$ is defined in Proposition \ref{eigenAsp} when $k$ is one of the weights considered (so that $d=d_{+}>d_{-}$) are not meromorphic.
\end{proof}
Note that the case where $\varphi$ has to be annihilated by some power $d-e$ of $\partial_{\tau}$ is precisely the case where this number $d-e$ is (up to sign) one of the multipliers forming the denominator appearing under $\partial_{\tau}^{d-r}\varphi$ for every $0 \leq r \leq e$ in the proof of Proposition \ref{eigenAtypes}.

Using Proposition \ref{eigenAtypes} and its proof, we obtain the following results for quasi-modular forms that are eigenfunctions of $\Delta_{k}^{A}$.
\begin{cor}
The quasi-modular forms that are eigenfunctions of $\Delta_{k}^{A}$ are as follows.
\begin{enumerate}[$(i)$]
\item Quasi-modular forms $f\in\widetilde{\mathcal{M}}_{k}^{sing}(\rho)$ that are eigenfunctions of $\Delta_{k}^{A}$ with eigenvalue $\lambda$ exist only if there is a solution $d\in\mathbb{N}$ to $\lambda=d(k-1-d)$.
\item Assuming that there is only one solution $d\in\mathbb{N}$ to the equation from part $(i)$, the associated eigenfunctions are precisely those functions of the form $\partial_{\tau}^{d}\varphi\Big/\prod_{q=0}^{d-1}(k-2d+q)$ for $\varphi\in\mathcal{M}_{k-2d}^{*}(\rho)$, for which the function  $f_{r}$ with $0 \leq r \leq d$ is $\binom{d}{r}\partial_{\tau}^{d-r}\varphi\Big/\prod_{q=0}^{d-r-1}(k-2d+q)$ (with a non-vanishing denominator). Such an element lies in $\widetilde{\mathcal{M}}_{k}^{*,\leq d}(\rho)$ with the same superscript $*$ as $\varphi$, but not in $\widetilde{\mathcal{M}}_{k}^{*,\leq d-1}(\rho)$ if it is non-zero.
\item If there are two distinct such solutions $d_{-}<d_{+}$ to the equation from part $(i)$, then the eigenfunctions lying in $\widetilde{\mathcal{M}}_{k}^{*,\leq d_{+}-1}(\rho)$ are those of the form  $\frac{(d_{+}-d_{-})!}{d_{+}!}\partial_{\tau}^{d_{-}}\varphi\in\widetilde{\mathcal{M}}_{k}^{*,\leq d_{-}}(\rho)$ for $\varphi\in\mathcal{M}_{k-2d_{-}}^{*}(\rho)$,
    having the functions $f_{r}=\binom{d_{-}}{r}\frac{(d_{+}-d_{-})!}{(d_{+}-r)!}\partial_{\tau}^{d_{-}-r}\varphi=\binom{d_{+}}{r}\frac{\partial_{\tau}^{d_{-}-r}\varphi}{(d_{-}-r)!}\Big/\binom{d_{+}}{d_{-}}$. Such a quasi-modular form has the same differential properties of $\varphi$ as in part $(ii)$. Another class of elements are of the form
    $\sum_{s=d_{-}+1}^{d_{+}}\binom{d_{+}}{s}\frac{(-1)^{d}(s-d_{-}-1)!}{(d_{+}-d_{-}-1)!(2iy)^{d_{-}+1}}\overline{h^{(d_{+}-s)}}$ for some $h$ from $\mathcal{M}_{d-e-1}^{nm,\leq d-e-1}(\overline{\rho})$, with depth $d_{+}$, for which $f_{r}$ with $0 \leq r \leq d_{+}$ equals $\binom{d_{+}}{r}\frac{(-1)^{d_{+}-r}}{(d_{+}-d_{-}-1)!(2iy)^{d_{-}+1-r}}\sum_{s=d_{-}+1}^{d_{+}}\binom{d_{+}-r}{s-r}(s-d_{-}-1)!\overline{h^{(d_{+}-s)}}$. Any eigenfunction of $\Delta_{k}^{A}$ is the sum of two elements, one of each sort.
\end{enumerate} \label{DeltaAQMF}
\end{cor}

\begin{proof}
Recalling that the coefficients of an element $F\in\mathcal{M}_{k-\infty}^{sing}(V_{\infty}\otimes\rho)$ in the basis $\binom{\tau}{1}^{\infty-r}\binom{1}{0}^{r}$ are the functions $f_{r}$ from the quasi-modular form $f$ associated with $F$ via Theorem \ref{QMFVVMF}, we transfer the results of Proposition \ref{eigenAtypes} via the commutative diagrams from Remark \ref{limops} (as in Definition \ref{Lapdef} for $\Delta_{k}^{A}$ itself). This establishes parts $(i)$ and $(ii)$ as well as the assertion about eigenfunctions in $\widetilde{\mathcal{M}}_{k}^{*,\leq d_{+}-1}(\rho)$ in part $(iii)$. For the remaining assertion in part $(iii)$ we use the explicit expressions for the functions $F_{s}$ appearing in Proposition \ref{eigenAsp}, substitute them again into the formula from Theorem \ref{QMFVVMF}, and apply the same transfer.
\end{proof}

\begin{rmk}
Recall that the formula for $f_{r}$ with $d_{-}<r \leq d_{+}$ appearing in the proof of Proposition \ref{eigenAtypes} is
$\binom{d_{+}}{r}(-1)^{d_{+}-r}\frac{(r-d_{-}-1)!}{(d_{+}-d_{-}-1)!}$ times the image of the action of $\partial_{\tau}^{d-r}$ on $(2iy)^{d_{+}-d_{-}-1}\overline{h}$ for $h$ in $\mathcal{M}_{d_{+}-d_{-}-1}^{nm,\leq d_{+}-d_{-}-1}(\overline{\rho})$. On the space of polynomials of given degree bound $\mu$ we define the twisting operation sending $p(x)=\sum_{i=0}^{\mu}a_{i}x^{i}$ to the polynomial sending $x$ to $x^{\mu}p\big(\frac{1}{x}\big)=\sum_{i=0}^{\mu}a_{\mu-i}x^{i}$. Considering $h$ as a polynomial of degree at most $d_{+}-d_{-}-1$ in $\frac{1}{2iy}$ whose coefficients $\{\eta_{j}\}_{j=0}^{d_{+}-d_{-}-1}$ are meromorphic functions, the modular form that we differentiate can be seen to equal $(-1)^{d_{+}-d_{-}-1}\sum_{j=0}^{d_{+}-d_{-}-1}\overline{\eta_{d_{+}-d_{-}-1-j}}(-2iy)^{j}$ (up to the external sign, this is the polynomial arising from $h$ by twisting and conjugating the coefficients, evaluated at $-2iy$). By denoting this twisted and conjugated polynomial by $\widehat{\overline{h}}$, we see that applying $\partial_{\tau}$ to such polynomials evaluated at $-2iy$ amounts to negating the derivative. Denoting the $p$th derivative of $\widehat{\overline{h}}$ as a polynomial by $\widehat{\overline{h}}^{(p)}$ (it is no longer modular, just quasi-modular), the function $f_{r}$ with $d_{-}<r \leq d_{+}$ from the proof of Proposition \ref{eigenAtypes} and part $(iii)$ of Corollary \ref{DeltaAQMF} becomes just
$\binom{d}{r}(-1)^{d_{+}-d_{-}-1}\frac{(r-d_{-}-1)!}{(d_{+}-d_{-}-1)!}\widehat{\overline{h}}^{(d_{+}-r)}(-2iy)$. The proof of Proposition \ref{eigenAtypes} can be modified slightly to give us the values of the functions $f_{r}$ with $0 \leq r \leq d_{-}$ as well, if we just notice that the argument remains unchanged except for the fact that the inner sum over $s$ is now taken from $d_{-}+1$ (rather than $r$) to $m$. Using Lemma \ref{cancbinom} below we can evaluate such a function $f_{r}$ as $\frac{d_{+}!(-1)^{d_{+}-r}}{r!(d_{+}-d_{-}-1)!(d_{-}-r)!}\sum_{m=d_{-}+1}^{d_{+}}\frac{1}{(d_{+}-m)!(-2iy)^{m-r}(m-r)}\cdot\widehat{\overline{h}}^{(d_{+}-m)}(-2iy)$. The associated quasi-modular form is attained by setting $r=0$ in the latter expression. \label{twistpol}
\end{rmk}

Since the map $\tau\mapsto\widehat{\overline{h}}(-2iy)$ is a modular form from $\mathcal{M}_{d_{-}+1-d_{+}}^{sing}(\rho)$, the functions $\widehat{\overline{h}}^{(d_{+}-m)}(-2iy)\big/(-2iy)^{m}$ appearing in the expression for $f_{0}$ in Remark \ref{twistpol} are all in $\widetilde{\mathcal{M}}_{k}^{sing,\leq d_{+}}(\rho)$ (since $k=d_{+}+d_{-}+1$), and of depth precisely $d_{+}$. It follows that the coefficients $\frac{d_{+}!(-1)^{d_{+}}}{(d_{+}-d_{-}-1)!d_{-}!(d_{+}-m)!m}$ with which they appear in $f_{0}$ in that remark are the ones in which each function $f_{r}$ with $d_{-}<r \leq d_{+}$ involves just the derivative $\widehat{\overline{h}}^{(d_{+}-r)}(-2iy)$, with some coefficient
$\binom{d_{+}}{r}(-1)^{d_{+}-d_{-}-1}\frac{(r-d_{-}-1)!}{(d_{+}-d_{-}-1)!}$ which reduces to a sign for $r=d_{+}$.

\subsection{Eigenfunction of the Other Laplacian}

The analysis of the eigenfunctions with respect to $\Delta_{k-\infty}^{V}$ is different for functions of depth 0 and for functions with higher depth. This is because the equality associated with $s=d+1$ in Corollary \ref{compeigen} (or Remark \ref{relsalls}) involves the multiplier $d$, whose vanishing for $d=0$ changes the analysis significantly. We begin with the case of depth 0, since it is simpler.
\begin{prop}
For any eigenvalue $\lambda$, the space $\mathcal{M}_{k,\lambda}^{sing}(\rho)$ embeds into $\mathcal{M}_{k-\infty,\lambda}^{sing,V}(V_{\infty}\otimes\rho)$ via the limit map of the $i_{m}$'s. \label{eigenVd0}
\end{prop}

\begin{proof}
The equation for $s=1$ in Corollary \ref{compeigen} reduces to just $0=0$, so that for elements of
$\mathcal{M}_{k-\infty}^{sing}(V_{\infty}\otimes\rho)$ to be in $\mathcal{M}_{k-\infty,\lambda}^{sing,V,0}(V_{\infty}\otimes\rho)$ the only equation that $F_{0}$ must satisfy is the one associated with $s=0$. But this equation reduces to $\Delta_{k}F_{0}=-\lambda F_{0}$, so that the element $F:\tau \mapsto F_{0}\binom{\tau}{1}^{\infty}$ lies in $\mathcal{M}_{k-\infty,\lambda}^{sing,V}(V_{\infty}\otimes\rho)$ if and only if $F_{0}\in\mathcal{M}_{k,\lambda}^{sing}(\rho)$.
\end{proof}
Since the construction from Proposition \ref{eigenVd0} satisfies the conditions of Definition \ref{liftdef} (trivially), we denote the limit map of the $i_{m}$'s on $\mathcal{M}_{k,\lambda}^{sing}(\rho)$ by $\mathcal{L}_{k,0}^{V,\lambda}$. This is one analogue of the maps from Theorem \ref{eigenAgen} and Proposition \ref{eigenAsp}.

\smallskip

We now turn our attention to the case of higher depth. Assuming that $F_{d+1}=0$ but $F_{d}\neq0$ in Corollary \ref{compeigen}, the equality associated with $s=d+1$ in that corollary implies the meromorphicity of $F_{d}$ (since $d\neq0$). We are interested in defining lifts also in this case, so that we shall need to know what to substitute for $\Delta_{k-2s}F_{s}$ and $4y^{2}\partial_{\overline{\tau}}F_{s-1}$ in the formula from Corollary \ref{compeigen}.

\begin{lem}
Let $h$ be an element of $\mathcal{M}_{l,\mu}^{sing}(\rho)$ for some weight $l$ and some eigenvalue $\mu$. Then $\delta_{l}h$ lies in $\mathcal{M}_{l+2,\mu+l}^{sing}(\rho)$, and applying $4y^{2}\partial_{\overline{\tau}}$ to $\delta_{l}h$ gives $-(\mu+l)h$. \label{LDeldelmer}
\end{lem}

\begin{proof}
When we let $\partial_{\overline{\tau}}$ act of $\delta_{l}h=\partial_{\tau}h+\frac{lh}{2iy}$ we get three terms. The ones in which $\partial_{\overline{\tau}}$ operates on $\partial_{\tau}h$ or on $h$ itself give the two terms appearing in $\delta_{l}\partial_{\overline{\tau}}h$. After multiplying by $4y^{2}$, these two terms combine to $\Delta_{l}h$, which equals $-\mu h$ by our assumption on $h$. But we also have the third term $lh\partial_{\overline{\tau}}\frac{1}{2iy}=lh\partial_{\overline{\tau}}\frac{1}{\tau-\overline{\tau}}$, which equals $\frac{lh}{(\tau-\overline{\tau})^{2}}=-\frac{lh}{4y^{2}}$. Multiplying that term by $4y^{2}$ as well, we obtain the second assertion. Finally, we recall that $\Delta_{l+2}\delta_{l}h$ can be decomposed as the action of $\delta_{l}$ on $4y^{2}\partial_{\overline{\tau}}\delta_{l}h$, and as the latter modular form was evaluated as $-(\mu+l)h$, the assertion about the eigenvalue of $\delta_{l}$ follows as well.
\end{proof}
In fact, we have already mentioned that \cite{[V]} considers the $\mathfrak{sl}_{2}$-triple consisting of $(\{E=\delta_{k}\}_{k},F=4y^{2}\partial_{\overline{\tau}},H=W)$, acting on the spaces of modular forms. Then Lemma \ref{LDeldelmer} amounts to the trivial commutation relations with the associated Casimir operator, taking into consideration the difference in normalization between the latter operator and the Laplacians $\{\Delta_{k}\}_{k}$.

\begin{cor}
Let $\varphi$ be an arbitrary non-zero element of $\mathcal{M}_{k-2d}^{mer}(\rho)$, which we assume to be non-constant if $k=2d$, and take $\lambda\in\mathbb{C}$. Consider an element $F$ of depth $d$ in $\mathcal{M}_{k-\infty}^{sing}(V_{\infty}\otimes\rho)$, expanded as in Lemma \ref{Lapeval}, and assume that $F_{s}$ is of the form $a_{s}\delta_{k-2d}^{d-s}\varphi$ for some constant $a_{s}$ for every $0 \leq s \leq d$, with $a_{d}=1$. Then $F\in\mathcal{M}_{k-\infty,\lambda}^{sing,V}(V_{\infty}\otimes\rho)$ if and only if the equality
\[\lambda a_{s}=(d-s)(k-1-d-s)a_{s}-(s+1)a_{s+1}-(s-1)\big[(d-s+1)(k-d-s)a_{s-1}-sa_{s}\big]\] holds for every $0 \leq s \leq d$. \label{coeffLkdV}
\end{cor}

\begin{proof}
We claim that $F_{s}$ is an eigenfunction for every $0 \leq s \leq d$, with eigenvalue $(d-s)(k-1-d-s)$. We prove this by decreasing induction on $s$, starting from the fact that $F_{d}=\varphi$ is meromorphic, hence harmonic and with the eigenvalue 0. Assuming that $F_{s}$ has the desired eigenvalue, Lemma \ref{LDeldelmer} shows that $F_{s-1}$ is also an eigenfunction, whose eigenvalue is obtained by adding $k-2s$ to that of $F_{s}$. As this combines to desired value $(d-s+1)(k-d-s)$, the claim follows. Moreover, $4y^{2}\partial_{\overline{\tau}}F_{s-1}$ would give the latter coefficient times $-a_{s-1}\delta_{k-2d}^{d-s}\varphi$ by the second assertion of Lemma \ref{LDeldelmer}. Therefore all the terms in the equality associated with $s$ in Corollary \ref{compeigen} are multiples of $\delta_{k-2d}^{d-s}\varphi$, where $\Delta_{k-2s}$ multiplies $F_{s}$ by minus its eigenvalue as well. Now, this function $\delta_{k-2d}^{d-s}\varphi$ does not vanish by our assumption on $\varphi$ (we have found the kernel of $\delta_{k-2d}^{d-s}$ in Lemma \ref{nearmeroBol}, and it intersects $\mathcal{M}_{k-2d}^{mer}(\rho)$ trivially, unless $k=2d$ where the intersection consists of the constant functions), so that we can compare the coefficients themselves on both sides and invert the sign. This yields the desired equality.
\end{proof}

Once again, when we fix a depth $d>0$ and avoid the integral weights between $d+1$ and $2d$, the existence of lifts is possible only for finitely many eigenvalues (but now typically more than one, unlike in Theorem \ref{eigenAgen}).
\begin{thm}
For a depth $d>0$ and a weight $k$ that is not an integer between $d+1$ and $2d$, a lift from $\mathcal{M}_{k-2d}^{mer}(\rho)$ into $\mathcal{M}_{k-\infty,\lambda}^{sing,V}(V_{\infty}\otimes\rho)$ exists, independently of $\Gamma$ and $\rho$, if and only if $\lambda$ is one of the numbers $j(k-1-j)$ with $j$ an integer between 0 and $d-1$. \label{eigenVgen}
\end{thm}

\begin{proof}
The equation associated with $s>1$ in Corollary \ref{coeffLkdV} can be rewritten as
\begin{equation}
a_{s-1}=\frac{[(d-s)(k-1-d-s)+s(s-1)-\lambda]a_{s}-(s+1)a_{s+1}}{(s-1)(d-s+1)(k-d-s)}, \label{recexp}
\end{equation}
since the numbers $1-s$ and $d+1-s$ do not vanish for $1<s \leq d$, and our assumption on $k$ implies that $k-d-s\neq0$ as well. Equation \eqref{recexp} for $s=d$ is valid as well, with $a_{d+1}$ considered to be 0, and we know that $a_{d}=1$. A simple induction now shows that $a_{s}$ is a polynomial of degree $d-s$ in $\lambda$ for every $1 \leq s \leq d$, with coefficients depending on $d$ and $k$, and in particular $a_{1}$ and $a_{2}$ are polynomials of degree $d-1$ and $d-2$ respectively. But setting $s=1$ in Corollary \ref{coeffLkdV} implies the vanishing of $[(d-1)(k-2-d)-\lambda]a_{1}-2a_{2}$, which is a polynomial of degree $d$ in $\lambda$. Therefore $\lambda$ must be a root of this polynomial in order for these coefficients to exist. The coefficient $a_{0}$ is determined by the equality with $s=0$ in Corollary \ref{coeffLkdV} as $\frac{a_{1}}{d(k-1-d)-\lambda}$, provided that $\lambda \neq d(k-1-d)$.

Now, the proof of Corollary \ref{eigenAtot} shows that $j(k-1-j) \neq d(k-1-d)$ for any $0 \leq j<d$. This means that if these values are roots of the polynomial in $\lambda$, then they produce lifts, since the expression for $a_{0}$ is also well-defined. Therefore the theorem will be proved once we provide, for fixed $d$ and $j$ (and $k$), coefficients $a_{d,s}^{(j)}$ for $0 \leq s \leq d$ with $a_{d,d}^{(j)}=1$ that satisfy all the equalities from Corollary \ref{coeffLkdV}. For such $d$, $s$, and $j$ we therefore set
\begin{equation}
a_{d,s}^{(j)}=\sum_{l=0}^{j}(-1)^{l}\binom{j}{l}\binom{d-l}{s}\frac{(d-1-l)!}{(d-1)!}\cdot\frac{\prod_{p=1}^{l}(k-j-p)}{\prod_{q=0}^{d-s-1}(k-2d+q)}. \label{adsj}
\end{equation}
Note that the factorial in the numerator is finite since $l \leq j<d$, and the denominator involving $k$ does not vanish as in Theorem \ref{eigenAgen}. Since only for indices $l \leq d-s$ the binomial coefficient does not vanish, and the term with $l=0$ gives 1 in the numerator, the coefficient $a_{d,d}^{(j)}$ indeed equals 1. Moreover, the natural extension to $s<0$ or to $s>d$ is 0, and we have to verify that the equalities from Corollary \ref{coeffLkdV} hold for every $s$ with our given value of $\lambda$.

Now, the right hand side of the equation associated to $s$ in that corollary involves $(d-s)(k-1-d-s)a_{d,s}^{(j)}-(s+1)a_{d,s+1}^{(j)}$, as well as the same expression but with $s$ replaced by $s-1$ and multiplied by $s-1$. When we evaluate the first difference, the denominator becomes the same product over $q$ in both terms (by cancelation). Moreover, in the summand associated with $l$ in that difference we get, instead of the binomial coefficient involving $s$, the difference $(d-s)\binom{d-l}{s}-(s+1)\binom{d-l}{s+1}$. But the second expression equals $(d-l-s)\binom{d-l}{s}$, so that the difference reduces to $l\binom{d-l}{s}$. From this we have to subtract $s-1$ times the same expression with $s$ replaced by $s-1$, and we compare the products over $q$ in both denominators. The required equality reduces to the statement that replacing the binomial coefficient involving $s$ in Equation \eqref{adsj} by the expression $l(k-d-s-1)\binom{d-l}{s}-l(s-1)\binom{d-l}{s-1}$ results in the same total value as the one obtained from multiplying that equation by $j(k-1-j)$.

In order to do so, we shall write the required combination of binomial coefficients as the sum of $l(k-1-l)\binom{d-l}{s}$, $-l(d-l)\binom{d-l}{s}$, $-ls\binom{d-l}{s}$, and $-l(s-1)\binom{d-l}{s-1}$. The latter two expressions equal $-l(d-l)\binom{d-l-1}{s-1}$ and $-l(d-l)\binom{d-l-1}{s-2}$ respectively, and their sum, which equals $-l(d-l)\binom{d-l}{s-1}$, merges with the second original expression to $-l(d-l)\binom{d-l+1}{s}$. Plugging in the remaining terms from Equation \eqref{adsj}, we would like to change the summation index (but not for the part involving the initial coefficient $l(k-1-l)\binom{d-l}{s}$). The multiplier $l$ allows us to restrict attention to $l\geq1$, the expression $l\binom{j}{l}$ equals $(j-l+1)\binom{j}{l-1}$, and $d-l$ combines with the factorial from Equation \eqref{adsj} to produce $(d-l)!$. The summation index change $l \mapsto l+1$ now gives us all the multipliers associated with $l$ in Equation \eqref{adsj} back, including an inversion of the sign $(-1)^{l}$, and we only have an extra multiplier of $k-j-l-1$ from the product over $p$. Noting that we have a multiplier of $j-l$ (after the index change), we can take the sum over $l$ to be from 0 to $j$ yet again. Recalling the initial terms with $l(k-l-1)\binom{d-l}{s}$, we find that the $l$th summand in our expression is the one from Equation \eqref{adsj} multiplied by $l(k-1-l)+(j-l)(k-1-j-l)$. But since this expression reduces to $j(k-1-j)$, we indeed get the desired equality to $j(k-1-j)a_{d,s}^{(j)}$.

This establishes the existence of the lifts to the eigenvalues $j(k-1-j)$. Now, if $k$ is not an integer between 2 and $2d$, then the proof of Corollary \ref{eigenAtot} implies that these are all distinct eigenvalues, yielding $d$ roots of our polynomial in $\lambda$. Therefore the resulting polynomial in $\lambda$ is some non-zero scalar multiple $\prod_{j=0}^{d-1}[\lambda-j(k-1-j)]$. Letting $k$ tend to an integer between 2 and $d$, the continuity of all the expressions involved and the non-vanishing of any denominator involved imply that this is the form of the polynomial also for these values of $k$ (now with some double roots). Hence the lifts exist, and are clearly well-defined by Definition \ref{liftdef}, only for the asserted values of $\lambda$ also in this case.
\end{proof}

\begin{rmk}
Note that once $\mathcal{L}_{k,d}^{V,\lambda}$ is defined for some $d>0$ and some $\lambda$, the lift $\mathcal{L}_{k,d+1}^{V,\lambda}$ is also defined for that eigenvalue $\lambda$. This will be useful for characterizing all the eigenfunctions (regardless of the depth) in Corollary \ref{eigenVtot} below. Moreover, the form of the eigenvalues from Theorem \ref{eigenVgen} allows us to simplify the notation $\mathcal{L}_{k,d}^{V,\lambda}$ for $\lambda=j(k-1-j)$ (and $k$ not one of the $d$ problematic weights) to $\mathcal{L}_{k,d}^{V,j}$. It is an injective map from $\mathcal{M}_{k-2d}^{mer}(\rho)$ into $\mathcal{M}_{k-\infty,j(k-1-j)}^{sing,V}(\rho)$. Note that when two different natural values of $j$, say $d>j_{+}>j_{-}$, describe the same eigenvalue, then $k=j_{+}+j_{-}+1$, and the two lifts $\mathcal{L}_{k,d}^{V,j_{\pm}}$ are defined into the same eigenspace. Now, in this case the multiplier with $p=j_{-}+1$ in each summand with $l>j_{-}$ in the definition of $a_{d,s}^{(j_{+})}$ vanishes (recall the value of $k$), so that the sum goes up to $j_{-}$ in any case. Moreover, substituting the value of $k$ in the remaining values of $l$ shows that $\binom{j_{+}}{l}\prod_{p=1}^{l}(k-j_{+}-p)$ coincides with $\binom{j_{-}}{l}\prod_{p=1}^{l}(k-j_{-}-p)$, and we deduce that $\mathcal{L}_{k,d}^{V,j_{+}}=\mathcal{L}_{k,d}^{V,j_{-}}$. Since Equation \eqref{adsj} is simpler when $j$ is smaller, we shall use the notation $\mathcal{L}_{k,d}^{V,j_{-}}$ in this case (in correspondence with the situation where $k=j_{+}+j_{-}+1$ with $j_{+} \geq d>j_{-}$ and $k>2d$, where only $\mathcal{L}_{k,d}^{V,j_{-}}$ is defined to begin with). \label{deptheigen}
\end{rmk}

\smallskip

The behavior of eigenfunctions with depth $d>0$ and weight $k$ between $d+1$ and $2d$ depends on the eigenvalue. In order to analyze them we shall investigate the expression for the coefficients $a_{d,j}^{(s)}$ from Equation \eqref{adsj} a little deeper, in order to transform it to a more reduced form as a rational function of the weight $k$. We shall make use of the following lemma, which will help us in later evaluations as well.
\begin{lem}
Let $0\leq\beta\leq\sigma$ be integers, and fix some constants $\xi$ and $\eta$. Then the sum $\sum_{\nu=0}^{\beta}(-1)^{\beta-\nu}\binom{\beta}{\beta-\nu}\prod_{u=1}^{\nu}(\xi-u)\prod_{w=\nu+1}^{\sigma}(\eta-w)$ can be written as $\prod_{w=\beta+1}^{\sigma}(\eta-w)\prod_{\mu=1}^{\beta}(\xi-\eta+\beta-\mu)$. \label{prodtrans}
\end{lem}

\begin{proof}
We apply induction on $\beta$. When this parameter vanishes the product over $\mu$ is trivial, and both sides equal $\prod_{w=1}^{\sigma}(\eta-w)$. We therefore assume that $\beta>0$ and that the formula holds for $\beta-1$. The binomial coefficient then splits as $\binom{\beta-1}{\beta-\nu}+\binom{\beta-1}{\beta-\nu-1}$. From the terms associated with the second summand we obtain $\sum_{\nu=0}^{\beta-1}(-1)^{\beta-\nu}\binom{\beta-1}{\beta-\nu-1}\prod_{u=1}^{\nu}(\xi-u)\prod_{w=\nu+1}^{\beta}(\eta-w)$, while in those involving the first summand we replace $\nu$ by $\nu+1$. This latter operation yields $-\sum_{\nu=0}^{\beta-1}(-1)^{\beta-\nu}\binom{\beta-1}{\beta-\nu-1}\prod_{u=1}^{\nu+1}(\xi-u)\prod_{w=\nu+2}^{\beta}(\eta-w)$, and we note that the product $\prod_{u=\alpha+1}^{\nu}(\xi-u)\prod_{w=\nu+2}^{\beta}(\eta-w)$ appears in both terms. In the first expression we have the extra multiplier $\eta-\nu-1$, while the second one is obtained after multiplying by $\xi-\nu-1$. As they come with different signs, this produces $\xi-\eta$ times $\sum_{\nu=0}^{\beta-1}(-1)^{\beta-\nu-1}\binom{\beta-1}{\beta-\nu-1}\prod_{u=1}^{\nu}(\xi-u)\prod_{w=\nu+1}^{\beta-1}(\eta-1-w)$, which is the sum associated with $\beta-1$, $\sigma-1$, $\xi$, and $\eta-1$ (note the total sign and the index change in $w$). Since the external multiplier is the one associated with $\mu=\beta$, the induction hypothesis implies that our sum gives the desired expression.
\end{proof}

Let us now prove the desired reduced form of the coefficients $a_{d,s}^{(j)}$.
\begin{lem}
The formula for the coefficient $a_{d,0}^{(j)}$ in Equation \eqref{adsj} reduces to $(-1)^{j}\frac{(d-j-1)!}{(d-1)!}\Big/\prod_{q=0}^{d-j-1}(k-2d+q)$. On the other hand, the coefficients with $j \geq s\geq1$ in that equation can be written as rational functions of $k$ in which the denominator is $\prod_{q=0}^{d-j-2}(k-2d+q)$. \label{redadsj}
\end{lem}
For $j=d-1$ the latter denominator is an empty product, meaning that $a_{d,d-1}^{(s)}$ is a polynomial in $k$ (with no denominator) for every $s\geq1$.

\begin{proof}
We consider the binomial coefficient $\binom{d-l}{s}$ from Equation \eqref{adsj} as a polynomial of degree $s$ in $l$, when $d$ is considered as a parameter. Since we also have the multiplier $(d-1-l)!$, we express the polynomial $\binom{d-l}{s}$ using the basis $\frac{(d+t-1-l)!}{(d-1-l)!}$ with $0 \leq t \leq s$ (this quotient is, for fixed $d$ and $t$, a polynomial of degree $t$ in $l$). It follows that the combination $\binom{d-l}{s}(d-1-l)!$ from Equation \eqref{adsj} can be written as $\sum_{t=0}^{s}b_{d,s,t}\frac{(d+t-1-l)!}{(d-1)!}$, with $b_{d,s,t}$ rational constants depending on $d$, $t$, and $s$. For $s=0$ we get a single term with $t=0$ and $b_{d,0,0}=1$. On the other hand, if $s\geq1$ then the coefficient $\binom{d-l}{s}$ involves a multiplier of $d-l$, so that the constant $b_{d,s,0}$ vanishes for such $s$ and the sum over $t$ goes only from 1 to $s$. We take a coefficient of $(-1)^{j}\frac{b_{d,s,t}}{(d-1)!}$ out, and write the factorial $(d+t-1-l)!$ as $\prod_{w=l+1}^{d+t-1}(d+t-w)$. We can then apply Lemma \ref{prodtrans} with the parameters $\beta=j$, $\sigma=d+t-1$, $\xi=k-j$, and $\eta=d+t$, and deduce that the sum $\sum_{l=0}^{j}(-1)^{j-l}\binom{j}{l}(d+t-1-l)!\prod_{p=1}^{l}(k-j-p)$ equals the product of $\prod_{w=j+1}^{d+t-1}(d+t-w)=(d+t-j-1)!$ (as $j<d$, this factorial is well-defined) and $\prod_{\mu=1}^{j}(k-j-d-t+j-\mu)=\prod_{q=d-t-j}^{d-t-1}(k-2d+q)$.

Now, for $s=0$ only the term with $t=0$ occurs, and our product over $q$ cancels with the relevant terms from the denominator of Equation \eqref{adsj}. Recalling the value of $b_{d,0,0}$, this gives the desired value of $a_{d,0}^{(j)}$. Turning our attention to the indices $s\geq1$, we get a sum over $t$, each of whose terms gives a different polynomial. We are interested in the common divisor of all these polynomials, so that we may ignore the coefficients $(-1)^{j}(d+t-j-1)!b_{d,s,t}$. Note that adding 1 to $t$ translates the set of indices $q$ involved by $-1$. Moreover, the minimal index $q$ appearing for $t=1$ is $d-1-j$, while the maximal one for $t=s$ is $d-s-1$. It follows that when $s \leq j$ all the terms with $d-1-j \leq q \leq d-s-1$ appear in the polynomials arising from every index $t$. We can thus cancel them from the denominator of $a_{d,s}^{(j)}$ in Equation \eqref{adsj}, and indeed just the asserted denominator remains.
\end{proof}

We can now extend the definition of the lifts $\mathcal{L}_{k,d}^{V,j}$ from Theorem \ref{eigenVgen} for some of the excluded weights for any of the indices $j$.
\begin{prop}
Assume that the weight $k$ is $d+1+e$ for some integer $0 \leq e<d$ as in Proposition \ref{eigenAsp}. In this case the lift $\mathcal{L}_{k,d}^{V,j}$ from Theorem \ref{eigenVgen} is well-defined for any $e<j<d$, and it is the unique lift for such modular forms also in these cases. Moreover, the eigenvalues $\lambda=j(k-1-j)$ with such $j$ are the only numbers $\lambda$ for which a lift from $\mathcal{M}_{k-2d}^{mer}(\rho)$ to $\mathcal{M}_{k-\infty,\lambda}^{sing,V}(V_{\infty}\otimes\rho)$ can be defined. \label{liftprobwt}
\end{prop}

\begin{proof}
First we see, as in the proof of Theorem \ref{eigenVgen}, that lifts exist only for at most $d-e$ values of $\lambda$. Indeed, Equation \eqref{recexp} is valid for $s>e+1$, since its denominator does not vanish, so that $a_{e+1}$ and $a_{e+2}$ are polynomials of degrees $d-e-1$ and $d-e-2$ in $\lambda$ respectively. But in the equation associated with $s=e+1$ in Corollary \ref{coeffLkdV} the coefficient $k-d-e-1$ in front of $a_{e}$ vanishes. This transforms this equation into the vanishing of $[(e+1)^{2}-d-\lambda]a_{e+1}-(e+2)a_{e+2}$, which is a polynomial of degree $d-e$ in $\lambda$. Note that for $e=0$ we get the same polynomial from the proof of Theorem \ref{eigenVgen} (since this happens with $k=d+1$), so that its roots are the ones given in that theorem.

Now, the value of $a_{d,s}^{(j)}$ from Equation \eqref{adsj} is finite for $s>e$, since the product over $q$ in the denominator was seen in the proof of Proposition \ref{eigenAsp} to be the non-zero number $(-1)^{d-s}\frac{(d-e-1)!}{(s-e-1)!}$ in this case. On the other hand, if $j \geq e\geq1$ then we can apply Lemma \ref{redadsj} with $s=e$, and find that $a_{d,e}^{(j)}$ has the same non-zero denominator as $a_{d,j+1}^{(j)}$. The expression defining $a_{d,e}^{(j)}$ in Equation \eqref{adsj} is therefore finite also for $k=d+1+e$. The verifications from Theorem \ref{eigenVgen} now show that for these values of $\lambda$ we have a non-trivial solution, yielding the $d-e$ roots of the polynomial for $\lambda$ (both for $e=0$ and for positive $e$). While these roots are not distinct in general (this is easily seen using the expression for $k$ and $j$ in Remark \ref{deptheigen}), the same limit argument from the end of the proof of Theorem \ref{eigenVgen} (using the finiteness of $a_{d,s}^{(j)}$ with $1 \leq s<e$ from Lemma \ref{redadsj} as well if necessary) shows that these are the solutions for the resulting polynomial in $\lambda$, and the only ones.

Let us now see for which of these values does a lift exist. For $j>e$ the value of $a_{d,0}^{(j)}$ from Lemma \ref{redadsj} is also finite (by the same argument), proving the well-definedness of $\mathcal{L}_{k,d}^{V,j}$ from Theorem \ref{eigenVgen} in this case. As for the uniqueness, note that the values of $a_{d,s}^{(j)}$ with $s>e$ are determined by the equations with $s>e+1$. In addition, if $a_{d,e}^{(j)}$ is given, then the equation with index $1<s \leq e$ determines $a_{d,s-1}^{(j)}$, and $a_{d,0}^{(j)}$ is determined using the equation with $s=0$. Therefore the difference between two lifts with subscripts $k=d+1+e$ and $d$ and superscript $e<j<d$ is a lift of the form $\mathcal{L}_{k,e}^{V,j}$, which does not exist by Theorem \ref{eigenVgen} since both $j$ and $k-1-j$ are larger than $e$ (the weight is $d+1+e>2e$, so that the consideration of such a lift is indeed the one appearing in that theorem). Therefore $\mathcal{L}_{k,d}^{V,j}$, which by Remark \ref{deptheigen} should be written as $\mathcal{L}_{k,d}^{V,j_{-}}$, is unique also in this case.

On the other hand, for $j=e$ the term with $q=d-j-1$ in the denominator of $a_{d,0}^{(j)}$ from Lemma \ref{redadsj} or Equation \eqref{adsj} vanishes, so that $\mathcal{L}_{d+1+e,d}^{V,e}$ cannot be defined as in Theorem \ref{eigenVgen}. Moreover, the proof of that theorem implies that when a lift exists, it must be the one described by that theorem. Note that the possibility of $\varphi$ being a non-zero constant (so that $\delta_{k-2d}^{d-s}\varphi$ itself may vanish) cannot occur, since the inequalities $e<j<d$ imply that $k$ cannot be $2d$ (since $e<d-1$). Therefore the lift $\mathcal{L}_{d+1+e,d}^{V,e}$ does not exist at all (using another interpretation like in Remark \ref{deptheigen}, with $j_{-}=e$ and $j_{+}=d$ say, cannot overcome this problem as well).
\end{proof}

\subsection{Sesqui-Harmonic Modular Forms}

The fact that lifts do not exist for some weights, depths, and eigenvalues does not mean that there cannot be different types of eigenfunctions associated with the parameters. While we shall see in Corollary \ref{eigenVtot} below that lifts (either to depth 0 as in Proposition \ref{eigenVd0} or to a higher depth as in Theorem \ref{eigenVgen} and Proposition \ref{liftprobwt}) produce all the eigenfunctions in most cases, the combinations of weights, depths, and eigenvalues excluded in Proposition \ref{liftprobwt} do have non-trivial eigenfunctions. To find those we shall need the notion of \emph{sesqui-harmonic modular forms}, as defined in \cite{[BDR]}. We make the obvious generalization of this notion as follows.
\begin{defn}
Given a weight $l$ and a representation (or multiplier system) of $\Gamma$, we call a modular form $\Psi\in\mathcal{M}_{l}^{sing}(\rho)$ \emph{sesqui-harmonic} if $\Delta_{l}\Psi\in\mathcal{M}_{l}^{mer}(\rho)$. We denote the the space of sesqui-harmonic elements of $\mathcal{M}_{l}^{sing}(\rho)$ by $\mathcal{M}_{l,ses}^{sing}(\rho)$, and its intersection with $\mathcal{M}_{l}^{an}(\rho)$ by $\mathcal{M}_{l,ses}^{an}(\rho)$. \label{sesquidef}
\end{defn}
Indeed, the Laplacian $\Delta_{l}$ can be written as the composition $\xi_{2-l}\circ\xi_{l}$ (up to constants arising from the normalizations of these operators) of the $\xi$-operators from \cite{[BFu]}. Therefore harmonic modular forms are those vanishing after two applications of $\xi$-operators, and \cite{[BDR]} defined sesqui-harmonic modular forms to be those that are eliminated after three such applications. Generalizing to arbitrary weights and representations, and noting that the $\xi$-operator is essentially an anti-holomorphic derivative, this becomes Definition \ref{sesquidef}.

It is clear that $\Delta_{l}$ takes an element of $\mathcal{M}_{l,ses}^{an}(\rho)$ to $\mathcal{M}_{l}^{hol}(\rho)$, but if $\Gamma$ has cusps then this image will be in $\mathcal{M}_{l}^{wh}(\rho)$ in general, by the growth condition imposed on elements of $\mathcal{M}_{l}^{an}(\rho)$ in Definition \ref{MFQMFdef}. By restricting to subspaces of $\mathcal{M}_{l,ses}^{an}(\rho)$ with stronger assumptions on the growth of the relevant parts, we can assure that their $\Delta_{l}$ images will be in $\mathcal{M}_{l}^{hol}(\rho)$ or in $\mathcal{M}_{l}^{cusp}(\rho)$. Results of \cite{[BDR]} (combined with \cite{[BFu]}), \cite{[JKK]}, and some others can be interpreted as showing that some scalar-valued holomorphic and weakly holomorphic modular forms of integral weight with respect to $SL_{2}(\mathbb{Z})$ are in the image of the appropriate restriction of $\Delta_{l}$. Using the vector-valued Poincar\'{e} series appearing in, e.g., \cite{[Ze2]}, generalized to arbitrary Fuchsian groups with cusps, one can prove surjectivity for some spaces of vector-valued modular forms as well. However, we shall only be needing special pre-images under $\Delta_{l}$ of the following meromorphic modular forms.
\begin{lem}
Take an integer $\upsilon\geq1$, a Fuchsian group $\Gamma$, and a representation $\rho$ of $\Gamma$. Then any element of $\mathcal{M}_{\upsilon+1}^{mer}(\rho)$ that is of the form $\delta_{1-\upsilon}^{\upsilon}\varphi=\partial_{\tau}^{\upsilon}\varphi$ for some $\varphi\in\mathcal{M}_{1-\upsilon}^{mer}(\rho)$ is the $\Delta_{\upsilon+1}$-image of a modular form $\Psi\in\mathcal{M}_{\upsilon+1,ses}^{sing}(\rho)$. Moreover, $\Psi$ can be taken such that $4y^{2}\partial_{\overline{\tau}}\Psi=\delta_{1-\upsilon}^{\upsilon-1}\varphi$, a case in which it is the sum of a meromorphic quasi-modular form of depth $\leq\upsilon$ and the expression $\sum_{p=1}^{\upsilon}\frac{(\upsilon-1)!}{(\upsilon-p)! \cdot p}\cdot\frac{\partial_{\tau}^{\upsilon-p}\varphi}{(-2iy)^{p}}$. \label{sesonholder}
\end{lem}

\begin{proof}
The paper \cite{[Ze1]}, and even its predecessor \cite{[A]}, shows that there exists an element $\mathbf{G}\in\mathcal{M}_{2}^{nm,\leq1}(\Gamma)$ (with the trivial representation of $\Gamma$) that is of the form $\tau \mapsto G(\tau)+\frac{1}{2iy}$, where $G$ is an element of $\widetilde{\mathcal{M}}_{2}^{mer,\leq1}(\Gamma)$ whose transformation law from Equation \eqref{QMFdef} is of the form $G(\gamma\tau)=j_{\gamma}^{2}(\tau)G(\tau)+j_{\gamma}'j_{\gamma}(\tau)$ (i.e., with $f_{1}=1$). Set $\Psi=-\sum_{h=1}^{\upsilon}\frac{(\upsilon-1)!}{(\upsilon-h)! \cdot h}\delta_{1-\upsilon}^{\upsilon-h}\varphi\cdot\mathbf{G}^{h}$, and recall that $\Delta_{\upsilon+1}$ is the composition $\delta_{\upsilon-1}\circ4y^{2}\partial_{\overline{\tau}}$. Now, $4y^{2}\partial_{\overline{\tau}}\mathbf{G}$ is $-1$ (since $G$ is meromorphic), and the proof of Corollary \ref{coeffLkdV} using Lemma \ref{LDeldelmer} shows that the image of $\delta_{1-\upsilon}^{\upsilon-h}\varphi$ under that operator is $h(\upsilon-h)\delta_{1-\upsilon}^{\upsilon-h-1}\varphi$ (this is well-defined also for $h=\upsilon$ since the coefficient vanishes then). Therefore $4y^{2}\partial_{\overline{\tau}}\Psi$ decomposes as the sum of $\sum_{h=1}^{\upsilon}\frac{(\upsilon-1)!}{(\upsilon-h)!}\delta_{1-\upsilon}^{\upsilon-h}\varphi\cdot\mathbf{G}^{h-1}$ and $-\sum_{h=1}^{\upsilon-1}\frac{(\upsilon-1)!}{(\upsilon-h-1)!}\delta_{1-\upsilon}^{\upsilon-h-1}\varphi\cdot\mathbf{G}^{h}$. A simple index change shows that the latter sum cancels with the elements associated with $1<h\leq\upsilon$ in the former one (this amounts to no cancelation at all if $\upsilon=1$), and we indeed obtain the value $\delta_{1-\upsilon}^{\upsilon-1}\varphi$ from the second assertion. The first assertion now follows after applying $\delta_{\upsilon-1}$.

We now turn to the explicit formula for $\Psi$ up to meromorphic expressions (recall that the defining property of $\Psi$ is its image under $4y^{2}\partial_{\overline{\tau}}$, and altering $\Psi$ by an element of $\mathcal{M}_{\upsilon+1}^{mer}(\rho)$ does not affect this property). We expand $\delta_{1-\upsilon}^{\upsilon-h}\varphi$ using Equation (56) of \cite{[Za]} (in which the internal product equals $(-1)^{m}\frac{(h+m-1)!}{(h-1)!}$) and $\mathbf{G}^{h}=\big(G+\frac{1}{2iy}\big)^{h}$ binomially, and since $(h-1)!$ cancels with most of the numerator of the binomial coefficient $\binom{h}{t}$ get that $\Psi(\tau)$ is $-\sum_{h=1}^{\upsilon}\sum_{m=0}^{\upsilon-h}\sum_{t=0}^{h}\frac{(\upsilon-1)!(-1)^{m}(h+m-1)!}{m!(\upsilon-h-m)!t!(h-t)!}\frac{\partial_{\tau}^{\upsilon-h-m}\varphi \cdot G(\tau)^{t}}{(2iy)^{m+h-t}}$. We set $p=m+h$, and find that for every pair of indices $1 \leq p\leq\upsilon$ and $0 \leq t \leq p$ the term $-\frac{(\upsilon-1)!(p-1)!}{(\upsilon-p)!t!}\cdot\frac{\partial_{\tau}^{\upsilon-p}\varphi \cdot G(\tau)^{t}}{(2iy)^{p-t}}$ is multiplied by $\sum_{h=\max\{1,t\}}^{p}\frac{(-1)^{p-h}}{(p-h)!(h-t)!}$, an expression that we write as $\frac{(-1)^{p-t}}{(p-t)!}\big[\sum_{h=t}^{p}(-1)^{h-t}\binom{p-t}{h-t}-\delta_{t,0}\big]$ using the term $\delta_{t,0}$ to cover for the fact that the required sum over $h$ starts from 1 when $t=0$. Now, the (modified) sum over $h$ reduces to $\delta_{t,p}$, and the contribution of these terms combines to the meromorphic expression $-\sum_{p=1}^{\upsilon}\frac{(\upsilon-1)!}{(\upsilon-p)! \cdot p}\partial_{\tau}^{\upsilon-p}\varphi \cdot G(\tau)^{p}$ (which is clearly an element of $\widetilde{\mathcal{M}}_{\upsilon+1}^{mer,\leq\upsilon}(\rho)$). The remaining terms, with $t=0$, are now easily seen to produce the asserted expression.
\end{proof}
The quasi-modular form $G$ from the proof of Lemma \ref{sesonholder} is called a \emph{automorphic integral of weight 2 and trivial multiplier system, with rational period functions} in the terminology of \cite{[Kn2]} (which would consider the weight as $-2$, actually). The simplest example of such a function is defined for $\Gamma=SL_{2}(\mathbb{Z})$, where we can take the holomorphic weight 2 quasi-modular Eisenstein series $E_{2}$, normalized such that $\mathbf{E}_{2}(\tau)=E_{2}(\tau)+\frac{1}{2iy}$ is modular of weight 2. For this $\Gamma$, when $\rho$ is the trivial representation, we illustrate the case $\upsilon=1$ in Lemma \ref{sesonholder} by choosing $\varphi$ to be the classical weakly holomorphic hauptmodul $j$. Then the associated function $\Psi$ can be taken as $-j\mathbf{E}_{2}$, the image of which under $4y^{2}\partial_{\overline{\tau}}$ is indeed $-4y^{2}j\cdot\partial_{\overline{\tau}}\frac{1}{2iy}=j$, so that $\Delta_{2}(-j\mathbf{E}_{2})=\delta_{0}j=j'$. Or course, $j$ can be replaced by any non-constant element of $\mathbb{C}(j)$ here, or alternatively $j'$ can be replaced by meromorphic modular form of weight 2 and level 1, provided that its constant term at the cusp vanishes. Similar concrete examples can be presented in other weights as well, using, e.g., the constructions from \cite{[P1]}, \cite{[P2]}, and others, which also describe the relations to the Eichler integrals, defined essentially in \cite{[E]}.

\begin{cor}
Elements of $\mathcal{M}_{\upsilon+1,ses}^{sing}(\rho)$ whose $\Delta$-images lie in the subspace $\partial_{\tau}^{\upsilon}\big(\mathcal{M}_{1-\upsilon}^{mer}(\rho)\big)$ of $\mathcal{M}_{\upsilon+1}^{mer}(\rho)$ decompose into a nearly meromorphic part and the non-holomorphic part of a harmonic Maa\ss\ form (which may have singularities as well). The space of functions satisfying the stronger equality from Lemma \ref{sesonholder} are the ones whose non-holomorphic Maa\ss\ form part vanishes, i.e., the elements of the intersection $\mathcal{M}_{\upsilon+1,ses}^{sing}(\rho)\cap\mathcal{M}_{\upsilon+1}^{nm}(\rho)$, which we denote by $\mathcal{M}_{\upsilon+1,ses}^{nm}(\rho)$. \label{sesnh}
\end{cor}

\begin{proof}
Since the difference between two elements of $\mathcal{M}_{\upsilon+1,ses}^{sing}(\rho)$ mapping to the same function $\partial_{\tau}^{\upsilon}\varphi$ is harmonic, every element $\Xi\in\mathcal{M}_{\upsilon+1}^{sing}(\rho)$ satisfying $\Delta_{\upsilon+1}\Xi=\partial_{\tau}^{\upsilon}\varphi$ is the sum of $\Psi$ from Lemma \ref{sesonholder} and an element of $\mathcal{M}_{\upsilon+1,0}^{sing}(\rho)$. But the description of the latter space in \cite{[BFu]} (the results of which easily extend to the vector-valued case with singularities, also without $\Gamma$ having cusps) shows that its elements consist of a meromorphic part and an non-meromorphic part. The meromorphic part does not affect the image under $4y^{2}\partial_{\overline{\tau}}$, while the image of the non-meromorphic part under that operator is an element of $\overline{\mathcal{M}_{1-\upsilon}^{mer}(\rho)}/y^{\upsilon-1}$ (as follows immediately from the $\xi$-image of that part being meromorphic). But elements of $\mathcal{M}_{\upsilon+1,0}^{sing}(\rho)$ without non-holomorphic parts lie in $\mathcal{M}_{\upsilon+1}^{mer}(\rho)$, and this space is contained in $\mathcal{M}_{\upsilon+1,0}^{nm}(\rho)$.
\end{proof}

\begin{rmk}
The equality between $\delta_{1-\upsilon}^{\upsilon}\varphi$ and $\partial_{\tau}^{\upsilon}(\varphi)$, implying the meromorphicity of $\delta_{1-\upsilon}^{\upsilon}\varphi$, is Bol's identity mentioned above. In addition, the proofs of Lemma \ref{sesonholder} and Corollary \ref{sesnh} show that elements of $\mathcal{M}_{\upsilon+1,ses}^{nm}(\rho)$ lie in $\mathcal{M}_{\upsilon+1}^{nm,\leq\upsilon}(\rho)$, and unless they are meromorphic (i.e., with trivial image under $4y^{2}\partial_{\overline{\tau}}$), their depth is precisely $\upsilon$. The depth of the meromorphic quasi-modular part from Lemma \ref{sesonholder} is also exactly $\upsilon$ wherever $\varphi$ is not a constant (note that $\varphi$ can be a non-zero constant if $\upsilon=1$). Note that the function $\Psi$ constructed in Lemma \ref{sesonholder} is not canonical, since it depends on the choice of $\mathbf{G}$, and changing $\mathbf{G}$ will alter $\Psi$ from Lemma \ref{sesonholder} by a non-trivial meromorphic function. We remark that by the results from \cite{[Ze1]} we can take, in case $\Gamma$ has cusps, the modular form $\mathbf{G}$ to be in $\mathcal{M}_{2}^{nh,\leq1}(\Gamma)$ (or equivalently $G\in\widetilde{\mathcal{M}}_{2}^{hol,\leq1}(\Gamma)$), so that for $\varphi\in\mathcal{M}_{1-\upsilon}^{hol}(\rho)$ we get $\Psi\in\mathcal{M}_{\upsilon+1,ses}^{an}(\rho)$ (and in $\mathcal{M}_{\upsilon+1}^{nh,\leq\upsilon}(\rho)$). If, furthermore, $\mathcal{M}_{2}^{hol}(\Gamma)=\{0\}$, then we can construct a canonical pre-image of $\partial_{\tau}^{\upsilon}\varphi$ in $\mathcal{M}_{\upsilon+1,ses}^{an}(\rho)$ under $\Delta_{\upsilon+1}$ in the method, but typically such a canonical pre-image does not exist (especially since for $\Gamma$ without cusps we cannot take $\mathbf{G}$ and $G$ without singularities---see, e.g., Theorem 4 of \cite{[A]}). \label{depsesE2}
\end{rmk}

\smallskip

Apart from the sesqui-harmonic modular forms from Definition \ref{sesquidef}, we shall require their images under weight raising operators.
\begin{lem}
Assume that $\Psi\in\mathcal{M}_{l,ses}^{sing}(\rho)$ and $\varphi\in\mathcal{M}_{l}^{mer}(\rho)$ satisfy $\Delta_{l}\Psi=\varphi$. Then we have the equality $\big(\Delta_{l+2m}+m(l+m-1)\big)\delta_{l}^{m}\Psi=\delta_{l}^{m}\varphi$ for every $m\in\mathbb{N}$. Moreover, if $m\geq1$ then $4y^{2}\partial_{\overline{\tau}}(\delta_{l}^{m}\Psi)$ equals $\delta_{l}^{m-1}\varphi-m(l+m-1)\delta_{l}^{m-1}\Psi$. \label{wrsesqui}
\end{lem}

\begin{proof}
Lemma \ref{LDeldelmer} shows that the combination $4y^{2}\partial_{\overline{\tau}}\delta_{r}$ is $\Delta_{r}-r$, and after composing with $\delta_{r}$ from the left it also follows that $\Delta_{r+2}\delta_{r}=\delta_{r}(\Delta_{r}-r)$, or equivalently $\delta_{r}\Delta_{r}=(\Delta_{r+2}+r)\delta_{r}$. We can now prove the first assertion by an induction similar to the one from the proof of Corollary \ref{coeffLkdV}, in which the case $m=0$ is true by definition. Assuming that this equality holds for $m$, we apply $\delta_{l+2m}$, giving us the desired right hand side with the sum of $\delta_{l+2m}\Delta_{l+2m}\delta_{l}^{m}\Psi$ and $m(l+m-1)\delta_{l}^{m+1}\Psi$ on the left hand side. But we have seen that $\delta_{l+2m}\Delta_{l+2m}$ can be written as $(\Delta_{l+2m+2}+l+2m)\delta_{l+2m}$, and adding the image of $\delta_{l}^{m}\Psi$ indeed gives the desired expression with $m+1$. This proves the first equality. Letting now $4y^{2}\partial_{\overline{\tau}}$ operate on $\delta_{l}^{m}\psi=\delta_{l+2m-2}(\delta_{l}^{m-1}\Psi)$ for some $m\geq1$, the beginning of our proof shows that we get $\big(\Delta_{l+2m-2}-(l+2m-2)\big)\delta_{l}^{m-1}\Psi$. But using the first equality, the first operator sends $\delta_{l}^{m-1}\Psi$ to $\delta_{l}^{m-1}\varphi-(m-1)(l+m-2)\delta_{l}^{m-1}\Psi$. Combining this with the second term yields the second assertion as well.
\end{proof}

\smallskip

There are two types of eigenfunctions of $\Delta_{k-\infty}^{V}$ using sesqui-harmonic modular forms. The simpler one is the following. Recall that if $k=d+1+e$ and $e<d$, then the operator $\delta_{k-2d}^{d-e}=\delta_{e+1-d}^{d-e}$ is just $\partial_{\tau}^{d-e}$ (Bol's identity again), hence it takes elements of $\mathcal{M}_{k-2d}^{mer}(\rho)=\mathcal{M}_{e+1-d}^{mer}(\rho)$ to elements $\mathcal{M}_{d+1-e}^{mer}(\rho)$.
\begin{prop}
Take a depth $d\geq1$, a weight $k=d+1+e$ for some integer $0 \leq e<d$, and an element $\varphi\in\mathcal{M}_{k-2d}^{mer}(\rho)$ with $\Gamma$ and $\rho$ as usual. Then there exists an element $F\in\mathcal{M}_{k-\infty,de}^{sing,V,d}(V_{\infty}\otimes\rho)$ whose presentation from Lemma \ref{Lapeval} is with $F_{d}=\varphi$. \label{s0ses}
\end{prop}

\begin{proof}
The proof of Proposition \ref{liftprobwt} shows that we can set $F_{s}=a_{d,s}^{(e)}\delta_{k-2d}^{d-s}\varphi$ for any $1 \leq s \leq d$, expressions whose finiteness for our value of $k$ is provided by Lemma \ref{redadsj}, and all the equalities with $1 \leq s \leq d+1$ in Corollary \ref{compeigen} will be satisfied. It therefore remains to find a function $F_{0}$ that will satisfy the equation with $s=0$, which now takes the form $(\Delta_{k}+de)F_{0}=-a_{d,1}^{(e)}\delta_{k-2d}^{d}\varphi$ (the right hand side is $-\delta_{k-2}F_{1}$). As the operator $\delta_{k-2d}^{d}=\delta_{e+1-d}^{d}$ on the right hand side decomposes as $\delta_{d+1-e}^{e}\circ\partial_{\tau}^{d-e}$, we seek a pre-image under $\Delta_{d+1+e}+de$ of a multiple of $\delta_{d+1-e}^{e}(\partial_{\tau}^{d-e}\varphi)$. But Lemma \ref{sesonholder}, with $\upsilon=d-e\geq1$, provides us with an element $\Psi\in\mathcal{M}_{d+1-e,ses}^{sing}(\rho)$ satisfying $\Delta_{d+1-e}\Psi=\partial_{\tau}^{d-e}\varphi$. Setting $l=d+1-e$ and $m=e$ in Lemma \ref{wrsesqui} (so that indeed $l+2m=d+1+e=k$), we find that taking $-a_{d,1}^{(e)}\delta_{d+1-e}^{e}\Psi$ for the function $F_{0}$ satisfies the desired equation. Now, the value of $a_{d,1}^{(e)}$ (for arbitrary $k$) is $(-1)^{e}(d-e)\frac{(d-e-1)!}{(d-1)!}\Big/\prod_{q=0}^{d-e-2}(k-2d+q)$. Indeed, either we recall the value $(-1)^{e}\frac{(d-e-1)!}{(d-1)!}\Big/\prod_{q=0}^{d-e-1}(k-2d+q)$ of $a_{d,0}^{(e)}$ from Lemma \ref{redadsj} and use the equality $a_{d,1}^{(e)}=(d-e)(k-1-d-e)a_{d,0}^{(e)}$ from Corollary \ref{coeffLkdV} (with $\lambda=e(k-1-e)$), or we just note that with $s=1$ the numbers $b_{d,1,0}$ and $b_{d,1,1}$ from the proof of Lemma \ref{redadsj} are 0 and 1 respectively. As for $k=d+1+e$, the product over $q$ is just $(-1)^{d-e-1}(d-e-1)!$, so that the coefficient $-a_{d,1}^{(e)}$ of $\delta_{d+1-e}^{e}\Psi$ is just $(-1)^{d}\frac{(d-e)}{(d-1)!}$. Now just add this $F_{0}$ to the definition of $F$.
\end{proof}

\begin{rmk}
We cannot present $F$ here as a lift of $\varphi$ into $\mathcal{M}_{k-\infty,de}^{sing,V,d}(V_{\infty}\otimes\rho)$, since it depends on the choice of $\Psi$ (though see Remark \ref{depsesE2} above for a case where this becomes possible). On the other hand, $\Psi$ determines $\partial_{\tau}^{d-e}\varphi$, which in turn determines $\varphi$ itself up to polynomials of degree $d-e-1$. As the only polynomials that are modular are the constant ones, of weight 0 (this is essentially the statement that the only representation $V_{m}$ containing invariant vectors is the trivial representation $V_{0}$), we find that $\Psi$ determines $\varphi$ unless $e=d-1$ and $k-2d=0$. We can therefore write, at least if $e<d-1$ (see Remark \ref{consteigen} below for the remaining case), the modular form $F$ from Proposition \ref{s0ses} as $\mathcal{L}_{d+1+e,d}^{V,e}(\Psi)$ (but recalling that its corresponding function $F_{d}$ is \emph{not} $\Psi$ but $\varphi$). Note that if we insist on the extra condition that $4y^{2}\partial_{\overline{\tau}}\Psi=\delta_{1-\upsilon}^{\upsilon-1}\varphi$ from Lemma \ref{sesonholder} (i.e., consider elements of the space $\mathcal{M}_{d+1-e,ses}^{nm}(\rho)$ from Corollary \ref{sesnh}), then when $\upsilon=d-e$ is 1 the function $\Psi$ still determines $\varphi$ as its image under $4y^{2}\partial_{\overline{\tau}}$. We also remark that the eigenvalue $de$ in Proposition \ref{s0ses} corresponds, with $k=d+1+e$, to the limit parameter $j$ from Theorem \ref{eigenVgen} and Proposition \ref{liftprobwt} being $e$, the maximal value of $j$ for which no lift can defined by that proposition. In view of Corollary \ref{eigenVtot} below, we observe that with these values of $k$ and $j$ we have $j_{-}=e$ and $j_{+}=d$. \label{liftpair}
\end{rmk}

\smallskip

For the remaining eigenfunctions, not appearing in Theorem \ref{eigenVgen} and Propositions \ref{liftprobwt} and \ref{s0ses}, let us first consider elements $\varphi\in\mathcal{M}_{k-2d}^{mer}(\rho)$ that are themselves $\Delta_{k-2d}$-images from $\mathcal{M}_{k-2d,ses}^{sing}(\rho)$ (though we will be able to remove this restriction later). More explicitly, if $\Xi\in\mathcal{M}_{k-2d,ses}^{sing}(\rho)$ satisfies $\Delta_{k-2d}\Xi=\varphi$, then we are looking at elements $F\in\mathcal{M}_{k-\infty}^{sing}(V_{\infty}\otimes\rho)$ whose expansion from Lemma \ref{Lapeval} (with the depth $d\geq1$) consists of functions $F_{s}$ that are linear combinations of $\delta_{d-2k}^{d-s}\varphi$ and $\delta_{d-2k}^{d-s}\Xi$. Moreover, the latter function appears only if $s<d$, since we recall that when looking for eigenfunctions the equation with $s=d+1$ in Corollary \ref{compeigen} implies the meromorphicity of $F_{d}$ (this is how we defined $\varphi$ to begin with). Combining the expressions from Lemmas \ref{LDeldelmer} and \ref{wrsesqui}, we obtain the following analogue of Corollary \ref{coeffLkdV}.
\begin{cor}
Let $d\geq1$ and $\Gamma$ be as usual, take some weight $k$ and an appropriate multiplier system $\rho$, and consider a function $F\in\mathcal{M}_{k-\infty}^{sing}(V_{\infty}\otimes\rho)$ and an element $\varphi\in\mathcal{M}_{k-2d}^{mer}(\rho)$ that is the $\Delta_{k-2d}$-image of the sesqui-harmonic modular form $\Xi\in\mathcal{M}_{k-2d,ses}^{sing}(\rho)$. Assume that there is an integer $1 \leq e<d$ such that the function $F_{s}$ from Lemma \ref{Lapeval} is of the form $a_{s}\delta_{k-2d}^{d-s}\varphi$ if $e<s \leq d$, with $a_{d}=1$, but for $0 \leq s \leq e$ it is of the form $b_{s}\delta_{k-2d}^{d-s}\Xi+c_{s}\delta_{k-2d}^{d-s}\varphi$, and $b_{e}\neq0$. Then the function $-F_{s,\Delta}^{V}$ from that lemma equals
\[\big\{(d-s)(k-1-d-s)a_{s}-(s+1)a_{s+1}-(s-1)\big[(d-s+1)(k-d-s)a_{s-1}-sa_{s}\big]\big\}\delta_{k-2d}^{d-s}\varphi\]
for $e+1<s \leq d$, while for $0 \leq s<e$ it is the sum of
\[\big\{(d-s)(k-1-d-s)b_{s}-(s+1)b_{s+1}-(s-1)\big[(d-s+1)(k-d-s)b_{s-1}-sb_{s}\big]\big\}\delta_{k-2d}^{d-s}\Xi\]
and
\[\big\{(d-s)(k-1-d-s)c_{s}-(s+1)c_{s+1}-(s-1)\big[(d-s+1)(k-d-s)c_{s-1}-sc_{s}\big]+\] \[+(s-1)b_{s-1}-b_{s}\big\}\delta_{k-2d}^{d-s}\varphi,\] where for $s=0$ we define $b_{-1}$ and $c_{-1}$ to be 0. The function with $s=e$ is as in the latter expression, but with the term with $b_{e+1}$ omitted and the coefficient $c_{e+1}$ replaced by $a_{e+1}$. On the other hand, for $-F_{e+1,\Delta}^{V}$ we get the sum of
\[(d-e-1)(k-2-d-e)a_{e+1}-(e+2)a_{e+2}-e\big[(d-e)(k-d-e-1)c_{e}-(e+1)a_{e+1}\big]+eb_{e}\] times $\delta_{k-2d}^{d-e-1}\varphi$, and $-e(d-e)(k-d-e-1)b_{e}\delta_{k-2d}^{d-e-1}\Xi$. \label{expwithPsi}
\end{cor}

\begin{proof}
The images of the multiples of $\delta_{k-2d}^{d-s}\varphi$ under $\Delta_{k-2s}$ and of the multiples of $\delta_{k-2d}^{d-s+1}\varphi$ under $4y^{2}\partial_{\overline{\tau}}$ were evaluated in the proof of Corollary \ref{coeffLkdV}. On the other hand, Lemma \ref{wrsesqui} provides us with the values of the expressions $\Delta_{k-2s}\delta_{k-2d}^{d-s}\Xi$ and $4y^{2}\partial_{\overline{\tau}}\delta_{k-2d}^{d-s+1}\Xi$. Gathering the resulting multiples of $\delta_{k-2d}^{d-s}\varphi$ and $\delta_{k-2d}^{d-s}\Xi$ from all the terms appearing in Lemma \ref{Lapeval}, and observing for which index we have to use which formula, we obtain the asserted expressions.
\end{proof}

We can now determine the form of the remaining eigenfunctions. Since it will be difficult to obtain an explicit expression for the functions in general, we shall bypass this difficulty using the following lemma.
\begin{lem}
For $\Gamma$ and $\rho$ as usual, take a weight $l$, and assume that the element $F\in\mathcal{M}_{l-\infty}^{sing}(V_{\infty}\otimes\rho)$ is of the form described in Corollary \ref{expwithPsi}, with $k=l$. Then the element $\tilde{\delta}_{l}F$ of $\mathcal{M}_{l+2-\infty}^{sing}(V_{\infty}\otimes\rho)$ is also of this form, with the weight $k=l+2$, the parameters $d+1$ and $e+1$, and its component with maximal index $d+1$ is $-d\varphi$. Moreover, if $F$ is in $\mathcal{M}_{l-\infty,\lambda}^{sing,V}(V_{\infty}\otimes\rho)$, then $\tilde{\delta}_{l}F\in\mathcal{M}_{l+2-\infty,\lambda+l}^{sing,V}(V_{\infty}\otimes\rho)$. \label{solextind}
\end{lem}

\begin{proof}
The evaluation of $\tilde{\delta}_{l}F$ is carried out explicitly in the proof of Lemma \ref{Lapeval}. As the new component with index $s$ is $\delta_{l-2s}F_{s}+(1-s)F_{s-1}$, the form from Corollary \ref{expwithPsi} and the fact that $d>e\geq1$ immediately imply the first assertion (including the formula for the component with index $d+1$). For the second one, note that the assertions of Lemma \ref{LDeldelmer} continue to hold if we replace the spaces of the form $\mathcal{M}_{l,\mu}^{sing}(\rho)$ by the corresponding spaces $\mathcal{M}_{l-\infty,\mu}^{sing,V}(V_{\infty}\otimes\rho)$ and the operator $\delta_{l}$ by the operator $\tilde{\delta}_{l}$ from Corollary \ref{comim}. This can either be seen using part $(vi)$ of Proposition \ref{comrels} (showing that the commutation relation between the $\tilde{\delta}_{l}$'s and $y^{2}\partial_{\overline{\tau}}$ are the same as for $\delta_{l}$), or via Theorem \ref{difopShim} and the relations with $\widetilde{\mathcal{M}}_{l}^{sing}(\rho)$ (after the limit map from Corollary \ref{comim}).
\end{proof}

\begin{thm}
A function $F$ of the form described in Corollary \ref{expwithPsi} can be an eigenfunction for some eigenvalue $\lambda$ only if the weight $k$ is $d+1+e$ and $\lambda$ is one of the distinct $e$ numbers $j(k-1-j)$ with $0 \leq j<e$. With these parameters we do not need $\varphi$ to have a pre-image $\Xi$ under $\Delta_{k-2d}=\Delta_{e+1-d}$, but it suffices to take $\Psi\in\mathcal{M}_{d+1-e,ses}^{nm}(\rho)$ with $4y^{2}\partial_{\overline{\tau}}\Psi=\delta_{e+1-d}^{d-e-1}\varphi$ as in Lemma \ref{sesonholder}, and replace each $\delta_{k-2d}^{d-s}\Xi$ with $s \leq e$ in Corollary \ref{expwithPsi} by $\delta_{d+1-e}^{e-s}\Psi$. The number $a_{s}$ is then $a_{d,s}^{(j)}$ from Equation \eqref{adsj} for any $s>e$, while for $s \leq e$ the number $b_{s}$ is also the expression from Equation \eqref{adsj}, but multiplied by $e-d$ and with the value $q=d-e-1$ omitted from the product in the denominator. Moreover, the parameter $c_{e}$ is free, and its value determines all the coefficients $c_{s}$ with $0 \leq s<e$. \label{lasteigenV}
\end{thm}

\begin{proof}
For $F$ to be an eigenfunction of $\Delta_{k-\infty}^{V}$ we have to compare each $-F_{s,\Delta}^{V}$, the expressions for which are given in Corollary \ref{expwithPsi}, with $\lambda F_{s}$. Since $F_{e+1}$ does not involve $\Xi$ but $-F_{e+1,\Delta}^{V}$ has the term $-e(d-e)(k-d-e-1)b_{e}\delta_{k-2d}^{d-e-1}\Xi$, and we know that neither $e$ nor $d-e$ can vanish, equality may hold only if the weight is the asserted one. An application of Lemma \ref{sesonholder} with $\upsilon=d-e$ thus provides us with the existence of $\Psi$, and allows us to write $\delta_{d+1-e}^{e-s}\Psi$ instead of each $\delta_{k-2d}^{d-s}\Xi$.

Next, we prove that the polynomial in $\lambda$ determining the eigenvalues for a general weight in Theorem \ref{eigenVgen} is the same one that we obtain here. To see this, we forget the known value of $k$ for a moment (but substitute it back later), and recall that this polynomial was obtained by the determination of the numbers $a_{s}$ in decreasing order via Equation \eqref{recexp} and using the equality with $s=1$ in Corollary \ref{coeffLkdV} involving $a_{1}$ and $a_{2}$ alone. Now, in our case we can express $a_{s-1}$ for $s>e+1$ via Equation \eqref{recexp}, since the functions $F_{s}$ with $s>e$ have the same form. The expression for $a_{e}$ in that equation, with $s=e+1$, involves the denominator $e(d-e)(k-d-e-1)$, which with our weight vanishes. But when we consider the expression with $s=e+1$ in Corollary \ref{expwithPsi}, the parameter appearing there in the place of $a_{e}$ is $b_{e}$, whose resulting value is $\frac{[(d-e-1)(k-2-d-e)+e(e+1)-\lambda]a_{e+1}-(e+2)a_{e+2}}{-e}$ plus $(d-e)(k-d-e-1)c_{e}$. The multiplier in front of $c_{e}$ vanishes, so that this expression determines the value of $b_{e}$ in the same manner from Theorem \ref{eigenVgen}, but with an extra multiplier of $e-d$ and with the problematic difference $k-d-e-1$ omitted from the denominator. Therefore $b_{e}$ indeed takes the asserted value.

For the other parameters $b_{s}$ with $s<e$ we use the equalities between the coefficients multiplying $\delta_{k-2d}^{d-s}\Xi$ (or $\delta_{d+1-e}^{e-s}\Psi$) in $\lambda F_{s}$ and the expression from Corollary \ref{expwithPsi} for $-F_{s,\Delta}^{V}$. Letting $k$ be general again for a moment, the resulting expression for $b_{e-1}$ is the same as in Equation \eqref{recexp}, but with $a_{e}$ replaced by $b_{e}$ and the term with $a_{e+1}$ missing. Observing that the ratio between $b_{e}$ and $a_{e}$ is the coefficient $(d-e)(k-d-e-1)$, the product of $a_{e+1}$ (which is finite also for our value of $k$) by this coefficient vanishes for $k=d+1+e$. It follows (by comparing with Theorem \ref{eigenVgen}) that $b_{e-1}$ attains the value of $(e-d)a_{e-1}$ from that theorem with the (vanishing) term $k-d-e-1$ omitted from the denominator. The expressions for $b_{s-1}$ with $1<s \leq e$ arising from the comparison of the parts involving $\delta_{k-2d}^{d-s}\Xi=\delta_{d+1-e}^{e-s}\Psi$ in $-F_{s,\Delta}^{V}$ and $\lambda F_{s}$ is as in Equation \eqref{recexp} (with every $a$ replaced by $b$), and the relation from the equality with $s=1$ is again as in the proof of Theorem \ref{eigenVgen} (with the same modification). We therefore indeed get, up to a global multiplying coefficient, the same polynomial in $\lambda$ from that theorem, and its roots are as described there.

We now follow the proof of Proposition \ref{liftprobwt}, but noting the following difference. The same numbers $a_{s}$ with $s>e$ are still defined via Equation \eqref{recexp}, but the expression we used here for establishing the value of $b_{e}$ (with the explicit value of $k$) is precisely the one that is assumed to vanish in that proposition (indeed, they come from essentially the same equation, but in our case we have the extra contribution with $b_{e}$). It follows that the roots for which $b_{e}\neq0$ (as we assume in Corollary \ref{expwithPsi}) are those from Theorem \ref{eigenVgen} that do not appear as roots in Proposition \ref{liftprobwt}, which were seen in that proposition to be those with $0 \leq j<e$. Moreover, they are all distinct (see Remark \ref{deptheigen}). In addition, for every such $j$ the verifications from Theorem \ref{eigenVgen} using the values from Equation \eqref{adsj} and the relation between our value of $b_{s}$ with $0 \leq s \leq e$ and the one of the associated number $a_{s}$ from that theorem determine $b_{s}$ as the asserted value, which is finite since the omitted value of $q$ is precisely the one that causes the denominator of that equation to vanish for $k=d+1+e$ and $s \leq e$. In particular, any relation that holds between the numbers $a_{d,s}^{(j)}$ with $s \leq e$ is also valid between the numbers $b_{s}$.

We have therefore established the value of the weight, the possible eigenvalues, and the values of the parameters $a_{s}$ and $b_{s}$. We have already seen that $c_{e}$ comes with a vanishing coefficient in the equality arising from $-F_{e+1,\Delta}^{V}$ in Corollary \ref{expwithPsi}, but we have to show that a solution $\{c_{s}\}_{s=0}^{e}$ to the equations comparing the coefficients of $\delta_{k-2d}^{d-s}\varphi$ in $-F_{s,\Delta}^{V}$ and $\lambda F_{s}$ exists. While these are $e+1$ equations in $e+1$ indeterminates, our values of $\lambda$ are precisely those for which the resulting matrix is singular. We shall prove the existence of such coefficients directly only for $j=0$, where we take $c_{s}=b_{s}\sum_{m=1}^{e-s}\frac{1}{m(d-e+m)}$ (this is 0 for $s=e$). Then $(s-1)(d+1-s)(e+1-s)c_{s-1}$ involves a sum up to $e-s+1$, but a cancelation with the term $(s-1)b_{s-1}$ reduces this sum to go only up to $e-s$. We add and subtract $\frac{(s+1)b_{s+1}}{(e-s)(d-s)}$, and obtain the relation among the numbers $b_{s}$ (multiplied by the scalar $\sum_{m=1}^{e-s}\frac{1}{m(d-e+m)}$), which vanishes by our value of $\lambda$. The remaining expression is $\frac{(s+1)b_{s+1}}{(k-1-d-s)(d-s)}-b_{s}$. But as this difference vanishes for the numbers $a_{d,s}^{(j)}$ (a simple examination of Equation \eqref{adsj} with $j=0$ reveals this), it vanishes also here, and we have a solution for $j=0$.

This proves the existence of such eigenfunctions with $j=0$. Taking now a general index $j<e<d$, we have an element of $\mathcal{M}_{k-2j-\infty,0}^{sing,V}(V_{\infty}\otimes\rho)$ with depth $d-j$, that takes the form from Corollary \ref{expwithPsi} with the parameter $e-j$ (indeed, the weight $k-2j$ equals $(d-j)+1+(e-j)$). Applying Lemma \ref{solextind} shows that the image of our element under $(-1)^{j}\frac{(d-1-j)!}{(d-j)!}\tilde{\delta}_{k-2j}^{j}$ lies in $\mathcal{M}_{k-\infty,0}^{sing,V}(V_{\infty}\otimes\rho)$, is of the required form with the parameter $e$, and has depth $d$ with $F_{d}=\varphi$. The additional multipliers appearing in the component of maximal index in that lemma are canceled with our external coefficient. Moreover, it follows from that lemma that the function under consideration is an eigenfunction of $\Delta_{k-\infty}^{V}$, and the eigenvalue, which is the sum of the numbers $k-2j+2r$ with $0 \leq r<j$, was seen to be the required value $j(k-1-j)$.
\end{proof}
Choosing the free coefficient $c_{e}$ from Theorem \ref{lasteigenV} to be 0 determines $F$ completely from $\varphi$ and $\Psi$, and we have seen in Remark \ref{liftpair} that $\Psi$ determines $\varphi$ (here also when $e=d-1$, since we consider just $\Psi\in\mathcal{M}_{d+1-e,ses}^{nm}(\rho)$). We can therefore write $F$ (with $c_{e}=0$) as $\mathcal{L}_{d+1+e,d}^{V,j}(\Psi)$ also in this case. Note that the construction with $j>0$, using the operator $(-1)^{j}\frac{(d-1-j)!}{(d-j)!}\tilde{\delta}_{k-2j}^{j}$, does not preserve the vanishing of $c_{e}$. The lift $\mathcal{L}_{d+1+e,d}^{V,j}(\Psi)$ is therefore not a multiple of $\tilde{\delta}_{k-2j}^{j}\big(\mathcal{L}_{d+1+e-2j,d-j}^{V,0}(\Psi)\big)$, but a linear combination of the latter function and the lift $\mathcal{L}_{d+1+e,e}^{V,j}(\varphi)$.

\smallskip

We can now gather the results about all the eigenforms of $\Delta_{k-\infty}^{V}$. We denote the lower integral part of $\frac{k-1}{2}$ (which is $\frac{k-1}{2}$ for odd $k$ and $\frac{k-2}{2}$ for even $k$) by $\big\lfloor\frac{k-1}{2}\big\rfloor$, and the upper integral part of $\frac{k}{2}$ (this is $\frac{k}{2}$ if $k$ is even and $\frac{k+1}{2}$ if it is odd) by $\big\lceil\frac{k}{2}\big\rceil$.
\begin{cor}
Fix a group $\Gamma$, a weight $k$, a multiplier system $\rho$, and an eigenvalue $\lambda$, and consider the number of natural solutions $j$ to the equality $\lambda=j(k-1-j)$.
\begin{enumerate}
\item If there are none, then $\mathcal{M}_{k-\infty,\lambda}^{sing,V}(V_{\infty}\otimes\rho)$ is just the space $\mathcal{L}_{k,0}^{V,\lambda}\big(\mathcal{M}_{k,\lambda}^{sing}(\rho)\big)$ from Proposition \ref{eigenVd0}.
\item In case there is a single solution $j$ to the equation in question, the space $\mathcal{M}_{k-\infty,\lambda}^{sing,V}(V_{\infty}\otimes\rho)$ is obtained from $\mathcal{L}_{k,0}^{V,\lambda}\big(\mathcal{M}_{k,\lambda}^{sing}(\rho)\big)$ by adding the direct sum $\bigoplus_{d=j+1}^{\infty}\mathcal{L}_{k,d}^{V,j}\big(\mathcal{M}_{k-2d}^{mer}(\rho)\big)$ of the spaces from Theorem \ref{eigenVgen}.
\item Assume now that there are two distinct solutions $j_{+}>j_{-}$, so that the weight $k$ is $j_{+}+j_{-}+1$. Then $\mathcal{M}_{k-\infty,\lambda}^{sing,V}(V_{\infty}\otimes\rho)$ is the direct sum of $\bigoplus_{d=j_{+}+1}^{\infty}\mathcal{L}_{k,d}^{V,j_{-}}\big(\mathcal{M}_{k-2d}^{mer}(\rho)\big)$, the spaces $\bigoplus_{d=\lceil k/2 \rceil}^{j_{+}-1}\!\mathcal{L}_{k,d}^{V,j_{-}}\big(\mathcal{M}_{2d+2-k,ses}^{nm}(\rho)\big)$ from Theorem \ref{lasteigenV}, and the space from Proposition \ref{s0ses}. The latter space is denoted by $\mathcal{L}_{k,j_{+}}^{V,j_{-}}\big(\mathcal{M}_{j_{+}+1-j_{-},ses}^{sing}(\rho)\big)$ when $j_{+}>j_{-}-1$ by Remark \ref{liftpair}, and if $j_{-}=j_{+}-1$ then elements of this space depend on pairs $(\Psi,\varphi)$ with $\Psi\in\mathcal{M}_{2,ses}^{sing}(\rho)$ and $\varphi\in\mathcal{M}_{0}^{mer}(\rho)$ with $\Delta_{2}\Psi=\partial_{\tau}\varphi$.
\end{enumerate} \label{eigenVtot}
\end{cor}

\begin{proof}
Proposition \ref{eigenVd0} shows that the space $\mathcal{L}_{k,0}^{V,\lambda}\big(\mathcal{M}_{k,\lambda}^{sing}(\rho)\big)$ is always contained in $\mathcal{M}_{k-\infty,\lambda}^{sing,V}(V_{\infty}\otimes\rho)$. Moreover, if a solution $j$ to the equality with $\lambda$ exists, then Theorem \ref{eigenVgen} shows that for large enough $d$ (this means $d>j$ in case $j$ is the unique solution and $d \geq k$ if two solutions exist), the space $\mathcal{L}_{k,d}^{V,j}\big(\mathcal{M}_{k-2d}^{mer}(\rho)\big)$ is contained in $\mathcal{M}_{k-\infty,\lambda}^{sing,V}(V_{\infty}\otimes\rho)$ as well. Assuming that there are two solutions $j_{+}>j_{-}$, Remark \ref{deptheigen} explains why we use $j_{-}$ in the notation of the lift, and the extension of $\mathcal{L}_{k,d}^{V,j_{-}}$ to indices $j_{+}<d \leq k$ is the content of Proposition \ref{liftprobwt}. The space from Proposition \ref{s0ses}, the notation for which is explained in Remark \ref{liftpair} (which also shows that this is the case with $d=j_{+}$ and $e=j_{-}$), contains $\mathcal{L}_{k,0}^{V,\lambda}\big(\mathcal{M}_{k,\lambda}^{sing}(\rho)\big)$ as the elements in which $\Psi$ is in the subspace $\mathcal{M}_{j_{+}+1-j_{-},0}^{sing}(\rho)$ of $\mathcal{M}_{j_{+}+1-j_{-},ses}^{sing}(\rho)$ (to which we attach the constant $\varphi=0$ if $j_{-}=j_{+}-1$). Indeed, it follows from Lemma \ref{LDeldelmer} that $\delta_{j_{+}+1-j_{-}}^{j_{-}}$ defines an isomorphism between $\mathcal{M}_{j_{+}+1-j_{-},0}^{sing}(\rho)$ and $\mathcal{M}_{j_{+}+j_{-}+1,j_{+}j_{-}}^{sing}(\rho)$. Finally, the conditions on $d$ and $e$ in Theorem \ref{lasteigenV}, which define $\mathcal{L}_{k,d}^{V,j_{-}}\big(\mathcal{M}_{d+1-e,ses}^{nh}(\rho)\big)$ as a subspace of $\mathcal{M}_{k-\infty,\lambda}^{sing,V}(V_{\infty}\otimes\rho)$, are equivalent to $\big\lceil\frac{k}{2}\big\rceil \leq d<j_{+}$. Since $k=d+1+e$, the weight $d+1-e$ is therefore indeed $2d+2-k$. Note that while Remark \ref{deptheigen} indicates the existence of the spaces $\mathcal{L}_{k,e}^{V,j}\big(\mathcal{M}_{k-2e}^{mer}(\rho)\big)$ for $j_{-}<e\leq\big\lfloor\frac{k-1}{2}\big\rfloor$ (where $k>2d$), such elements are obtained as $\mathcal{L}_{k,d}^{V,j_{-}}(\Psi)$ for $\Psi$ in the subspace $\mathcal{M}_{d+1-e}^{mer}(\rho)$ of $\mathcal{M}_{d+1-e,ses}^{nh}(\rho)$, for $d=k-1-e>e$. The sum is a direct sum since elements of different lifts have different depths (note that for integral $k$, if the parameter $d$ satisfies $\big\lceil\frac{k}{2}\big\rceil \leq d \leq j_{+}$, then we have seen that the depth is not $d$ when $\Delta_{2d+2-k}\Psi$ is trivial).

It remains to show why these are the only eigenfunctions of $\Delta_{k}^{V}$. The fact that the only eigenfunctions of depth 0 are the images of $\mathcal{L}_{k,0}^{V,\lambda}$ is obvious from the proof of Proposition \ref{eigenVd0}. We may thus consider only eigenfunctions of depth $d\geq1$. Let us assume that the direct sum in question produces all the eigenfunctions in $\mathcal{M}_{k-\infty,\lambda}^{sing,V}(V_{\infty}\otimes\rho)$ having depth $<d$ (a statement we just proved for $d=1$), and take an element $F\in\mathcal{M}_{k-\infty,\lambda}^{sing,V,d}(V_{\infty}\otimes\rho)$, expanded as in Lemma \ref{Lapeval}. Then Corollary \ref{compeigen} shows that $F_{d}=\varphi_{d}$ is meromorphic, and every function $F_{s-1}$ with $1<s<d$ is determined by $F_{s}$ and $F_{s+1}$ up to meromorphic functions (because the equality associated with $s$ in Corollary \ref{compeigen} involves $F_{s-1}$ only via its image under $4y^{2}\partial_{\overline{\tau}}$). Assume first that $k$ is not a positive integer. Then Lemma \ref{LDeldelmer} implies that a natural $4y^{2}\partial_{\overline{\tau}}$-pre-image of a function of the form $\delta_{k-2t}^{t-s}\varphi_{t}$ with $s \leq t \leq d$ is the appropriate multiple of $\delta_{k-2t}^{t+1-s}\varphi_{t}$, with eigenvalue $(t+1-s)(k-t-s)$. Therefore $F_{s}$ with $1 \leq s \leq d$ must be a linear combination of functions of the sort $\sum_{t=s}^{d}\alpha_{d,s,t}\delta_{k-2t}^{t-s}\varphi_{t}$ with $\varphi_{t}\in\mathcal{M}_{k-2t}^{mer}(\rho)$. Moreover, in this sum the summand associated with $t$ has the eigenvalue $(t-s)(k-1-t-s)$ with respect to $\Delta_{k-2s}$, and all these eigenvalues are distinct. Concentrating on the eigenfunction with eigenvalue $(d-s)(k-1-d-s)$ (associated with $\varphi_{d}$), we arrive at the same equalities from Theorem \ref{eigenVgen}. Therefore the eigenvalue is one of those from that theorem, our element $F$ is $\mathcal{L}_{k,d}^{V,j}(\varphi_{d})$ plus an element of smaller depth, and we are done. If $k$ is a positive integer, then at some point one of the functions $\delta_{k-2t}^{t-s}\varphi_{t}$ may not be a pre-image, but we can then replace it by the appropriate sesqui-harmonic form. Using Proposition \ref{s0ses} or Theorem \ref{lasteigenV} we again find that $F$ is a lift $\mathcal{L}_{k,d}^{V,j}(\Psi)$ plus an element of smaller depth, which proves our claim also in this case.
\end{proof}

The functions appearing in the sequences $\{F_{s}\}_{s}$ associated via Theorem \ref{QMFVVMF} to the elements of $\mathcal{M}_{k-\infty,\lambda}^{sing,V}(V_{\infty}\otimes\rho)$ considered in Corollary \ref{eigenVtot} are reminiscent of the modular forms called \emph{almost harmonic Maa\ss\ forms} in \cite{[BFo]} (which in our terminology would be called \emph{nearly harmonic Maa\ss\ forms}). Indeed, our functions are sums of images of meromorphic modular forms (of different weights) under powers of weight raising operators. However, that reference requires that the weight raising operators will always operate on the same modular form (which can more generally be a harmonic weak Maa\ss\ form), but on the other hand allows multiplication by other nearly holomorphic modular forms, thus producing, in general, objects that are different from ours.

\begin{rmk}
The dependence, mentioned in Remark \ref{liftpair}, of an element of the space from Proposition \ref{s0ses} on $\varphi$ up to constants when $k=2d$ and $\lambda=d(d-1)$ (which means $j_{+}=d$, and $j_{-}=d-1$) appears only in the function $F_{d}$ (since $F_{s}$ with $1 \leq s<d$ is a multiple of $\delta_{0}^{d-s}\varphi=\delta_{2}^{d-1-s}(\partial_{\tau}\varphi)$. This is related to the fact, which one easily verifies directly using Corollary \ref{compeigen}, that the function sending $\tau$ to $\frac{1}{(-2iy)^{d}}\binom{\tau}{1}^{\infty-d}\binom{\overline{\tau}}{1}^{d}$ lies in $\mathcal{M}_{k-\infty,d(d-1)}^{sing,V,d}(V_{\infty}\otimes\rho)$. We can write this function as the image under $\mathcal{L}_{2d,d}^{V,d-1}$ of the pair with $\Psi=0$ and $\varphi=1$. As a related example, consider the element of $\mathcal{M}_{1}^{mer}(\Gamma,V_{1})$ denoted $w$ in Equation (13) of \cite{[Ze1]}, which in our notation sends $\tau$ to $\mathbf{G}(\tau)\binom{\tau}{1}+\frac{1}{2iy}\binom{\overline{\tau}}{1}=G(\tau)\binom{\tau}{1}+\binom{1}{0}$. Its image in the limit of the $i_{m}$'s is meromorphic, hence with eigenvalue 0 (with $j=0$), and it can be described as the lift $\mathcal{L}_{2,1}^{V,0}$ of the pair with $\Psi=\mathbf{G}$ and $\varphi=1$. Observing that $\mathbf{G}$ is harmonic (since $\Delta_{2}$ is $\delta_{0}\circ4y^{2}\partial_{\overline{\tau}}$ and $\delta_{0}=\partial_{\tau}$ annihilates $4y^{2}\partial_{\overline{\tau}}\mathbf{G}=1$), subtracting its image under $\mathcal{L}_{2,0}^{V,0}$ we again end up with the pair $\Psi=0$ and $\varphi=1$. This meromorphicity is, in fact, a special case of a much more general result---see Theorem \ref{harmmer} below. \label{consteigen}
\end{rmk}

\subsection{Simpler Description of the Eigenforms}

We recall that eigenfunctions of $\Delta_{k-\infty}^{V}$ in $\mathcal{M}_{k-\infty}^{sing}(V_{\infty}\otimes\rho)$ translate to eigenfunctions of $\Delta_{k}^{V}$ in $\widetilde{\mathcal{M}}_{k}^{sing}(\rho)$. In the realm of quasi-modular forms we have some natural candidates for eigenfunctions, which should be obtained using what we have proved. Indeed, since $\Delta_{k}^{V}$ is just $\Delta_{k}$ on the subspace $\mathcal{M}_{k}^{sing}(\rho)$ of $\widetilde{\mathcal{M}}_{k}^{sing}(\rho)$, elements of $\widetilde{\mathcal{M}}_{k,\lambda}^{sing}(\rho)$ are eigenfunctions of $\Delta_{k}^{V}$ as well. Moreover, the fact that $\Delta_{k}^{V}$ annihilates meromorphic functions shows that all the elements of $\widetilde{\mathcal{M}}_{k}^{mer}(\rho)$ are $\Delta_{k}^{V}$-harmonic. Moreover, considering $\Delta_{k}^{V}=\Delta_{k}$ as an operator on functions on $\mathcal{H}$, a repeated application of the commutation relations from Lemma \ref{LDeldelmer} (as done in the proof of Corollary \ref{coeffLkdV}) shows that applying $\delta_{k-2j}^{j}$ to an element of $\widetilde{\mathcal{M}}_{k-2j}^{mer}(\rho)$ (of some depth $d\geq0$) produces an element $\widetilde{\mathcal{M}}_{k}^{sing}(\rho)$ that is an eigenfunction of $\Delta_{k}^{V}$ with eigenvalue $j(k-1-j)$. Let us now examine the relations between these quasi-modular forms and the ones considered in Corollary \ref{eigenVtot}. We shall also see in Corollary \ref{geneigen} below that the natural quasi-modular eigenforms are the \emph{only} quasi-modular eigenforms.

\smallskip

We shall encounter in some evaluations below several sums of the following type, generalizing the vanishing of expressions of the form $(1-1)^{n}$ for $n>0$.
\begin{lem}
Let $r$ and $h$ be non-negative integers, and take another integer $r\leq\nu \leq h+r$. Then the sum $\sum_{s=\nu}^{h+r}\binom{h}{s-r}(-1)^{s}$ equals $(-1)^{r}$ if $h=0$ (and $\nu=r$) and $(-1)^{\nu}\binom{h-1}{\nu-r-1}$ if $h>0$, while $\sum_{s=r}^{\nu}\binom{h}{s-r}(-1)^{s}$ produces again $(-1)^{r}$ for $h=0$ (with $\nu=r$) and $(-1)^{\nu}\binom{h-1}{\nu-r}$ for $h>0$. \label{cancbinom}
\end{lem}
This lemma is probably well-known, but we give the proof since it is very simple. Note that the classical vanishing is indeed obtained for $\nu=r$ in the first sum from Lemma \ref{cancbinom}, as well as for $\nu=h+r$ in the second sum there.
\begin{proof}
The two statements for $h=0$ (which coincide) are immediate (a sum of a single term). Assuming now that $h>0$, we decompose the binomial coefficient $\binom{h}{s-r}$ as $\binom{h-1}{s-r}+\binom{h-1}{s-r-1}$ as in the proof of Lemma \ref{prodtrans}, and in the terms arising from the second summand we substitute $s+1$ for $s$. As this inverts the sign, all the terms with $\nu \leq s \leq h+r-1$ in the first sum (the one from $\nu$) cancel, and the same assertion holds for the terms with $r \leq s\leq\nu-1$ in the second sum (the one up to $\nu$). Moreover, the term associated with $s=h+r$ in the first sum and the one with $s=r-1$ in the second vanish as well, since then $s-r=h>h-1$ and $s-r=-1<0$ respectively. The only surviving terms now are the asserted ones.
\end{proof}

We shall also need the value of the following sum.
\begin{lem}
Let $0 \leq r<e<d$ and $1 \leq h \leq e-r$ be integers. Then $\sum_{m=e-r-h+1}^{e-r}\frac{(-1)^{e-m}}{m(d-e+m)}\binom{h-1}{e-r-m}$ equals $\frac{(-1)^{r-h-1}}{d-e}(h-1)!\big(\frac{(e-r-h)!}{(e-r)!}-\frac{(d-r-h)!}{(d-r)!}\big)$. \label{sumeval}
\end{lem}
It is possible that the expressions from Lemma \ref{sumeval} can be related, with an appropriate additional function, to what is called a \emph{Wilf--Zeilberger pair} in \cite{[WZ]} and others, and therefore this lemma would be a consequence of that reference. However, in the spirit of the appendix to \cite{[P2]} (which deals with a similar sum), we give a direct proof.

\begin{proof}
First we prove, by induction on $p$, that if $p\geq0$ and $t\geq1$ are integers, then $\sum_{l=0}^{p}\binom{p}{l}\frac{(-1)^{l}}{l+t}$ equals $\frac{p!(t-1)!}{(t+p)!}$. Indeed, if $p=0$ then both expressions are $\frac{1}{t}$. Now we assume that $p>0$ and break up $\binom{p}{l}$ as $\binom{p-1}{l}+\binom{p-1}{l-1}$, where the sum of the first terms in this decomposition is the one associated with $p-1$ and $t$. On the other hand, replacing $l$ by $l+1$ in the sum arising from the second term is easily seen to give us minus the expression coming from $p-1$ and $t+1$. Applying the induction hypothesis for both sums, we obtain the difference between $\frac{(p-1)!(t-1)!}{(t+p-1)!}$ and $\frac{(p-1)!t!}{(t+p)!}$. Taking out the common multiplier $\frac{(p-1)!(t-1)!}{(t+p)!}$, the remaining terms form the difference $(t+p)-t=p$, combining with the first factorial to give the asserted expression.

Let us now consider the sum in question. The substitution $m=e-r-h+1+l$ produces the sum $(-1)^{r-h-1}\sum_{l=0}^{h-1}\binom{h-1}{h-1-l}\frac{(-1)^{l}}{(e-r-h+1+l)(d-r-h+1+l)}$, and we can express the latter multiplier as $\frac{(-1)^{l}}{d-e}\big(\frac{1}{e-r-h+1+l}-\frac{1}{d-r-h+1+l}\big)$. The symmetry of binomial coefficients therefore presents our expression as $\frac{(-1)^{r-h-1}}{d-e}$ times the difference between two expressions, both of which take the form from the previous paragraph. In both expressions $p$ is $h-1$, and the values of $t$ are $e-r-h+1$ and $d-r-h+1$ respectively. A simple substitution therefore produces the asserted result.
\end{proof}

We shall also require the (holomorphic) derivatives of the modular form $\Psi$ from Lemma \ref{sesonholder}.
\begin{lem}
Consider the sesqui-harmonic modular form $\Psi$ defined in Lemma \ref{sesonholder}, and take $l\in\mathbb{N}$. Then $\partial_{\tau}^{l}\Psi$ is the sum of a meromorphic quasi-modular form and the expression $\sum_{h=1}^{\upsilon+l}\frac{(\upsilon+l)!}{(\upsilon+l-h)! \cdot h\upsilon}\cdot\frac{\partial_{\tau}^{\upsilon+l-h}\varphi}{(-2iy)^{h}}-\sum_{h=1}^{l}\frac{l!}{(l-h)! \cdot h\upsilon}\cdot\frac{\partial_{\tau}^{\upsilon+l-h}\varphi}{(-2iy)^{h}}$. \label{Psideriv}
\end{lem}

\begin{proof}
The image of the meromorphic part of $\Psi$ under the holomorphic operator $\partial_{\tau}^{l}$ is clearly meromorphic. On the other hand, the general Leibniz rule implies that $\partial_{\tau}^{l}$ sends each of the terms $\frac{\partial_{\tau}^{\upsilon-p}\varphi}{(-2iy)^{p}}$ to $\sum_{t=0}^{l}\binom{l}{t}\partial_{\tau}^{\upsilon+l-p-t}\varphi\cdot\partial_{\tau}^{t}\frac{1}{(-2iy)^{p}}$. A simple induction on $t$ shows, using the fact that this denominator is a power of $\overline{\tau}-\tau$, that the latter multiplier is $\frac{(p+t-1)!}{(p-1)!(-2iy)^{p+t}}$. Putting in the multiplier $\frac{(\upsilon-1)!}{(\upsilon-p)! \cdot p}$, summing over $p$ as well, and using the summation index $h=p+t$, we get $\sum_{h=1}^{\upsilon+l}\sum_{p=\max\{1,h-l\}}^{\min\{h,\upsilon\}}\frac{(\upsilon-1)!(h-1)!l!}{p!(\upsilon-p)!(h-p)!(l-h+p)!}\frac{\partial_{\tau}^{\upsilon+l-h}\varphi}{(-2iy)^{h}}$. Note that $p$ runs over the entire range of numbers for which the factorials in the denominator operate on non-negative integers, except for the value $p=0$ when $h \leq l$. We add and subtract the terms with $h \leq l$ and $p=0$, which are easily seen to produce the second asserted sum (after the minus). In the remaining expression we write the quotient of factorials as the product of $\frac{(\upsilon-1)!l!}{(\upsilon+l-h)! \cdot h}$ and $\binom{h}{p}\binom{\upsilon+l-h}{\upsilon-p}$, and the internal sum over $p$ therefore reduces to just $\binom{\upsilon+l}{\upsilon}$. Multiplying this binomial coefficient by the external quotient produces the asserted coefficient from the first required sum.
\end{proof}

\begin{thm}
The lifts $\mathcal{L}_{k,d}^{V,0}\big(\mathcal{M}_{k-2d}^{mer}(\rho)\big)$ from the case $j=0$ of Theorem \ref{eigenVgen}, in which $k$ is not an integer between $d+1$ and $2d$, are meromorphic. In case $k=d+1+e$ with $0 \leq e<d$ (and again $j=0$), the lifts $\mathcal{L}_{k,d}^{V,0}\big(\mathcal{M}_{2d+2-k,ses}^{nm}(\rho)\big)$, described in Theorem \ref{lasteigenV} for $e>0$ and as the appropriate restriction of Proposition \ref{s0ses} in case $e=0$, also have this property. \label{harmmer}
\end{thm}

\begin{proof}
We recall from Theorem \ref{QMFVVMF} that the meromorphicity of an element of $\mathcal{M}_{k-\infty}^{sing}(V_{\infty}\otimes\rho)$, of some depth $d$, is equivalent to the meromorphicity of the components $f_{r}$, $0 \leq r \leq d$. We shall therefore evaluate these quasi-modular forms using the formula from that theorem, where $F_{s}$, $0 \leq s \leq d$ are given as the lifts with $j=0$ in Theorem \ref{eigenVgen}, Proposition \ref{s0ses} (with $e=0$), and Theorem \ref{lasteigenV}. For the lifts $\mathcal{L}_{k,d}^{V,0}$ from Theorem \ref{eigenVgen} and Remark \ref{deptheigen} the evaluations are precisely those appearing in the proof of Proposition \ref{eigenAtypes}, showing that $f_{r}$ is the meromorphic quasi-modular form $\binom{d}{r}(-1)^{r}\partial_{\tau}^{d-r}\varphi\Big/\prod_{q=0}^{d-r-1}(k-2d+q)$. This argument covers the lifts $\mathcal{L}_{k,d}^{V,0}$ for any $k$ that is not an integer between $d+1$ and $2d$, as well as the functions $f_{r}$ with $e<r \leq d$ when $k=d+1+e$ with $0 \leq e<d$ (as the expression for $F_{s}$ with $e<s \leq d$ in Proposition \ref{s0ses} and Theorem \ref{lasteigenV} is defined by the formula from Theorem \ref{eigenVgen} as well).

We now turn our attention to the remaining functions, in which some of the functions $F_{s}$ involve the sesqui-harmonic modular form $\Psi$ from Lemma \ref{sesonholder}. We therefore take some $r \leq e<d$ and use the expressions from Theorem \ref{lasteigenV} with $j=0$ for the functions $F_{s}$, $0 \leq s \leq d$ again (we shall soon see that the case considered in Proposition \ref{s0ses} becomes a special case of this construction). The functions with $s>e$ are defined in the same manner as above, so that the considerations from the proof of Proposition \ref{eigenAtypes} show that their contribution to $f_{r}$ with $0 \leq r \leq e$ amounts to
\[\binom{d}{r}\sum_{h=e+1-r}^{d-r}\binom{d-r}{h}\frac{(r+h-e-1)!}{(d-e-1)!}\Bigg[\sum_{s=e+1}^{h+r}(-1)^{d-h+s-r}\binom{h}{s-r}\bigg]\cdot\frac{\partial_{\tau}^{d-r-h}\varphi}{(2iy)^{h}}\] (since $h=s+t-r$ is at least $e+1-r$ for $s>e$ and $t\geq0$, and the quotient over $q$ in the denominator becomes the quotient of factorials with the extra sign). Lemma \ref{cancbinom} now reduces the sum over $s$ to just $(-1)^{d-h+e-r+1}\binom{h-1}{e-r}$, and after expanding this binomial coefficient as well as $\binom{d-r}{h}$ and canceling, we find that every summand $\frac{\partial_{\tau}^{d-r-h}\varphi}{(2iy)^{h}}$ with $e-r<h \leq d-r$ appears with the coefficient $\binom{d}{r}\frac{(-1)^{d-h+e-r+1}}{(e-r)!(d-e-1)!}\cdot\frac{(d-r)!}{(d-r-h)! \cdot h}$.

Now, the proof of Theorem \ref{lasteigenV} shows that for $0 \leq s \leq e$ the number $b_{s}$ (with $j=0$) is $(e-d)\binom{d}{s}\Big/\prod_{q=0,q \neq d-e-1}^{d-s-1}(k-2d+q)$. With the weight $k=d+1+e$ this expression becomes $\frac{(-1)^{d-e}(d-e)}{(e-s)!(d-e-1)!}\binom{d}{s}$ (by the same considerations). Note that the case $d=e=s=0$ produces the coefficient $\frac{(-1)^{d}d}{(d-1)!}$, which is the one appearing in front of $\Psi$ in $F_{0}$ in the case $e=0$ of Proposition \ref{s0ses} with $e=0$. It follows that if we assume that the function $\Psi$ in that proposition satisfies the stronger condition from Lemma \ref{sesonholder} (i.e., with the statement about $4y^{2}\partial_{\overline{\tau}}\Psi$ and not just $\Delta_{d+1}\Psi$), then the resulting function is indeed a ``special case'' of the ones considered in Theorem \ref{lasteigenV}. Finally, the coefficient $c_{s}$ was seen to be the same coefficient $b_{s}$ multiplied by the sum $\sum_{m=1}^{e-s}\frac{1}{m(d-e+m)}$, and the function $\delta_{e+1-d}^{d-s}\varphi$ that it multiplies equals $\delta_{d+1-e}^{e-s}(\partial_{\tau}^{d-e}\varphi)$ by applying Bol's identity one more time.

We analyze the resulting expression using the same tools from the proof of Proposition \ref{eigenAtypes}. Ignoring the common constant coefficient $\frac{(-1)^{d-e}(d-e)}{(d-e-1)!}$ for the moment, we expand the two instances of the operator $\delta_{d+1-e}^{e-s}$ as in Equation (56) of \cite{[Za]} (with the summation index $t$, in which the inner product is $\frac{(d-s)!}{(d-s-t)!}$), substitute into the formula for $f_{r}$ in Theorem \ref{QMFVVMF} (noting the different signs in the powers of $2iy$ in the denominators), write the combinatorial expression $\binom{s}{r}\binom{d}{s}\binom{e-s}{t}\frac{(d-s)!}{(e-s)!(d-s-t)!}$ as $\binom{d}{r}\binom{d-r}{s+t-r}\binom{s+t-r}{s-r}\frac{1}{(e-s-t)!}$, and change the summation index to $h=s+t-r$ (now bounded by $e-r$) to obtain
\[\binom{d}{r}\sum_{h=0}^{e-r}\frac{\binom{d-r}{h}}{(e-r-h)!}\sum_{s=r}^{h+r}(-1)^{s}\binom{h}{s-r}\Bigg[\frac{\partial_{\tau}^{e-r-h}\Psi}{(2iy)^{h}}+
\sum_{m=1}^{e-s}\frac{\partial_{\tau}^{e-r-h}\varphi}{m(d-e+m)(2iy)^{h}}\bigg].\] Noting that the inner term involving $\Psi$ is independent of $s$, the same argument from above implies that for these terms the sum over $s$ reduces to $(-1)^{r}\delta_{h,0}$. Hence the total expression involving $\Psi$ in the above formula is just $\binom{d}{r}\frac{(-1)^{r}}{(e-r)!}\partial_{\tau}^{e-r}\Psi$. On the other hand, the terms involving $\varphi$ here appear only if $r<e$ (for an integer $1 \leq m \leq e-s \leq e-r$ to exist). When we evaluate, for some $0 \leq h \leq e-r$, the coefficient of $\binom{d}{r}\binom{d-r}{h}\frac{\partial_{\tau}^{e-r-h}\varphi}{(e-r-h)!(2iy)^{h}}$, we may interchange the order of summation between $m$ and $s$ and get the expression
$\sum_{m=1}^{e-r}\sum_{s=r}^{\min\{h+r,e-m\}}\binom{h}{s-r}\frac{(-1)^{s}}{m(d-e+m)}$. Now, for $m \leq e-r-h$ the sum over $s$ is until $h+r$, reducing to just $\delta_{0,h}$, and for $h=0$ the inequality on $m$ always holds. If $m>e-r-h$ we apply Lemma \ref{cancbinom}, and evaluate the sum over $s$ as $\frac{(-1)^{e-m}}{m(d-e+m)}\binom{h-1}{e-r-m}$. The terms with $h=0$ are all meromorphic, while for $h\geq1$ the sum of the previous expression, which must be carried out over $e-r-h+1 \leq m \leq e-r$, is evaluated in Lemma \ref{sumeval}. Plugging in the expressions from that lemma, and recalling the external coefficient $\frac{(-1)^{d-e}(d-e)}{(d-e-1)!}$, we find that for each $1 \leq h \leq e-r$ the coefficient of $\frac{\partial_{\tau}^{e-r-h}\varphi}{(2iy)^{h}}$ is, in total, $\binom{d}{r}\frac{(-1)^{d+h-e+r}}{(d-e-1)! \cdot h}\big(\frac{1}{(e-r-h)!}-\frac{(d-r)!}{(d-r-h)!(e-r)!}\big)$.

Gathering all these terms, the expression $\frac{\partial_{\tau}^{e-r-h}\varphi}{(2iy)^{h}}$ appears, for every index $1 \leq h \leq d-r$, with the coefficient $\binom{d}{r}\frac{(-1)^{d-h+e-r+1}}{(e-r)!(d-e-1)!}\cdot\frac{(d-r)!}{(d-r-h)! \cdot h}$, while for $1 \leq h \leq e-r$ we also have the extra coefficient $\binom{d}{r}\frac{(-1)^{d+h-e+r}}{(d-e-1)!(e-r-h)! \cdot h}$. To all this we must add the expression $\binom{d}{r}\frac{(-1)^{d-e+r}(d-e)}{(d-e-1)!(e-r)!}\partial_{\tau}^{e-r}\Psi$, for the evaluation of which we apply Lemma \ref{Psideriv} with the parameters $\upsilon=d-e$ and $l=e-r$. This yields a meromorphic quasi-modular form plus an explicit expression, in which the number $\upsilon=d-e$ alone cancels. For each $h$ the fact that the denominator in that lemma involves a power of $-2iy$ gives the right sign, and since $\upsilon+l=d-r$ and for $h \leq l=e-r$ the terms with $(e-r)!$ cancel, we find that indeed the non-meromorphic parts arising from the derivative of $\Psi$ are exactly the opposite of those involving $\varphi$ in $f_{r}$. Therefore only meromorphic expressions remain in each $f_{r}$ also in this case.
\end{proof}

Combining Theorem \ref{harmmer} with Corollary \ref{eigenVtot} allows us to obtain the general form of any eigenfunction of $\Delta_{k-\infty}^{V}$. Similarly to the definition of the powers $\delta_{l}^{m}$, we define the operator $\tilde{\delta}_{l}^{m}$ on $\mathcal{M}_{l-\infty}^{sing}(V_{\infty}\otimes\rho)$ for some weight $l$ and some power $m$ to be the composition $\tilde{\delta}_{l+2m-2}\circ\ldots\circ\tilde{\delta}_{l}$, sending elements of the latter space into $\mathcal{M}_{l+2m-\infty}^{sing}(V_{\infty}\otimes\rho)$.
\begin{cor}
Let $k$ be a weight and $\lambda$ be an eigenvalue. If there is no natural solution $j$ to the equation $\lambda=j(k-1-j)$, then $\mathcal{M}_{k-\infty,\lambda}^{sing,V}(V_{\infty}\otimes\rho)$ consists only of elements of depth 0, and they are obtained from $\mathcal{M}_{k,\lambda}^{sing}(\rho)$ via the lift $\mathcal{L}_{k,0}^{V,\lambda}$ from Proposition \ref{eigenVd0}. Otherwise we take the minimal such solution $j$ (it is either the single one or the smaller between two integers), and then $\mathcal{M}_{k-\infty,\lambda}^{sing,V}(V_{\infty}\otimes\rho)$ is the sum of the space $\mathcal{L}_{k,0}^{V,\lambda}\big(\mathcal{M}_{k,\lambda}^{sing}(\rho)\big)$ from above and images of elements of $\mathcal{M}_{k-2j-\infty}^{mer}(V_{\infty}\otimes\rho)$ under the operator $\tilde{\delta}_{k-2j}^{j}$. \label{eigenVform}
\end{cor}

\begin{proof}
The first assertion is just part $(i)$ of Corollary \ref{eigenVtot}. For the second one, note that Lemma \ref{LDeldelmer} continues to hold if we replace the spaces of the form $\mathcal{M}_{l,\mu}^{sing}(\rho)$ by the corresponding spaces $\mathcal{M}_{l-\infty,\mu}^{sing,V}(V_{\infty}\otimes\rho)$ and the operator $\delta_{l}$ by the operator $\tilde{\delta}_{l}$ from Corollary \ref{comim}. This can either be seen using part $(vi)$ of Proposition \ref{comrels} (showing that the commutation relation between the $\tilde{\delta}_{l}$'s and $y^{2}\partial_{\overline{\tau}}$ are the same as for the $\delta_{l}$'s), or via Theorem \ref{difopShim} and the relations with $\widetilde{\mathcal{M}}_{l}^{sing}(\rho)$ (after the limit map from Corollary \ref{comim}). It follows that $\tilde{\delta}_{k-2j}^{j}$ maps $\mathcal{M}_{k-2j-\infty,0}^{sing,V}(V_{\infty}\otimes\rho)$ into $\mathcal{M}_{k-\infty,j(k-1-j)}^{sing,V}(V_{\infty}\otimes\rho)$, and applying the operator $y^{2}\partial_{\overline{\tau}}$ repeatedly $j$ times sends the latter space to the former. Moreover, observing that the eigenvalues involved are of the form $r(k-2j+r-1)$ with $0 \leq r<j$ and adding the appropriate weight $k-2j+2r$ simply produces the value with $r+1$, we deduce from our version of Lemma \ref{LDeldelmer} that the composition of these two maps, in both directions, are the identity map multiplied by the scalar $\prod_{r=1}^{j}\frac{r(2j-r+1-k)}{4}$ (recall the relation between these compositions, the Laplacian, and the convention regarding the eigenvalues). Since the vanishing of this product implies that $k=2j+1-r$ for some $1 \leq r \leq j$, which would compare the value $j(k-1-j)$ of $\lambda$ to $j(j-r)=(j-r)(k-1-j+r)$, this contradicts our assumption that $j$ is the minimal solution to $\lambda=j(k-1-j)$. Hence the two aforementioned maps are isomorphisms, the space $\mathcal{M}_{k-\infty,j(k-1-j)}^{sing,V}(V_{\infty}\otimes\rho)$ is the image of $\mathcal{M}_{k-\infty,j(k-1-j)}^{sing,V}(V_{\infty}\otimes\rho)$ under $\tilde{\delta}_{k-2j}^{j}$, and the latter space is described in Corollary \ref{eigenVtot} and Theorem \ref{harmmer} as the sum of $\mathcal{L}_{k,0}^{V,0}\big(\mathcal{M}_{k-2j,0}^{sing}(\rho)\big)$ and $\mathcal{M}_{k-2j-\infty}^{mer}(V_{\infty}\otimes\rho)$.
\end{proof}

\begin{rmk}
The sum in the second assertion in Corollary \ref{eigenVform} is not direct. Indeed, if $j=0$ then $\mathcal{L}_{k,0}^{V,0}\big(\mathcal{M}_{k}^{mer}(\rho)\big)$ is contained in both spaces (and in fact forms the entire intersection), and in general the intersection consists of the $\mathcal{L}_{k,0}^{V,j(k-1-j)}$-images of those elements of $\mathcal{M}_{k,j(k-1-j)}^{sing}(\rho)\cap\mathcal{M}_{k}^{nm,\leq j}(\rho)$ that are $\delta_{k-2j}^{j}$-images of modular forms from $\mathcal{M}_{k-2j}^{mer}(\rho)$. To see this, apply the isomorphism $\tilde{\delta}_{k-2j}^{j}$ from the proof of Corollary \ref{eigenVform}, and use the commutation relation, following from Theorem \ref{difopShim} via the limit from Corollary \ref{comim}, between that operator, the operator $\delta_{k-2j}^{j}$, and the embeddings $\mathcal{L}_{k-2j,0}^{V,0}$ and $\mathcal{L}_{k,0}^{V,j(k-1-j)}$. As for the solution to $\lambda=j(k-1-j)$ that is not minimal, recall that this happens for $j=j_{+}$ when $k=j_{+}+j_{-}+1$ and $\lambda=j_{+}j_{-}$ with integers $j_{+}>j_{-}\geq0$. The proof of Corollary \ref{eigenVform} then shows that the operator $\tilde{\delta}_{k-2j_{+}}^{j_{+}}=\tilde{\delta}_{j_{-}+1-j_{+}}^{j_{+}}$ decomposes as the composition of $\tilde{\delta}_{k-2j_{-}}^{j_{-}}$ and $\tilde{\delta}_{j_{-}+1-j_{+}}^{j_{+}-j_{-}}$, and the latter operator sends $\mathcal{M}_{j_{-}+1-j_{+}-\infty,0}^{mer,V}(V_{\infty}\otimes\rho)$ into $\mathcal{M}_{j_{+}+1-j_{-}-\infty,0}^{mer,V}(V_{\infty}\otimes\rho)$. Therefore no new functions are obtained from that index. We also remark that heuristically the statement of Corollary \ref{eigenVform} (as well as of many of the previous results) can be extended to determining the eigenfunctions of the appropriate Laplacian on the space $\mathcal{M}_{k-m}^{sing}(V_{m}\otimes\rho)$ of finite $m$. This is so since the lifts $\mathcal{L}_{k,d}^{V,j}$ with $0<d \leq m$ have depth not exceeding $m$, and they can therefore be considered as elements of $\mathcal{M}_{k-m}^{sing}(V_{m}\otimes\rho)$ embedded in $\mathcal{M}_{k-\infty}^{sing}(V_{\infty}\otimes\rho)$ via the usual limit map. Moreover, since in such lifts the function $F_{d}$ (or $f_{d}$) with the maximal index has to be meromorphic, the appropriate Laplacian is well-defined on such elements of $\mathcal{M}_{k-\infty}^{sing}(V_{\infty}\otimes\rho)$ since after letting the operator $4y^{2}\partial_{\overline{\tau}}$ act we land in the space of $i_{m-1}$-images. However, as we have already observed in Remark \ref{LaponVm} that this Laplacian is not defined on the entire space $\mathcal{M}_{k-m}^{sing}(V_{m}\otimes\rho)$, we do not investigate this point of view further. \label{LapVm}
\end{rmk}

\smallskip

The translation of the operator $\Delta_{k-\infty}^{V}$ to an operator on $\widetilde{\mathcal{M}}_{k}^{sing}(\rho)$ via Theorem \ref{difopShim} and the limit maps from Corollary \ref{comim} produces just the usual Laplacian $\Delta_{k}$ from the classical theory of modular forms. Its eigenfunctions in $\widetilde{\mathcal{M}}_{k}^{sing}(\rho)$ can be characterized by translating the results of Corollary \ref{eigenVform} to this setting using Theorem \ref{difopShim} and Corollary \ref{comim}.
\begin{cor}
If $k$ and $\lambda$ are such that the equation $\lambda=j(k-1-j)$ has no natural solution $j$, then every quasi-modular form in $\widetilde{\mathcal{M}}_{k}^{sing}(\rho)$ that has the eigenvalue $\lambda$ is of depth 0 (i.e., is a classical modular eigenform). In case $\lambda$ equals $j(k-1-j)$ for some natural $j$, which we take to be the smaller one in case there are two solutions, the appropriate quasi-modular eigenforms are sums of (classical) modular eigenforms and images of meromorphic quasi-modular forms from $\widetilde{\mathcal{M}}_{k-2j}^{mer}(\rho)$ under $\delta_{k-2j}^{j}$. \label{geneigen}
\end{cor}

\begin{proof}
When we write an element $F$ of $\widetilde{\mathcal{M}}_{k-\infty}^{sing}(V_{\infty}\otimes\rho)$ using the limit image of the holomorphic basis from Theorem \ref{QMFVVMF}, we recall from that theorem that the quasi-modular form $f\in\widetilde{\mathcal{M}}_{k}^{sing}(\rho)$ that is associated to $F$ there is just the coefficient $f_{0}$ of $\binom{\tau}{1}^{\infty}$ in that expansion. As already mentioned in that theorem, the differential properties of $F$ are resembled in those of the coefficients $f_{r}$ from that basis. We also know from Theorem \ref{difopShim} and Corollary \ref{comim} that the operation of $\tilde{\delta}_{k-2j}^{j}$ from Corollary \ref{eigenVform} becomes, when one considers $f_{0}$ alone, just the action of $\delta_{k-2j}^{j}$. Observing that for elements of depth 0 we have $F(\tau)=F_{0}(\tau)\binom{\tau}{1}^{\infty}=f_{0}(\tau)\binom{\tau}{1}^{\infty}$, the required assertions are now easily seen to be the translation of Corollary \ref{eigenVform} to the quasi-modular setting.
\end{proof}

Corollary \ref{geneigen} translates to the the promised statement from the Introduction and from the beginning of this section: A quasi-modular eigenform is a sum of a modular eigenform and (perhaps, for specific eigenvalues) the image of a meromorphic quasi-modular form under the appropriate power of the (classical) weight raising operator.

\begin{rmk}
The proof of Theorem \ref{harmmer} shows (together with the usual argument from Theorem \ref{QMFVVMF}) that when $k$ is not one of the special weights the explicit value of the quasi-modular form $f_{0}$ associated with $\mathcal{L}_{k,d}^{V,0}(\varphi)$ for $\varphi\in\mathcal{M}_{k-2d}^{mer}(\rho)$ is $\partial_{\tau}^{d}\varphi\Big/\prod_{q=0}^{d-1}(k-2d+q)$. Combining this with Corollary \ref{geneigen} and the appropriate normalizations, one sees that $\mathcal{L}_{k,d}^{V,j}(\varphi)$ with any $0 \leq j<d$ corresponds, for such $k$, to $(-1)^{j}\frac{(d-1-j)!}{(d-1)!}\delta_{k-2j}^{j}(\partial_{\tau}^{d-j}\varphi)\Big/\prod_{q=0}^{d-j-1}(k-2d+q)$ (this can also be proved directly, by following the proof of Theorem \ref{harmmer}). In case $k$ does equal one of the special weights, the proof of Lemma \ref{Psideriv} implies that the meromorphic quasi-modular form $f_{0}$ from Theorem \ref{harmmer} involves derivatives of the meromorphic expression $-\sum_{p=1}^{\upsilon}\frac{(\upsilon-1)!}{(\upsilon-p)! \cdot p}\partial_{\tau}^{\upsilon-p}\varphi \cdot G(\tau)^{p}$ from Lemma \ref{sesonholder}. The functions from Corollary \ref{geneigen} are therefore also more complicated in this case, and we shall not write them explicitly. \label{expval}
\end{rmk}

\smallskip

We conclude by combining what we know about eigenfunctions of the two Laplacians to determine those elements of $\mathcal{M}_{k-\infty}^{sing}(V_{\infty}\otimes\rho)$ that are eigenfunctions with respect to the two Laplacian operators simultaneously.
\begin{cor}
The lifts $\mathcal{L}_{k,0}^{V,\lambda}(\varphi)=\mathcal{L}_{k,0}^{A}(\varphi)$ for $\varphi\in\mathcal{M}_{k,\lambda}^{sing}(\rho)$ have eigenvalue $\lambda$ with respect to $\Delta_{k-\infty}^{V}$ and eigenvalue 0 with respect to $\Delta_{k-\infty}^{A}$. Moreover, any lift of the form $\mathcal{L}_{k,d}^{V,0}(\varphi)=\mathcal{L}_{k,d}^{A}(\varphi)$ with $\varphi\in\mathcal{M}_{k-2d}^{mer}(\rho)$ and $k$ not an integer between $d+1$ and $2d$ is $\Delta_{k-\infty}^{V}$-harmonic and has the eigenvalue $d(k-1-d)$ with with respect to $\Delta_{k-\infty}^{A}$. These are the only elements of $\mathcal{M}_{k-\infty}^{sing}(V_{\infty}\otimes\rho)$ that are simultaneous eigenfunctions for both the Laplacian operators. \label{comeigen}
\end{cor}

\begin{proof}
The first assertion is obvious from Proposition \ref{eigenVd0} and Theorem \ref{eigenAgen} (or the fact that $\Delta_{k-\infty}^{A}$ annihilates functions of depth 0). Next, note that the coefficients in $\mathcal{L}_{k,d}^{A}(\varphi)$ appearing in Theorem \ref{eigenAgen} coincide with $a_{d,s}^{(0)}$ defined in Equation \eqref{adsj} inside the proof of Theorem \ref{eigenVgen}. This implies the second assertion. Now, as we have seen that eigenfunctions of $\Delta_{k-\infty}^{V}$ of positive depth $d$ have meromorphic $F_{d}$, the fact that for an integral weight $k$ between $d+1$ and $2d$ Proposition \ref{eigenAsp} forces the function $F_{d}$ in eigenfunctions of $\Delta_{k-\infty}^{A}$ to be non-meromorphic shows that only the lifts from Theorem \ref{eigenVgen} can produce simultaneous eigenfunctions. Moreover, by what we have seen this happens if and only if the coefficients $a_{d,s}^{(j)}$ appearing in $\mathcal{L}_{k,d}^{V,j}(\varphi)$ coincide with $a_{d,s}^{(0)}$ from $\mathcal{L}_{k,d}^{A}(\varphi)$, with $\varphi\in\mathcal{M}_{k-2d}^{mer}(\rho)$. Now, for $s=d-1$ this is equivalent to the vanishing of $\frac{k-j-1}{d-1}$ from $l=1$ (with $p=1$), and Remark \ref{deptheigen} (with $j_{+}=j$ and $j_{-}=0$) shows that in this case $\mathcal{L}_{k,d}^{V,j}(\varphi)$ is $\mathcal{L}_{k,d}^{V,0}(\varphi)$ yet again. Therefore the aforementioned functions are indeed the only simultaneous eigenfunctions for the two Laplacians, proving the third assertion as well.
\end{proof}
The quasi-modular eigenfunctions corresponding via Theorem \ref{QMFVVMF} to the elements from Corollary \ref{comeigen} are just the usual modular eigenforms and the meromorphic quasi-modular forms from the case $j=0$ in Corollary \ref{geneigen}.

\noindent\textsc{Einstein Institute of Mathematics, the Hebrew University of Jerusalem, Edmund Safra Campus, Jerusalem 91904, Israel}

\noindent E-mail address: zemels@math.huji.ac.il


\begin{thebibliography}{}{}

\bibitem[A]{[A]} Azaiez, N. O., \textsc{The Ring of Quasimodular Forms for a Cocompact Group}, J. Number Theory, vol 128 Issue 7, 1966–-1988 (2008).
    Number Theory, vol 1, 1–-16 (2015).
\bibitem[BDR]{[BDR]} Bringmann, K., Diamantis, N., Raum, M., \textsc{Mock Period Functions, Sesquiharmonic Maass forms, and Non-Critical Values of $L$-Functions}, Advances in Math., vol 233, 115-–134 (2013).
\bibitem[BFo]{[BFo]} Bringmann, K., Folsom, A., \textsc{Almost Harmonic Maass Forms and Kac-Wakimoto Characters}, J. reine angew. Math., vol 694, 179–-202 (2014).
\bibitem[BFu]{[BFu]} Bruinier, J. H., Funke, J., \textsc{On Two Geometric Theta Lifts}, Duke Math J., vol 125 no. 1, 45–-90 (2004).
\bibitem[BO]{[BO]} Bringmann, K., Ono, K., \textsc{Dyson's Ranks and Maass Forms}, Ann. Math., vol 171 issue 1, 419–-449 (2010).
\bibitem[CDO]{[CDO]} Chinta, G., Diamantis, N., O'Sullivan, C., \textsc{Second Order Modular Forms}, Acta Arith., vol 103 no. 3, 209-–223 (2002).
\bibitem[DS]{[DS]} Diamantis, N., Sim, D., \textsc{The Classification of Higher-Order Cusp Forms}, J. reine angew. Math., vol 622, 121-–153 (2008).
\bibitem[E]{[E]} Eichler, M., \textsc{Eine Verallgemeinerung der Abelschen Integrale}, Math. Zeitschr., vol 67 no 1, 267-–298 (1957).
\bibitem[JKK]{[JKK]} Jeon, D., Kang, S.-Y., Kim, C. H., \textsc{Cycle Integrals of a Sesqui-Harmonic Maass Form of Weight Zero}, J. Number Theory, vol 141, 92-–108 (2014).
\bibitem[Kn1]{[Kn1]} Knopp, M. I., \textsc{Some New Results on the Eichler Cohomology of Automorphic Forms}, Bull. Amer. Math. Soc., vol 80 no. 4, 607–-632 (1974).
\bibitem[Kn2]{[Kn2]} Knopp, M. I., \textsc{Rational Period Functions of the Modular Group}, Duke Math J., vol 45 no. 1, 47–-62 (1978).
\bibitem[KZ]{[KZ]} Kaneko, M., Zagier, D., \textsc{A Generalized Jacobi Theta Function and Quasimodular Forms}, in: The Moduli Space of Curves (Texel Island, 1994), Progr. Math. 129, Birkh\"{a}user Boston, Boston, MA, 165--172 (1995).
\bibitem[MR1]{[MR1]} Martin, F., Royer, E., \textsc{Formes Modulaires et P\'{e}riodes}, Formes Modulaires et Transcendance, S\'{e}min. Congr., vol. 12, Soc. Math. France, Paris, 1–-117 (2005).
\bibitem[MR2]{[MR2]} Martin, F., Royer, E., \textsc{Rankin--Cohen Brackets on Quasimodular Forms}, J. Ramanujan Math. Soc., vol 24 issue 3, 213-233 (2009).
\bibitem[N]{[N]} Niebur, D., \textsc{A Class of Nonanalytic Automorphic Functions}, Nagoya Math. J., vol 52, 133–-145 (1973).
\bibitem[P1]{[P1]} Pribitkin, W., \textsc{Eisenstein series and Eichler integrals}, in: Analysis, Geometry, Number Theory: The Mathematics of Leon Ehrenpreis, Contemp. Math. 251, AMS, Providence, RI 463--467 (2000).
\bibitem[P2]{[P2]} Pribitkin, W., \textsc{Poincar\'{e} series and Eichler integrals}, Illinois J. Math., vol 53 no. 3, 883–-897 (2009).
\bibitem[Sh]{[Sh]} Shimura, G., \textsc{Sur les Int\'{e}grales Attach\'{e}es aux Formes Automorphes}, J. Math. Soc. Japan, vol 11, 291–-311 (1959).
\bibitem[V]{[V]} Verdier, J. L., \textsc{Sur les int\'{e}grales attach\'{e}es aux formes automorphes (d'apr\'{e}s Goro SHIMURA)}, S\'{e}m. BOURBAKI 216, vol 13 (1961).
\bibitem[WZ]{[WZ]} Wilf, H. S., Zeilberger, D. \textsc{Rational Functions Certify Combinatorial Identities}, J. Amer. Math. Soc., vol 3 no. 1, 147--158 (1990).
\bibitem[Za]{[Za]} Zagier, D. et al., \textsc{The 1-2-3 of Modular Forms}, Universitext (2008).
\bibitem[Ze1]{[Ze1]} Zemel, S. \textsc{On Quasi-Modular Forms, Almost Holomorphic Modular Forms, and the Vector-Valued Modular Forms of Shimura}, Ramanujan J., vol 37 issue 1, 165--180 (2015).
\bibitem[Ze2]{[Ze2]} Zemel, S. \textsc{Regularized pairings of meromorphic modular forms and theta lifts}, J. Number Theory, vol 162, 275-–311 (2016).
\bibitem[Zu]{[Zu]} Zucker, S., \textsc{Hodge Theory with Degenerating Coefficients: $L^{2}$-Cohomology in the Poincar\'{e} Metric}, Ann. of Math., vol. 109, 415--476, (1979).
\bibitem[Zw]{[Zw]} Zwegers, S., P. \textsc{Mock Theta Functions}, Ph.D. Thesis, University of Utrecht, (2002).
\end{thebibliography}
\end{document}